\documentclass[dvips,aos,preprint]{imsart}
\usepackage{amssymb}
\usepackage{epsfig}
\usepackage{amsthm}
\usepackage{booktabs}
\usepackage{amsmath}
\usepackage{appendix}
\usepackage{amsfonts}
\usepackage{color}
\usepackage{xr}
\RequirePackage[OT1]{fontenc} \RequirePackage{amsthm,amsmath}
\RequirePackage[numbers]{natbib}
\RequirePackage[colorlinks,citecolor=blue,urlcolor=blue]{hyperref}
\RequirePackage{hypernat} \makeatletter
\newcommand{\field}[1]{\mathbb{#1}}

\newcommand{\E}{\field{E}}
\def\qed{\hfill$\diamondsuit$}

\theoremstyle{example} \theoremstyle{remark} \theoremstyle{lemma}
\theoremstyle{definition} \theoremstyle{corol}
\theoremstyle{proposition} \theoremstyle{condition}
\theoremstyle{assumption}
\newtheorem{assumption}{\n{Assumption}}[section]
\newtheorem{theorem}{\n{Theorem}}[section]
\newtheorem{example}{\n{Example}}[section]
\newtheorem{remark}{\n{Remark}}[section]
\newtheorem{lemma}{\n{Lemma}}[section]

\newtheorem{corollary}{\n{Corollary}}[section]

\newtheorem{proposition}{\n{Proposition}}[section]

\newcommand*{\I}{\imath}
\font\n=cmcsc10
\def\cov{{\mbox{cov}}}

\def\cum{{\mbox{cum}}}
\makeatletter
\newcommand{\rmnum}[1]{\romannumeral #1}
\newcommand{\Rmnum}[1]{\expandafter\@slowromancap\romannumeral #1@}

\begin{document}
\begin{frontmatter}
\title{Bootstrapping High Dimensional Time Series} \runtitle{Bootstrapping High Dimensional Time Series}
\begin{aug}
\author{\fnms{Xianyang}
\snm{Zhang}\ead[label=e1]{zhangxiany@missouri.edu}\thanks{Assistant Professor, Department of Statistics, University of Missouri-Columbia, Columbia, MO 65211. E-mail: zhangxiany@missouri.edu. Tel:
+1 (573) 882-4455. Fax: +1 (573) 884-5524.}} and
\author{\fnms{Guang} \snm{Cheng}\ead[label=e2]{chengg@purdue.edu}\thanks{Corresponding Author. Associate Professor, Department
of Statistics, Purdue University, West Lafayette, IN 47906. E-mail:
chengg@purdue.edu. Tel: +1 (765) 496-9549. Fax: +1 (765) 494-0558.
Research Sponsored by NSF CAREER Award DMS-1151692, DMS-1418042, Simons
Foundation 305266.}}
\runauthor{X. Zhang and G. Cheng} \affiliation{University of
Missouri-Columbia and Purdue University}
\address{
X. Zhang\\
Department of Statistics\\
University of Missouri-Columbia\\
Columbia, MO 65211.\\
\printead{e1}\\
\phantom{E-mail:\ }}

\address{
G. Cheng\\
Department of Statistics\\
Purdue University\\
West Lafayette, IN 47906.\\
\printead{e2}\\ \phantom{E-mail:\ }}
\end{aug}

\begin{abstract}
This article studies bootstrap inference for high dimensional weakly
dependent time series in a general framework of approximately linear
statistics. The following high dimensional applications are covered:
(\rmnum{1}) uniform confidence band for mean vector; (\rmnum{2})
specification testing on the second order property of time series
such as white noise testing and bandedness testing of covariance
matrix; (\rmnum{3}) specification testing on the spectral property
of time series. In theory, we first derive a Gaussian approximation
result for the maximum of a sum of weakly dependent vectors, where
the dimension of the vectors is allowed to be exponentially larger
than the sample size. In particular, we illustrate an interesting
interplay between dependence and dimensionality, and also discuss
one type of ``dimension free" dependence structure. We further
propose a blockwise multiplier (wild) bootstrap that works for time
series with unknown autocovariance structure. These distributional
approximation errors, which are finite sample valid, decrease
polynomially in sample size. A non-overlapping
block bootstrap is also studied as a more flexible alternative. The above results
are established under the general physical/functional dependence
framework proposed in Wu (2005). Our work can be viewed as a
substantive extension of Chernozhukov et al. (2013) to time series
based on a variant of Stein's method developed therein.
\end{abstract}


\end{frontmatter}

\section{Introduction}
High-dimensional data are increasingly encountered in many
applications of statistics such as bioinformatics, information
technology, medical imaging, astronomy and financial studies. In
recent years, there is a growing body of literature concerning
inference on the first and second order properties of high
dimensional data; see
\cite{clx13,clx14,cj11,cq2010,czz2010,lc2012,qc2012} among others.
The validity of these procedures is generally established under
independence amongst the data vectors, which can be quite
restrictive for situations that involve temporally observed data.
Examples include spatial-temporal modeling \cite{wh2010} and
financial study of a large number of asset returns \cite{twyz2011}.
Although high dimensional statistics has witnessed unprecedented
development, statistical inference for high dimensional time series
remains largely untouched so far. In the conventional low
dimensional setting, inference for time series data typically
involves the direct estimation of the asymptotic covariance matrix,
which is known to be difficult in the presence of heteroscedasticity
and autocorrelation of unknown forms \cite{andrews1991}. In the high
dimensional setting, where the dimension is
comparable or even larger than sample size, the classical
inferential procedures designed for the low dimensional case are no
longer applicable, e.g., the asymptotic covariance matrix is
singular. Along a different line, alternative nonparametric
procedures including block bootstrap, subsampling and blockwise
empirical likelihood \cite{Carl1986,kun89,ls92,prw99,kit97} have
been proposed to avoid the direct estimation of covariance matrices.
However, the extension of these procedures (coupled with suitable
testing procedures) to the high dimensional setting remains unclear.
One relevant high dimensional work (\cite{cxw2013}) we are aware is on
the estimation rates of the covariance/precision matrices of time
series.

In this paper, we establish a general framework of conducting bootstrap inference for high dimensional stationary time series under weak dependence. We start from three motivating examples that are mainly concerned with first or second order property of time series: (\rmnum{1}) uniform confidence band for mean vector; (\rmnum{2}) testing for serial correlation; (\rmnum{3}) testing on the bandedness of covariance matrix. The proposed bootstrap procedures are rather simple to implement and supported by simulation results. We want to emphasize that neither Gaussian assumption nor strong restrictions on the covariance structure are imposed in these applications. An important by-product of Examples (\rmnum{2}) and (\rmnum{3}) is the covariance structure testing for high dimensional time series that even does not rely on the existence of the null limit distribution. This new result is in sharp contrast with the existing literature for i.i.d data such as \cite{clx14, cq2010,czz2010}. We also remark that the maximum-type testing procedure considered in these examples is expected to be particularly powerful for detecting sparse alternatives (see \cite{clx14}). A comprehensive investigation along this line is left as our future topic.

The underlying theory in supporting these high dimensional applications is a general Gaussian approximation theory and its bootstrap version. The Gaussian approximation theory quantifies the Kolmogorov distance between the largest element of a sum of weakly dependent vectors and its Gaussian analog that shares the same autocovariance structure. We develop our theory in the general framework of dependency graph, which leads to delicate bounds on the Kolmogorov distance for various types of time series. The approximation error, which is finite sample valid, decreases polynomially in sample size even when the data dimension is exponentially high. Moreover, we study two important dependence structures in more details: $M$-dependent time series and weakly dependent time series. Although the sharpness of Kolmogorov distance is not established in this paper, our theoretical results (also see Figure \ref{fig:interplay}) strongly indicate an interesting interplay between dependence and dimensionality: the less dependent of the data vectors, the faster diverging rate of the dimension is allowed for obtaining an accurate Gaussian approximation. We also propose an interesting ``dimension free" dependence structure that allows the dimension to diverge at the rate as if the data were independent. However, in practice, the intrinsic dependence structure of time series is usually unknown. This motivates us to develop a bootstrap version of the Gaussian approximation theory that does not require such knowledge. Specifically, we propose a blockwise multiplier bootstrap that is able to capture the dependence amongst and within the data vectors. Moreover, it inherits the high quality approximation without relying on the autocovariance information. We also introduce a non-overlapping block bootstrap as a more flexible alternative. The above theoretical results are major building blocks of a general framework of conducting bootstrap inference for high dimensional time series. This general framework assumes that the quantity of interest admits an approximately linear expansion, and thus covers the three examples mentioned above. This quantity of interest can be expressed as a functional of the distribution of the time series with finite or infinite length. Hence, our result is also useful in making inference for the spectrum of time series.

Our general Gaussian approximation theory and its block bootstrap version substantially relax the independence assumption in \cite{cck13,abr10}, and is established using several techniques including the Slepian interpolation
\cite{Rollin2011}, leave-one-block-out argument (modification of
Stein's leave-one-out argument \cite{stein1986}), self-normalization
\cite{dls2009}, weak dependence measure \cite{wu2005}, and
$M$-dependent approximation \cite{ll2009}. It is worth pointing out
that our results are established under the physical/functional
dependence measure proposed in \cite{wu2005}. This framework (or its
variants) is known to be very general and easy to verify for linear
and nonlinear data-generating mechanisms, and it also provides a
convenient way for establishing large-sample theories for stationary
causal processes \cite{wu2005,cxw2013,wu2011}. In particular, our
work is largely inspired by a recent breakthrough in Gaussian
approximation for i.i.d data (\cite{cck13}) that obtained an
astounding improvement over the previous results in \cite{bent03} by
allowing the dimension of the data vectors to be exponentially
larger than the sample size.

The rest is organized as follows. In Section \ref{sec:statappl}, we describe three concrete bootstrap inference procedures mentioned above in details. Section \ref{sec:maxima} gives the Gaussian approximation result that works even when the dimension is exponentially larger than sample size, and Section \ref{sec:boot} proposes the blockwise multiplier (wild) bootstrap and also the non-overlapping block bootstrap that do not depend on the autocovariance structure of time series. Building on the results in Sections~\ref{sec:maxima} and~\ref{sec:boot}, a general framework of conducting bootstrap inference based on approximately linear statistics is established in Section \ref{sec:ts}. Three examples considered in \ref{sec:statappl} and one spectral testing example are covered by this framework. All the proofs are gathered in the supplementary material.

\section{High Dimensional Inference}\label{sec:statappl}
To motivate our general theory, we consider three concrete bootstrap
inference procedures for high dimensional time series: uniform
confidence band; white noise testing; and bandedness testing for
covariance matrix. These procedures are rather
straightforward to implement. The main focus of this section is
mostly on the methodological side, and the general theoretical
results are deferred to Section \ref{sec:ts}. An ad-hoc way of
choosing block size in bootstrap is discussed in
Section~\ref{sec:ucb}.

\subsection{Uniform confidence band}\label{sec:ucb}
Consider $n$ observations from a sequence of weakly dependent
$p$-dimensional time series $\{x_i\}$ with
$x_i=(x_{i1},\dots,x_{ip})'$. We are interested in constructing a
$100(1-\alpha)$th uniform confidence band for the mean vector
$\mu_0=(\mu_{01},\mu_{02},\dots,\mu_{0p})'$ in the form of
\begin{equation}\label{eg:ucb}
\left\{\mu=(\mu_1,\dots,\mu_p)'\in\mathbb{R}^p: \sqrt{n}\max_{1\leq
j\leq p}|\mu_j-\bar{x}_{nj}|\leq c(\alpha)\right\},
\end{equation}
where
$\bar{x}_n=(\bar{x}_{n1},\dots,\bar{x}_{np})'=\sum^{n}_{i=1}x_i/n$.
In the traditional low dimensional regime, confidence region for the
mean of a multivariate time series is typically constructed by
inverting a suitable test. A common choice is the Wald type test
which is of the form
$n(\bar{x}_n-\mu)'\widehat{\Sigma}^{-1}(\bar{x}_n-\mu)$, where
$\mu=(\mu_1,\dots,\mu_p)'$ and $\widehat{\Sigma}$ is a consistent
estimator of the so-called long run variance matrix. However,
obtaining a consistent $\widehat\Sigma$ could be difficult in
practice due to the unknown dependence structure. To avoid this
hassle, several appealing nonparametric alternatives, e.g., moving
block bootstrap method \cite{Carl1986,kun89,ls92}, subsampling
approach \cite{prw99} and block-wise empirical likelihood
\cite{kit97}, have been proposed. In the high dimensional regime,
where the dimension of the time series is comparable with or even
much larger than the sample size, inverting the Wald type test is no
longer applicable because the long run variance estimator
$\widehat{\Sigma}$ is singular for $p>n$. Moreover, the direct
application of the nonparametric approaches described above to the
high dimensional setting is unclear yet.

In this subsection, we propose a bootstrap-assisted method to obtain
the critical value $c(\alpha)$ in (\ref{eg:ucb}), whose theoretical
validity will be justified in  Section \ref{subsec:als}.
Specifically, we introduce the following blockwise multiplier (wild)
bootstrap. For simplicity, suppose $n=b_nl_n$ with $b_n,l_n\in
\mathbb{Z}$. Define the non-overlapping block sums,
$$\widehat{A}_{ij}=\sum^{ib_n}_{l=(i-1)b_n+1}(x_{lj}-\bar{x}_{nj}),\quad i=1,2,\dots,l_n,$$
and the bootstrap statistic,
\begin{align*}
T_{\widehat{A}}=\max_{1\leq j\leq
p}\frac{1}{\sqrt{n}}\left|\sum^{l_n}_{i=1}\widehat{A}_{ij}e_i\right|,
\end{align*}
where $\{e_{i}\}$ is a sequence of i.i.d. $N(0,1)$ random variables
independent of $\{x_i\}$. The bootstrap critical value is defined as
$c(\alpha):=\inf\{t\in\mathbb{R}:P(T_{\widehat{A}}\leq
t|\{x_i\}^{n}_{i=1})\geq 1-\alpha\}.$

We next conduct a small simulation study to assess the finite sample
coverage probability of the uniform confidence band. Consider a
$p$-dimensional VAR(1) (vector autoregressive) process,
\begin{equation}\label{eq:VAR}
x_{t}=\rho x_{t-1}+\sqrt{1-\rho^2}\epsilon_{t},
\end{equation}
where $\epsilon_t=(\epsilon_{t1},\dots,\epsilon_{tp})'$. For the
error process $\{\epsilon_t\}$, we consider three cases: (\rmnum{1})
$\epsilon_{tj}=(\varepsilon_{tj}+\varepsilon_{t0})/\sqrt{2},$ where
$(\varepsilon_{t0},\varepsilon_{t1},\dots,\varepsilon_{tp})'\overset{i.i.d.}{\sim}
N(0,I_{p+1})$; (\rmnum{2})
$\epsilon_{tj}=\rho_1\zeta_{tj}+\rho_2\zeta_{t(j+1)}+\cdots+\rho_p\zeta_{t(j+p-1)}$,
where $\{\rho_j\}^{p}_{j=1}$ are generated independently from
$\text{Unif}(2,3)$ (uniform distribution on [2,3]), and
$\{\zeta_{tj}\}$ are i.i.d $N(0,1)$ random variables; (\rmnum{3})
$\epsilon_{tj}$ is generated from the moving average model in
(\rmnum{2}) with $\{\zeta_{tj}\}$ being i.i.d centralized
Gamma$(4,1)$ random variables. Set $n=120$, $p=500,1000$, and
$\rho=0.2$ or $0.5$ in (\ref{eq:VAR}). To implement the blockwise
multiplier bootstrap, we choose $b_n=4,6,8,10,12,15,20.$

Table \ref{tab:mean} reports the coverage probabilities at 90\% and
95\% nominal levels based on 5000 simulations and 499 bootstrap
resamples. We note that the coverage probabilities appear to be low
for relatively small block size. When $\rho$ increases, a larger
block size is generally required to capture the dependence. Although
the coverage probability is generally sensitive to the choice of the
block size, with a proper block size, the coverage probability can
be reasonably close to the nominal level. For univariate time
series, there are two major approaches for selecting the optimal
block size: the nonparametric plug-in method (e.g. \cite{bk99}) and
the empirical criteria-based method \cite{HHJ95}. However, these
selection procedures are deduced based on the bias-variance
tradeoff, which are not intended to guarantee the best coverage of
confidence interval. Moreover, it is still unclear how these
selection rules can be extended to the high dimensional context.

Hence, we provide an ad-hoc way for choosing the block size below. Given a set of realizations $\{x_t\}^{n}_{t=1}$, we pick an initial
block size $b_{int}$ such that $n=b_{int}l_{int}$ where
$b_{int},l_{int}\in\mathbb{Z}$. Conditional on the sample
$\{x_t\}_{t=1}^{n}$, we let $s_1,\dots,s_{l_{int}}$ be i.i.d uniform
random variables on $\{0,\dots,l_{int}-1\}$ and define
$x_{(j-1)b_{int}+i}^*=x_{s_jb_{int}+i}$ with $1\leq j\leq l_{int}$
and $1\leq i\leq b_{int}.$ In other words, $\{x_t^*\}_{t=1}^{n}$ is
a non-overlapping block bootstrap sample with block size $b_{int}$.
For each $b_n$ (block size for the original sample), we can compute
the times that the sample mean $\bar{x}_n$ is contained in the
uniform confidence band constructed based on the bootstrap sample
$\{x_t^*\}_{t=1}^{n}$ and then compute the empirical coverage
probabilities based on $B$ bootstrap samples. This is based on the
notion that $\bar{x}_n$ is the true mean for the bootstrap sample
conditional on $\{x_t\}^{n}_{t=1}$. In this case, the block size,
which delivers the most accurate coverage for $\bar x_n$, can be
viewed as an estimate of the optimal $b_n$ for the original series.
We employ the above procedure with $b_{int}=6$ and $B=500$ to choose
the optimal block size. Based on 200 realizations from the original
data generating process, the coverage probabilities (given the
selected block size) in different simulation setup are summarized in
Table \ref{tab:mean-opt}. We observe that the coverage probability
based on the optimal block size is close to the best coverage
presented in Table \ref{tab:mean}. Finally we point out that it
might be possible to iterate the above procedure to further improve
the empirical performance.

\begin{table}[!h]\scriptsize
\caption{Coverage probabilities of the uniform confidence band for
the mean, where the block size $b_n=4,6,8,10,12,15,20,$ and
$n=120$.}\label{tab:mean}
\begin{center}
\begin{tabular}{l rrrrrrrrrrrr}\toprule
&\multicolumn{2}{c}{$p=500$,(\rmnum{1})}&\multicolumn{2}{c}{$p=500$,(\rmnum{2})}&\multicolumn{2}{c}{$p=500$,(\rmnum{3})}
&\multicolumn{2}{c}{$p=1000$,(\rmnum{1})}&\multicolumn{2}{c}{$p=1000$,(\rmnum{2})}&\multicolumn{2}{c}{$p=1000$,(\rmnum{3})}
\\\cmidrule(r){2-3}\cmidrule(r){4-5}\cmidrule(r){6-7}\cmidrule(r){8-9}\cmidrule(r){10-11}\cmidrule(r){12-13}
&
90\%& 95\% & 90\% & 95\%& 90\% & 95\% & 90\% & 95\% & 90\% & 95\% & 90\% & 95\% \\
\midrule
$\rho=0.2$\\
$b_n=4$  & 85.0 & 92.2 & 85.6 & 92.6 & 85.5 & 91.7 & 86.0 & 92.8 & 84.8 & 91.9 & 84.7 & 91.4\\
$b_n=6$  & 87.8 & 93.8 & 85.8 & 92.7 & 86.0 & 92.4 & 87.7 & 94.5 & 86.0 & 92.6 & 85.8 & 92.7\\
$b_n=8$  & 89.1 & 95.5 & 85.7 & 92.3 & 86.4 & 93.1 & 89.2 & 95.1 & 85.8 & 92.2 & 85.6 & 92.3\\
$b_n=10$ & 89.5 & 95.7 & 85.7 & 92.3 & 85.2 & 92.1 & 90.7 & 96.0 & 85.9 & 92.5 & 86.1 & 92.5\\
$b_n=12$ & 89.2 & 95.3 & 85.4 & 91.8 & 85.4 & 92.5 & 90.4 & 96.5 & 84.7 & 91.9 & 86.4 & 92.9\\
$b_n=15$ & 90.3 & 96.0 & 84.6 & 91.8 & 85.2 & 92.3 & 90.2 & 96.4 & 85.0 & 92.3 & 85.3 & 92.4\\
$b_n=20$ & 90.2 & 96.5 & 83.0 & 90.7 & 83.2 & 90.8 & 91.2 & 96.9 & 84.1 & 91.3 & 84.2 & 91.9\\
\hline
$\rho=0.5$\\
$b_n=4$  & 62.9 & 76.9 & 73.6 & 83.5 & 73.3 & 83.3 & 64.3 & 78.1 & 73.0 & 82.7 & 73.2 & 82.8\\
$b_n=6$  & 76.5 & 87.1 & 79.1 & 87.3 & 78.9 & 87.4 & 76.4 & 86.6 & 78.6 & 87.4 & 78.1 & 87.1\\
$b_n=8$  & 81.5 & 91.6 & 80.8 & 88.8 & 80.7 & 89.4 & 81.9 & 91.0 & 80.8 & 88.9 & 80.9 & 88.9\\
$b_n=10$ & 84.2 & 92.5 & 81.5 & 89.8 & 81.5 & 89.3 & 84.9 & 93.5 & 82.2 & 90.1 & 82.5 & 89.9\\
$b_n=12$ & 84.6 & 93.0 & 82.2 & 90.0 & 82.3 & 90.5 & 86.2 & 94.4 & 81.6 & 89.9 & 83.3 & 90.9\\
$b_n=15$ & 87.0 & 94.3 & 82.0 & 90.1 & 82.5 & 90.7 & 87.1 & 94.6 & 82.2 & 90.1 & 82.5 & 89.9\\
$b_n=20$ & 88.0 & 95.5 & 81.0 & 89.3 & 81.9 & 89.8 & 88.9 & 96.0 & 81.6 & 89.9 & 83.3 & 90.9\\
\midrule
\end{tabular}
\end{center}
\end{table}

\begin{table}[!h]\footnotesize
\caption{Coverage probabilities of the uniform confidence band for
the mean, where the block size is chosen automatically, $p=500$,
$n=120$, and the nominal level is 95\%.}\label{tab:mean-opt}
\begin{center}
\begin{tabular}{rrrrrr}\toprule
 $\rho=0.2$,(\rmnum{1}) & $\rho=0.2$,(\rmnum{2}) & $\rho=0.2$,(\rmnum{3}) & $\rho=0.5$,(\rmnum{1}) & $\rho=0.5$,(\rmnum{2}) & $\rho=0.5$,(\rmnum{3})\\
\midrule
95.0 & 91.5 & 92.5 & 95.0 & 90.5 & 89.0\\
\midrule
\end{tabular}
\end{center}
\end{table}

\subsection{Testing for serial correlation}\label{subsec:cov}

Covariance matrix plays a crucial role in many areas of statistical
inference. For independent vectors, many methods have been developed
for testing specific structures of covariance matrices (see e.g.
\cite{cj11,czz2010,lc2012,qc2012} for some recent developments). In
this subsection, we examine the serial correlation of a sequence of
time series data by testing its autocovariance matrix (a more general measure than covariance matrix).

To illustrate the idea, let
$\gamma(l)=(\gamma_{jk}(l))_{j,k=1}^{p}=\E x_{i}x_{i+l}'\in
\mathbb{R}^{p\times p}$ be the autocovariance matrix of a
$p$-dimensional stationary time series $\{x_i\}$ with $\E x_i=0$.
Consider the null hypothesis
$$H_0:\gamma(l)=\widetilde{\gamma}(l):=(\widetilde{\gamma}_{jk}(l))_{j,k=1}^p,$$
for any $l\in \Lambda \subset \{0,1,2,\dots\}$ versus the
alternative that $H_a: \gamma(l)\neq \widetilde{\gamma}(l)$ for some
$l\in \Lambda.$ The cardinality of $\Lambda$ is allowed to grow with
the dimension $p$. Let
$\widehat{\gamma}_{jk}(l)=\sum^{n-l}_{i=1}x_{ij}x_{(i+l)k}/n$ for
$1\leq j,k\leq p$ be the sample autocovariance at lag $l$. Our test
rejects the null hypothesis $H_0$ if
\begin{equation}\label{eq:cov-general}
\sqrt{n}\max_{l\in \Lambda}\max_{1\leq j,k\leq
p}|\widehat{\gamma}_{jk}(l)-\widetilde{\gamma}_{jk}(l)|>c(\alpha),
\end{equation}
where $c(\alpha)$ denotes the bootstrap critical value at level
$\alpha$. This framework includes several important applications
such as white noise testing (i.e., testing for serial correlation) and covariance testing.

In the white noise testing, we consider $H_0:
\gamma(l)=\mathbf{0}_{p\times p}$ for any $1\leq l\leq L$ v.s. $H_a:
\gamma(l)\neq \mathbf{0}_{p\times p}$ for some $1\leq l\leq L$,
where $\mathbf{0}_{p\times p}$ denotes a $p\times p$ matrix of all
zeros. This is a standard diagnostic procedure in time series
analysis, e.g., \cite{bp1970,Robinson1991,hong96,deo2000} among
others. However, in the high dimensional setting, i.e., $p^2\gg n$,
there seems no systematic method available to test the white noise
assumption. The proposed test statistic $\sqrt{n}\max_{1\leq l\leq
L}\max_{1\leq j,k\leq p}|\widehat{\gamma}_{jk}(l)|$ fills in this
gap. Again, we employ the blockwise multiplier bootstrap to obtain
the critical value $c(\alpha)$. To proceed, we let
$\nu_i=(\nu_{i,1},\dots,\nu_{i,p^2L})=(\text{vec}(x_{i}x_{i+1}')',\dots,\text{vec}(x_{i}x_{i+L}')')'\in\mathbb{R}^{p^2L}$
for $i=1,\dots,N:=n-L$, where $\text{vec}$ denotes the operator that
stacks the columns of a $p\times p$ matrix as a vector with $p^2$
components. Suppose $N=b_nl_n$ for $b_n,l_n\in\mathbb{Z}$. Define
\begin{align*}
\widetilde{T}_{\widehat{A}}=\max_{1\leq j\leq
p^2L}\frac{1}{\sqrt{n}}\left|\sum^{l_n}_{i=1}\widehat{A}_{ij}e_i\right|,\quad
\widehat{A}_{ij}=\sum^{ib_n}_{l=(i-1)b_n+1}(\nu_{l,j}-\bar{\nu}_{nj}),
\end{align*}
where $\{e_{i}\}$ is a sequence of i.i.d standard normal
independent of $\{x_i\}$, and
$\bar{\nu}_{nj}=\sum^{N}_{i=1}\nu_{i,j}/N.$ The bootstrap critical
value is then given by
$c(\alpha):=\inf\{t\in\mathbb{R}:P(\widetilde{T}_{\widehat{A}}\leq
t|\{x_i\}^{n}_{i=1})\geq 1-\alpha\}$. The above procedure can be
easily modified to get the critical value for the general test
described in (\ref{eq:cov-general}).

When assuming $\Lambda=\{0\}$, we obtain an important by-product:
covariance structure testing for high dimensional vector. In this
case, our test reduces to $\sqrt{n}\max_{1\leq j\leq k\leq
p}|\widehat{\gamma}_{jk}(0)-\widetilde{\gamma}_{jk}(0)|>c(\alpha).$
Compared to the existing work in the independence case, e.g.,
\cite{clx14}, our test enjoys three appealing features: (\rmnum{1})
it allows dependence amongst
data vectors and relaxes the Gaussian assumption; (\rmnum{2}) it does not require the existence of a
null limit distribution such as the extreme distribution of Type
\Rmnum{1} in \cite{cj11}. Hence, we can avoid the slow convergence
issue of the extreme value distribution (see \cite{lls08}), which
causes an inaccurate critical value. Rather, a blockwise multiplier
bootstrap is employed to provide high quality approximation;
(\rmnum{3}) it does not impose strong restrictions on the covariance
structure such as sparsity on the precision matrix \cite{clx14} or
pseudo-independence among its components \cite{cq2010,czz2010}.

To evaluate the finite sample performance of the white noise testing
procedure, we consider the following data generating processes:
(\rmnum{1}) independent normal random vectors whose covariance
structure is determined by a moving average model
$x_{ij}=\rho_1\zeta_{ij}+\rho_2\zeta_{i(j+1)}+\cdots+\rho_p\zeta_{i(j+p-1)}$,
where $\{\rho_j\}^{p}_{j=1}$ are generated independently from
$\text{Unif}(2, 3)$, and $\{\zeta_{ij}\}$ are i.i.d $N(0,1)$ random
variables; (\rmnum{2}) multivariate ARCH model defined as
$x_i=\Sigma^{1/2}_i\epsilon_i$ with $\epsilon_i\sim N(0,I_p)$ and
$\Sigma_i=0.1I_p+0.9x_{i-1}x_{i-1}'$, where $\Sigma^{1/2}_i$ is a
lower triangular matrix based on the Cholesky decomposition of
$\Sigma_i$; (\rmnum{3}) VAR(1) model $x_{i}=\rho
x_{i-1}+\sqrt{1-\rho^2}\epsilon_{i}$, where $\rho=0.3$ and the
errors $\{\epsilon_{i}\}$ are generated according to (\rmnum{1}). We
consider $n=60$ and $p=30$ or $50.$ Notice that the actual number of
parameters in consideration is $p^2\times L$, where $L$ is the
number of lags specified in the hypothesis. Table \ref{tab:white}
summarizes the rejection probabilities at 10\% and 5\% nominal
levels based on 5000 simulations and 499 bootstrap resamples. In
general, the proposed method delivers reasonable size and power,
although we still observe some downward size distortion and power
loss especially for $L=3$. The power loss here is presumably due to
the correlation structure of the VAR(1) model. It is also worth
noting that the choice of $b_n=1$ generally performs well for the
martingale difference sequences considered under the null.

\begin{table}[!h]\footnotesize
\caption{Rejection percentages for testing the uncorrelatedness,
where the block size $b_n=1,2,3,4,5,6,$ and $n=60$. Cases (i) and
(ii) are under null, while case (iii) is under
alternative.}\label{tab:white}
\begin{center}
\begin{tabular}{l rrrrrrrrrrrr}\toprule
&\multicolumn{2}{c}{$p=30$,(\rmnum{1})}&\multicolumn{2}{c}{$p=30$,(\rmnum{2})}&\multicolumn{2}{c}{$p=30$,(\rmnum{3})}
&\multicolumn{2}{c}{$p=50$,(\rmnum{1})}&\multicolumn{2}{c}{$p=50$,(\rmnum{2})}&\multicolumn{2}{c}{$p=50$,(\rmnum{3})}
\\\cmidrule(r){2-3}\cmidrule(r){4-5}\cmidrule(r){6-7}\cmidrule(r){8-9}\cmidrule(r){10-11}\cmidrule(r){12-13}
&
10\%& 5\% & 10\% & 5\%& 10\% & 5\% & 10\% & 5\% & 10\% & 5\% & 10\% &5\% \\
\midrule
$L=1$\\
$b_n=1$  & 8.1  & 3.5 & 9.2 & 3.2 & 73.1 & 59.1 & 8.3 & 3.9 & 8.9 & 2.8 & 72.4 & 58.9\\
$b_n=2$  & 9.2  & 3.6 & 7.0 & 2.4 & 67.0 & 49.6 & 8.7 & 3.2 & 6.7 & 2.2 & 68.2 & 49.4\\
$b_n=3$  &10.8  & 4.6 & 6.8 & 2.5 & 66.0 & 46.4 & 9.9 & 4.1 & 6.9 & 2.6 & 66.4 & 46.7\\
$b_n=4$  &10.9  & 4.6 & 6.7 & 3.0 & 67.0 & 46.8 & 11.0& 4.1 & 6.9 & 3.0 & 66.4 & 46.3\\
$b_n=5$  &11.4  & 4.5 & 7.8 & 3.7 & 69.2 & 47.5 & 11.6& 4.5 & 7.8 & 3.7 & 67.6 & 46.6\\
$b_n=6$  &12.7  & 5.2 & 9.2 & 4.7 & 67.7 & 47.7 & 12.3& 5.1 & 8.3 & 4.4 & 68.2 & 48.2\\
\hline
$L=3$\\
$b_n=1$  & 7.2  & 2.4 & 8.5 & 3.3 & 58.3 & 43.8 & 6.7 & 2.5 & 8.6 & 3.2 & 58.7 & 43.4\\
$b_n=2$  & 7.6  & 2.7 & 5.4 & 2.1 & 51.3 & 33.0 & 7.9 & 3.0 & 5.3 & 2.4 & 51.3 & 32.4\\
$b_n=3$  & 6.9  & 2.3 & 3.9 & 1.5 & 46.4 & 28.1 & 6.4 & 2.0 & 3.7 & 1.6 & 46.9 & 27.7\\
$b_n=4$  & 7.0  & 2.3 & 3.8 & 2.0 & 47.0 & 27.4 & 6.6 & 2.0 & 4.2 & 2.2 & 47.5 & 28.2\\
$b_n=5$  & 7.8  & 2.4 & 5.1 & 2.4 & 48.6 & 28.2 & 7.4 & 2.2 & 4.6 & 2.5 & 47.7 & 27.5\\
$b_n=6$  & 7.9  & 2.5 & 6.4 & 3.8 & 49.3 & 28.1 & 8.7 & 2.7 & 5.9 & 3.2 & 49.1 & 28.1\\
\midrule
\end{tabular}
\end{center}
\end{table}

\begin{remark}
{\rm The simulation results demonstrate the usefulness of the
proposed method but they also leave some room for improvement. Here
we point out two possibilities: (\rmnum{1}) it is of interest to
study the studentized version of the test statistic which may be
more efficient as expected in the low dimensional setting (see
Remark \ref{rk:student}); (\rmnum{2}) in the sparsity situation, the
test statistic can be constructed based on a suitable linear
transformation of the observations. The linear transformation aims
to magnify the signals owing to the dependence within the data
vector under alternatives, and hence improves the power of the
testing procedure, e.g., \cite{HJ10,clx14}.}
\end{remark}

\subsection{Bandedness testing of covariance matrix}\label{sec:bandtest}
In this subsection, we consider testing the bandedness of covariance
matrix $\gamma(0)$. This problem aries, for example, in econometrics
when testing certain economic theories; see \cite{andrews1991,lb95}
and reference therein. Also see \cite{cj11,qc2012} for independent case. 
For any integer $\iota\geq 1$ (which possibly
depends on $n$ or $p$), we want to test
\begin{equation}
H_0:\gamma_{jk}(0)=0,\quad |j-k|\geq \iota.
\end{equation}
Our setting significantly generalizes the one considered in
\cite{cj11} which focuses on independent Gaussian vectors. Here, we
shall allow non-Gaussian and dependent random vectors.

We define the test statistic as
\begin{equation}
T_{band}=\sqrt{n}\max_{|j-k|\geq
\iota}\left|\frac{\widehat{\gamma}_{jk}(0)}{\sqrt{\widehat{\gamma}_{jj}(0)\widehat{\gamma}_{kk}(0)}}\right|=\max_{|j-k|\geq
\iota}\frac{1}{\sqrt{n}}\left|\sum^{n}_{i=1}\frac{x_{ij}x_{ik}}{\sqrt{\widehat{\gamma}_{jj}(0)\widehat{\gamma}_{kk}(0)}}\right|.
\end{equation}
For $n=b_nl_n$ with $b_n,l_n\in\mathbb{Z}$, we define the block sums
\begin{align*}
\widehat{A}_{i,jk}=\sum^{ib_n}_{l=(i-1)b_n+1}\frac{x_{ij}x_{ik}-\widehat{\gamma}_{jk}(0)}{\sqrt{\widehat{\gamma}_{jj}(0)\widehat{\gamma}_{kk}(0)}},\quad
i=1,2,\dots,l_n,
\end{align*}
and the bootstrap statistic
\begin{align*}
T_{band,\widehat{A}}=\max_{|j-k|\geq \iota
}\left|\frac{1}{\sqrt{n}}\sum^{l_n}_{i=1}\widehat{A}_{i,jk}e_i\right|,
\end{align*}
where $\{e_{i}\}$ is a sequence of i.i.d $N(0,1)$ 
independent of $\{x_i\}$. We reject the null $H_0$ if
$T_{band,\widehat{A}}>c_{band}(\alpha)$, where $
c_{band}(\alpha):=\inf\{t\in\mathbb{R}:P(T_{band,\widehat{A}}\leq
t|\{x_i\}^{n}_{i=1})\geq 1-\alpha\}.$ Alternatively, one can employ
the non-overlapping block bootstrap (to be presented in Sections \ref{subsec:block}) to obtain the critical value. 

\section{Gaussian Approximation Theory}\label{sec:maxima}
In this section, we derive a Gaussian approximation theory that
serves as the first step in studying high dimensional inference
procedures in Section~\ref{sec:statappl}. Consider a sequence of
$p$-dimensional dependent random vectors $\{x_i\}^{n}_{i=1}$ with
$x_i=(x_{i1},\dots,x_{ip})'$. Suppose $\E x_i=0$ and
$\Sigma_{i,j}:=\cov(x_i,x_j)\in\mathbb{R}^{p\times p}$. The Gaussian
counterpart is defined as a sequence of Gaussian random variables
$\{y_i\}^{n}_{i=1}$ independent of $\{x_i\}_{i=1}^n$. In addition,
$\{y_i\}^{n}_{i=1}$ preserves the autocovariance structure of
$\{x_i\}$ in the sense that $\E y_i=0$ and
$\cov(y_i,y_j)=\Sigma_{i,j}$ (note that this assumption can be
weakened, see Remark \ref{rm:cov}). Gaussian approximation theory
quantifies the Kolmogorov distance defined as
\begin{equation}\label{dfn:rhon}
\rho_n:=\sup_{t\in\mathbb{R}}\left|P(T_X\leq t)-P(T_Y\leq t)\right|,
\end{equation}
where $T_X=\max_{1\leq j\leq p}X_j$, $T_Y=\max_{1\leq j\leq p}Y_j$,
and
\begin{equation}
X=(X_1,\dots,X_p)'=\frac{1}{\sqrt{n}}\sum^{n}_{i=1}x_i,\quad
Y=(Y_1,\dots,Y_p)'=\frac{1}{\sqrt{n}}\sum^{n}_{i=1}y_i.
\end{equation}

Chernozhukov et al (2013) recently showed that for independent data
vectors, $\rho_n$ decays to zero polynomially in the sample size. In
Section~\ref{subsec:dep-graph}, we substantially relax their
independence assumption by first establishing a general proposition,
i.e., Proposition~\ref{prop1}, in the framework of dependency graph.
This general result leads to delicate bounds on the Kolmogorov
distance for various types of weakly dependent time series even when
their dimension is exponentially high, i.e.,  Sections
\ref{subsec:m-dep} -- \ref{subsec:weak-dep}.

\subsection{General framework: dependency graph}\label{subsec:dep-graph}

In this subsection, we introduce a flexible framework in modelling the
dependence among a sequence of $p$-dimensional dependent ({\em
unnecessarily identical}) random vectors $\{x_i\}^{n}_{i=1}$. We
call it as dependency graph $G_n=(V_n,E_n)$, where
$V_n=\{1,2,\dots,n\}$ is a set of vertices and $E_n$ is the
corresponding set of undirected edges. For any two disjoint subsets
of vertices $S,T\subseteq V_n$, if there is no edge from any vertex
in $S$ to any vertex in $T$, the collections $\{x_i\}_{i\in S}$ and
$\{x_i\}_{i\in T}$ are independent. Let $D_{\max,n}=\max_{1\leq
i\leq n}\sum_{j=1}^{n}\mathbf{I}\{\{i,j\}\in E_n\}$ be the maximum
degree of $G_n$ and denote $D_{n}=1+D_{\max,n}$. Throughout the
paper, we allow $D_n$ to grow with the sample size $n.$ For example,
if an array $\{x_{i,n}\}^{n}_{i=1}$ is a $M:=M_n$ dependent sequence
(that is $x_{i,n}$ and $x_{j,n}$ are independent if $|i-j|>M$), then
we have $D_n=2M+1$.

Within this general framework, we want to understand the largest
possible diverging rate of $p$ (w.r.t. $n$) under which the
Kolmogorov distance between the distributions of $T_X$ and $T_Y$,
i.e., $\rho_n$ defined in (\ref{dfn:rhon}), converges to zero.
Recall that $T_X=\max_{1\leq j\leq p}X_j$, $T_Y=\max_{1\leq j\leq
p}Y_j$. The problem of comparing distributions of maxima is
nontrivial since the maximum function $z=(z_1,\dots,z_p)'\rightarrow
\max_{1\leq j\leq p}z_j$ is non-differentiable. To overcome this
difficulty, we consider a smooth approximation of the maximum
function,
$$F_{\beta}(z):=\beta^{-1}\log\left(\sum^{p}_{j=1}\exp(\beta z_j)\right),\quad z=(z_1,\dots,z_p)',$$
where $\beta>0$ is the smoothing parameter that controls the level
of approximation. Simple algebra yields that (see \cite{c2005}),
\begin{align}\label{eq:Fbeta}
0\leq F_{\beta}(z)-\max_{1\leq j\leq p}z_j\leq \beta^{-1}\log p.
\end{align}
Denote by $C^k(\mathbb{R})$ the class of $k$ times continuously
differentiable functions from $\mathbb{R}$ to itself, and denote by
$C^k_b(\mathbb{R})$ the class of functions $f\in C^k(\mathbb{R})$
such that $\sup_{z\in\mathbb{R}}|\partial^j f(z)/\partial
z^j|<\infty$ for $j=0,1,\dots,k.$ Set $m=g\circ F_{\beta}$ with
$g\in C_b^{3}(\mathbb{R})$. In Proposition~\ref{prop1} below, we
derive a non-asymptotic upper bound for the quantity
$|\E[m(X)-m(Y)]|$ by employing the Slepian interpolation
\cite{Rollin2011}, and modifying Stein's leave-one-out argument
\cite{stein1986} to the leave-one-block-out argument for capturing
the local dependence of the data.

Denote the truncated variables $\widetilde{x}_{ij}=(x_{ij}\wedge
M_x)\vee(-M_x)-\E[(x_{ij}\wedge M_x)\vee(-M_x)]$ and
$\widetilde{y}_{ij}=(y_{ij}\wedge M_y)\vee(-M_y)$ for some $M_x,
M_y>0$. Let
$\widetilde{x}_i=(\widetilde{x}_{i1},\dots,\widetilde{x}_{ip})'$ and
$\widetilde{y}_i=(\widetilde{y}_{i1},\dots,\widetilde{y}_{ip})'$.
For $1\leq i\leq n$, let $N_i=\{j: \{i,j\}\in E_n\}$ be the set of
neighbors of $i$, and $\widetilde{N}_i=\{i\}\cup N_i.$ Let
$\phi(M_x)$ be a constant depending on the threshold parameter $M_x$
such that
\begin{align*}
\max_{1\leq j,k\leq p}\frac{1}{n}\sum^{n}_{i=1}\left|\sum_{l\in
\widetilde{N}_i}\left(\E
x_{ij}x_{lk}-\E\widetilde{x}_{ij}\widetilde{x}_{lk}\right)\right|
\leq & \phi(M_{x}).
\end{align*}
Analogous quantity $\phi(M_y)$ can be defined for $\{y_{i}\}.$ Set
$\phi(M_x,M_y)=\phi(M_x)+\phi(M_y)$. Define
\begin{align*}
&m_{x,k}=(\bar{\E}\max_{1\leq j\leq p}|x_{ij}|^k)^{1/k},\quad m_{y,k}=(\bar{\E}\max_{1\leq j\leq p}|y_{ij}|^k)^{1/k},\\
&\bar{m}_{x,k}=\max_{1\leq j\leq p}(\bar{\E}|x_{ij}|^k)^{1/k},\quad
\bar{m}_{y,k}=\max_{1\leq j\leq p}(\bar{\E}|y_{ij}|^k)^{1/k},
\end{align*}
where $\bar{\E}[z_i]=\sum^{n}_{i=1}\E z_i/n$ for a sequence of
random variables $\{z_i\}^{n}_{i=1}$. Note that $\bar{m}_{x,k}\leq
m_{x,k}$ and $\bar{m}_{y,k}\leq m_{y,k}$. Further define an
indicator function,
\begin{align*}
\mathcal{I}:=\mathcal{I}_{\Delta}=\mathbf{1}\left\{\max_{1\leq j\leq
p}|X_j-\widetilde{X}_j|\leq \Delta, \max_{1\leq j\leq
p}|Y_j-\widetilde{Y}_j|\leq \Delta\right\},
\end{align*}
where
$\widetilde{X}=(\widetilde{X}_1,\dots,\widetilde{X}_p)'=\frac{1}{\sqrt{n}}\sum^{n}_{i=1}\widetilde{x}_i$
and
$\widetilde{Y}=(\widetilde{Y}_1,\dots,\widetilde{Y}_p)'=\frac{1}{\sqrt{n}}\sum^{n}_{i=1}\widetilde{y}_i.$
\begin{proposition}\label{prop1}
Assume that $2\sqrt{5}\beta D_n^2M_{xy}/\sqrt{n} \leq 1$ with
$M_{xy}=\max\{M_x,M_y\}$. Then we have for any $\Delta>0,$
\begin{equation}\label{eq:prop1}
\begin{split}
|\E [m(X)-m(Y)]| \lesssim &
(G_2+G_1\beta)\phi(M_x,M_y)+(G_3+G_2\beta+G_1\beta^2)\frac{D_n^{2}}{\sqrt{n}}(\bar{m}_{x,3}^3+\bar{m}_{y,3}^3)
\\&+(G_3+G_2\beta+G_1\beta^2)\frac{D_n^{3}}{\sqrt{n}}(m_{x,3}^3+m_{y,3}^3)+G_1\Delta+G_0\E[1-\mathcal{I}],
\end{split}
\end{equation}
where $G_k=\sup_{z\in\mathbb{R}}|\partial^k g(z)/\partial z^k|$ for
$k\geq 0$. In addition, if $2\sqrt{5}\beta D_n^3M_{xy}/\sqrt{n} \leq
1$, we can replace $m_{x,3}^3+m_{y,3}^3$ by
$\bar{m}_{x,3}^3+\bar{m}_{y,3}^3$ in the above expression.
\end{proposition}
\noindent The proof of Proposition~\ref{prop1} is adapted from that
of Theorem 2.1 in \cite{cck13} for i.i.d case.

By approximating the indicator function $I\{\cdot\leq t\}$ with a
suitable smooth function $g(\cdot)$, Proposition \ref{prop1} leads
to an upper bound on the Kolmogorov distance, i.e., $\rho_n$ defined
in (\ref{dfn:rhon}). In fact, the upper bound in (\ref{eq:prop1})
can be further simplified using the self-normalization technique
(see Lemma \ref{lemma:self}) and certain arguments under weak
dependence assumption. Finally, by optimizing the simplified upper
bound (see Theorem \ref{thm1}), we obtain various convergence rates
for $\rho_n$ in Sections \ref{subsec:m-dep} --
\ref{subsec:weak-dep}.

\begin{remark}\label{rm:cov}
{\rm In view of the proof of Proposition \ref{prop1} (see e.g.
(\ref{eq:cov-ts})), the assumption that $\{y_i\}$ preserves the
autocovariance structure of $\{x_i\}$ can be weakened by assuming
that for all $i,$
$$\sum_{k\in\widetilde{N}_i}\E x_{i}x_{k}'=\sum_{k\in\widetilde{N}_i}\E y_{i}y_{k}'.$$
Thus $\{y_i\}$ is allowed to be a sequence of independent
(mean-zero) $p$-dimensional Gaussian random variables such that
$\cov(y_i)=\sum_{k\in\widetilde{N}_i}\E x_{i}x_{k}'$ (provided that
$\sum_{k\in\widetilde{N}_i}\E x_{i}x_{k}'$ is positive-definite). }
\end{remark}

\begin{remark}
{\rm The arguments in the proof of Proposition~\ref{prop1} allow us
to derive a non-asymptotic upper bound on $\E|m^*(X)-m^*(Y)|$ for a
more general function $m^*(\cdot)$ on the high dimensional vector
sum (after some suitable componentwise transformation); see Section
\ref{subsec:extension}. Such general results are potentially useful
in studying higher criticism test (\cite{zcx2013}); see Example
\ref{example:general-g} and Remark~\ref{os:rem}.}
\end{remark}

\subsection{Dependence structure \Rmnum{1}: $M$-dependent time series}\label{subsec:m-dep}
This subsection is devoted to the analysis of $M$-dependent time
series, which fits in the framework of dependency graph. Here, we
allow $M$ to grow slowly with the sample size $n.$ Using the
arguments in the proof of Proposition \ref{prop1}, we obtain the
following result for $M$-dependent ({\em unnecessarily stationary})
sequence.

\begin{corollary}\label{corollary:m-dep}
When $\{x_i\}$ is a $M$-dependent sequence, under the assumption
that $2\sqrt{5}\beta(6M+1)M_{xy}/\sqrt{n}\leq 1,$ we have
\begin{equation}\label{eq:m-dep1}
\begin{split}
|\E [m(X)-m(Y)]|\lesssim &
(G_3+G_2\beta+G_1\beta^2)\frac{(2M+1)^{2}}{\sqrt{n}}(\bar{m}_{x,3}^3+\bar{m}_{y,3}^3)
\\&+(G_2+G_1\beta)\phi(M_x,M_y)+G_1\Delta+G_0\E[1-\mathcal{I}].
\end{split}
\end{equation}
\end{corollary}

Let $n=(N+M)r$, where $N\geq M$ and $N,M,r\rightarrow +\infty$ as
$n\rightarrow+\infty.$ Define the block sums
\begin{equation}\label{eq:AB}
A_{ij}=\sum^{iN+(i-1)M}_{l=iN+(i-1)M-N+1}x_{lj},\quad
B_{ij}=\sum^{i(N+M)}_{l=i(N+M)-M+1}x_{lj}.
\end{equation}
It is not hard to see that $\{A_{ij}\}^{r}_{i=1}$ and
$\{B_{ij}\}^{r}_{i=1}$ with $1\leq j\leq p$ are two sequences of
i.i.d random variables. Let $V_{nj}=\sqrt{V_{1nj}^2+V_{2nj}^2}$ with
$V_{1nj}^2=\sum^{r}_{i=1}A_{ij}^2$ and
$V_{2nj}^2=\sum^{r}_{i=1}B_{ij}^2$. By generalizing Theorem 2.16 of
de la Pe\~{n}a et al (2009), we obtain the following lemma.
\begin{lemma}\label{lemma:self}
Suppose $\{x_i\}$ is a $p$-dimensional $M$-dependent sequence.
Assume that there exist $a_j,b_j>0$ such that
\begin{align*}
P\left(\sum^{n}_{i=1}x_{ij}>a_j\right)\leq 1/4,\quad
P(V_{nj}^2>b^2_j)\leq 1/4.
\end{align*}
Then we have
\begin{equation}
P\left(\left|\sum^{n}_{i=1}x_{ij}\right|\geq
x(a_j+b_j+V_{nj})\right)\leq 8\exp(-x^2/8),
\end{equation}
for any $1\leq j\leq p.$ In particular, we can choose $b^2_j=4\E
V_{nj}^2$ and $a^2_j=2b^2_j=8\E V_{nj}^2$.
\end{lemma}
\noindent It is worth noting that Lemma \ref{lemma:self} holds
without the stationarity assumption. This lemma is particularly
useful in controlling the last two terms in (\ref{eq:m-dep1}).

Throughout the rest of this subsection, we consider the case where
$\{x_i\}$ is a $M$-dependent \emph{stationary} time series. Define
$\gamma_{x,jk}(l)=\E x_{1j}x_{(1+l)k}$ for $l\geq 0$ and
$\gamma_{x,jk}(l)=\gamma_{x,kj}(-l)$ for $l<0$, where $1\leq j,k\leq
p.$ Let
$\sigma_{j,k}^{(n)}:=\sigma_{j,k}^{(n)}(M)=\sum^{n-1}_{l=1-n}(n-|l|)\gamma_{x,jk}(l)/n$,
$\sigma_{j,k}:=\sigma_{j,k}(M)=\sum^{+\infty}_{l=-\infty}\gamma_{x,jk}(l)$
and
$\sigma^2_{j}=\sigma_j^2(M)=\sum^{+\infty}_{l=-\infty}|\gamma_{x,jj}(l)|$.
Let $\varphi(M_x):=\varphi_{N,M}(M_x)$ be the smallest finite
constant which satisfies that uniformly for $j$,
\begin{equation}\label{eq:varphi}
\begin{split}
&E(A_{ij}-\breve{A}_{ij})^2\leq N\varphi^2(M_x)\sigma_j^2,\quad
E(B_{ij}-\breve{B}_{ij})^2\leq M\varphi^2(M_x)\sigma_j^2,
\end{split}
\end{equation}
where $\breve{A}_{ij}$ and $\breve{B}_{ij}$ are the truncated
versions of $A_{ij}$ and $B_{ij}$ defined as follows:
\begin{align*}
&\breve{A}_{ij}=\sum^{iN+(i-1)M}_{l=iN+(i-1)M-N+1}(x_{lj}\wedge
M_x)\vee (-M_x),\\
&\breve{B}_{ij}=\sum^{i(N+M)}_{l=i(N+M)-M+1}(x_{lj}\wedge M_x)\vee
(-M_x).
\end{align*}
Similarly, we can define the quantity $\varphi(M_y)$ for the
Gaussian sequence $\{y_i\}$. Set
$\varphi(M_{x},M_y)=\varphi(M_x)\vee\varphi(M_y)$. Further let
$u_x(\gamma)$ and $u_y(\gamma)$ be the smallest quantities such that
\begin{equation}\label{eq:gamma}
P\left(\max_{1\leq i\leq n}\max_{1\leq j\leq p}|x_{ij}|\leq
u_x(\gamma)\right)\geq 1-\gamma,\quad P\left(\max_{1\leq i\leq
n}\max_{1\leq j\leq p}|y_{ij}|\leq u_y(\gamma)\right)\geq 1-\gamma.
\end{equation}

Building on the above results, we are ready to derive an upper bound
for $\rho_n$. To this end, consider a ``smooth'' indicator function
$g_0\in C^3(\mathbb{R}): \mathbb{R}\rightarrow [0,1]$ such that
$g_0(s)=1$ for $s\leq 0$ and $g_0(s)=0$ for $s\geq 1.$ Fix any
$t\in\mathbb{R}$ and define $g(s)=g_0(\psi(s-t-e_{\beta}))$ with
$e_{\beta}=\beta^{-1}\log p$. For this function $g$, $G_0=1$,
$G_1\lesssim\psi$, $G_2\lesssim\psi^2$ and $G_3\lesssim\psi^3$.
Here, $\psi$ is a smoothing parameter we will choose carefully in
the proof. Corollary \ref{corollary:m-dep} and Lemma
\ref{lemma:self} imply the following result.
\begin{theorem}\label{thm1}
Consider a $M$-dependent stationary time series $\{x_i\}$. Suppose
$2\sqrt{5}\beta(6M+1)M_{xy}/\sqrt{n}\leq 1$ with
$M_{xy}=\max\{M_x,M_y\}$, and $M_x>u_x(\gamma)$ and
$M_y>u_y(\gamma)$ for some $\gamma\in (0,1)$. Further suppose that
there exist constants $0<c_1<c_2$ such that $c_1<\min_{1\leq j\leq
p}\sigma_{j,j}^{(n)}\leq \max_{1\leq j\leq p}\sigma_{j,j}^{(n)}<c_2$
uniformly holds for all large enough $n$, $M$ and $p$. Then for any
$\psi>0$,
\begin{align*}
\rho_n=&\sup_{t\in\mathbb{R}}|P(T_X\leq t)-P(T_Y\leq t)|
\\ \lesssim &
(\psi^2+\psi\beta)\phi(M_{x},M_y)+(\psi^3+\psi^2\beta+\psi\beta^2)\frac{(2M+1)^{2}}{\sqrt{n}}(\bar{m}_{x,3}^3+\bar{m}_{y,3}^3)
\\&+\psi\varphi(M_{x},M_y)\sigma_j\sqrt{8\log(p/\gamma)}+\gamma+(e_{\beta}+\psi^{-1})\sqrt{1\vee\log (p\psi)}.
\end{align*}
\end{theorem}
\noindent We point out that the stationarity assumption is
non-essential in the proof of Theorem \ref{thm1}.

To characterize the dependence of $M$-dependent time series, we adopt the
idea of viewing the weakly dependent time series as outputs on
inputs in physical systems \cite{wu2005}. This framework is very
general and easy to verify for specific (linear or nonlinear)
data-generating mechanism; see \cite{wu2011}. With some abuse of
notation, let $\epsilon_i$ be a sequence of mean-zero i.i.d random
variables. Consider a physical system
$\mathcal{G}(\dots,\epsilon_{i-1},\epsilon_{i})$, where
$\{\epsilon_i\}$ are the inputs and
$\mathcal{G}=(\mathcal{G}_1,\dots,\mathcal{G}_p)'$ is a
($p$-dimensional) measurable function such that its output is well
defined. Define the sigma field
$\mathcal{F}_M(i)=\sigma(\epsilon_{i-M},\epsilon_{i-M+1},\dots,\epsilon_i)$
with $M\geq 0.$ We suppose the $M$-dependent sequence $\{x_i\}$ has
the following representation (also see the discussions in the next
subsection),
\begin{align*}
x_i:=x_i^{(M)}=\E[\mathcal{G}(\dots,\epsilon_{i-1},\epsilon_{i})|\mathcal{F}_M(i)]:=\mathcal{G}^{(M)}(\epsilon_{i-M},\epsilon_{i-M+1},\dots,\epsilon_i).
\end{align*}
For any $l\in\mathbb{N}$, let
$x_i^{(l-1)}=\E[x_i|\epsilon_{i+1-l},\dots,\epsilon_i]
=\E[\mathcal{G}(\dots,\epsilon_{i-1},\epsilon_{i})|\mathcal{F}_{l-1}(i)]$
for $l\leq M$, and $x_i^{(l-1)}=x_i$ for $l>M$. By construction,
$x_{1j}$ and $x_{(1+l)k}^{(l-1)}$ are independent for any $1\leq
j,k\leq p$.

Let $h:[0,+\infty)\rightarrow [0,+\infty)$ be a convex and strictly
increasing function with $h(0)=0$. Denote by $h^{-1}(\cdot)$ the
inverse function of $h(\cdot).$ Let
$l_n:=l_n(p,\gamma)=\log(pn/\gamma)\vee 1$.

\begin{assumption}\label{assum:tail}
Suppose one of the following two conditions holds: (\rmnum{1}) $\E
h(\max_{1\leq j\leq p}|x_{ij}|/\mathfrak{D}_n)\leq 1$ with
$\mathfrak{D}_n>0$, and
\begin{align}
n^{3/8}M^{-1/2}l_n^{-5/8}\geq
C_1\max\{\mathfrak{D}_nh^{-1}(n/\gamma),l_n^{1/2}\},\quad
n^{7/4}M^{-1}l_n^{-9/4}\geq C_2N, \label{eq:h}
\end{align}
for some constants $C_1,C_2>0$; (\rmnum{2}) $\max_{1\leq j\leq p}\E
\exp(|x_{ij}|/\mathfrak{D}_n)\leq 1$ with $\mathfrak{D}_n>0$, and
\begin{align}
n^{3/8}M^{-1/2}l_n^{-5/8}\geq
C_3\max\{\mathfrak{D}_nl_n,l_n^{1/2}\},\quad
n^{7/4}M^{-1}l_n^{-9/4}\geq C_4N, \label{eq:h2}
\end{align}
for some constants $C_3,C_4>0$.
\end{assumption}

\begin{theorem}\label{thm2}
Assume that there exist constants $c_1,c_2,c_3>0$ such that
\begin{align*}
&c_1<\min_{1\leq j\leq p}\sigma_{j,j}^{(n)}\leq \max_{1\leq j\leq
p}\sigma_{j,j}^{(n)}<c_2,\quad \max_{1\leq j\leq p}\sigma_j^2<c_3,
\end{align*}
uniformly for all large enough $M,p$, and
\begin{align}
&\limsup_{p}\max_{1\leq k\leq p}\E
|\mathcal{G}_k(\dots,\epsilon_{i-1},\epsilon_{i})|^4<\infty,
\label{eq:condition}\\
&\limsup_{M,p}\max_{1\leq k\leq
p}\sum^{M}_{l=1}(\E|(x_{(1+l)k}-x_{(1+l)k}^{(l-1)})|^3)^{1/3}<\infty.
\label{eq:m-dep-couple}
\end{align}
Condition (\ref{eq:m-dep-couple}) also holds for $\{y_i\}$. Then
under Assumption \ref{assum:tail}, we have
\begin{align}
\rho_n=&\sup_{t\in\mathbb{R}}|P(T_X\leq t)-P(T_Y\leq t)|\lesssim
n^{-1/8}M^{1/2}l_n^{7/8}+\gamma.
\end{align}
\end{theorem}
Suppose $\E (\max_{1\leq j\leq p}|x_{ij}|/\mathfrak{D}_n)^4\leq 1$.
Then with $p\lesssim \exp(n^b)$, $M\asymp N\lesssim n^{b'}$,
$\gamma\asymp n^{-(1-4b'-7b)/8}=o(1)$, and $\mathfrak{D}_n\lesssim
n^{(3-12b'-13b)/32}$, we have Condition (\rmnum{1}) in Assumption
\ref{assum:tail} holds with $h(x)=x^4$, and
\begin{equation}\label{dep1:rhon}
\rho_n\lesssim n^{-(1-4b'-7b)/8}.
\end{equation}
If Condition (\rmnum{2}) in Assumption \ref{assum:tail} holds, we
can still have (\ref{dep1:rhon}) when $\max_{1\leq j\leq p}\E
\exp(|x_{ij}|/\mathfrak{D}_n)\leq 1$, $p\lesssim \exp(n^b)$, $M\asymp
N\lesssim n^{b'}$, $\gamma\asymp n^{-(1-4b'-7b)/8}=o(1)$ and
$\mathfrak{D}_n\lesssim n^{(3-4b'-13b)/8}$.

When $b'=0$ (i.e. $M=O(1)$), our result allows $p=O(\exp(n^b))$ with
$b<1/7$, which is consistent with Corollary 2.1 in \cite{cck13} for
i.i.d random vectors (assuming that $B_n=O(1)$ therein).
\begin{remark}\label{rem:interp}
{\rm The sharpness of $\rho_n$ is not established in
Theorem~\ref{thm2}. However, the upper bound of $\rho_n$ given in
(\ref{dep1:rhon}) leads to two conjectures: (i) Gaussian
approximation becomes less accurate when the data vectors are more
dependent or the data dimension diverges at a faster rate; (ii) the
less dependent of the data vectors, the faster diverging rate of the
dimension is allowed for obtaining an accurate Gaussian approximation.
The above phenomena will also be observed for the weakly dependent
data in Section~\ref{subsec:weak-dep}. Interestingly, we will show
some empirical evidence of both conjectures in that section.}
\end{remark}

\begin{remark}\label{rem:xass}
{\rm Assumption \ref{assum:tail} and (\ref{eq:condition}) impose
tail restrictions on $\{x_{i}\}$. Condition (\ref{eq:m-dep-couple})
requires $\{x_i\}$ to be weakly dependent uniformly as $M$ grows,
and, in particular, (\ref{eq:m-dep-couple}) allows us to quantify
$\phi(M_x,M_y)$ and $\varphi(M_{x},M_y)$; see (\ref{eq:varphi}). }
\end{remark}

\subsection{Dependence structure \Rmnum{2}: Weakly dependent time series}\label{subsec:weak-dep}
In this subsection, we extend the results in Section
\ref{subsec:m-dep} to the weakly dependent case, i.e., $D_n=n+1$.
The key idea here is to approximate the weakly dependent time series
by a $M$-dependent time series, see the approximation error
(\ref{eq:m-approx}) below.

With slightly abuse of notation, suppose the sequence $\{x_i\}$ has
the following causal representation,
\begin{align}\label{eq:causal}
x_i:=x_i^{(\infty)}=\mathcal{G}(\dots,\epsilon_{i-1},\epsilon_{i}),
\end{align}
where $\mathcal{G}=(\mathcal{G}_1,\dots,\mathcal{G}_p)'$ is a
$p$-dimensional measurable function such that $x_i$ is well defined.
To measure the strength of dependence, we let $\{\epsilon_i'\}$ be
an i.i.d copy of $\{\epsilon_i\}$ and
$x_i^*=\mathcal{G}(\dots,\epsilon_{-1},\epsilon_0',\epsilon_1,\dots,\epsilon_i)$,
and define
\begin{equation}
\theta_{i,j,q}(x)=(\E|x_{ij}-x_{ij}^*|^q)^{1/q},\quad
\Theta_{i,j,q}(x)=\sum^{+\infty}_{l=i}\theta_{l,j,q}(x).
\end{equation}
In the subsequent discussions, we assume that the dependence measure
$\sup_{1\leq j\leq p}\Theta_{i,j,q}(x)<\infty$ for some $q>0$. Analogous quantity
$\theta_{i,j,q}(y)$ can be defined for the Gaussian sequence
$\{y_i\}$.

Let
$x_i^{(M)}=(x_{i1}^{(M)},\dots,x_{ip}^{(M)})'=\E[x_i|\mathcal{F}_M(i)]$
be the $M$-dependent approximation sequence for $\{x_i\}$. Define
$X^{(M)}$ in the same way as $X$ by replacing $x_i$ with
$x_i^{(M)}$. Because $|m(x)-m(y)|\leq 2G_0$ and $|m(x)-m(y)|\leq
G_1\max_{1\leq j\leq p}|x_j-y_j|$ (by the Lipschitz property of
$F_{\beta}$), we have
\begin{equation}\label{eq:M-approx-1}
\begin{split}
|\E[m(X)-m(X^{(M)})]| \leq&
|\E[(m(X)-m(X^{(M)}))\mathcal{I}_M]|+|\E[(m(X)-m(X^{(M)}))(1-\mathcal{I}_M)]|
\\ \lesssim& G_1\Delta_M+G_0\E[1-\mathcal{I}_M],
\end{split}
\end{equation}
where
$\mathcal{I}_M:=\mathcal{I}_{\Delta_M,M}=\mathbf{1}\{\max_{1\leq
j\leq p}|X_j-X_j^{(M)}|\leq \Delta_M\}$ for some $\Delta_M>0$
depending on $M$. Suppose $\max_{1\leq j\leq p}\E|x_{ij}|^q<\infty$
for some $q>0$. By Lemma A.1 of \cite{ll2009}, we have
\begin{align*}
(\E |X_j-X_j^{(M)}|^q)^{q'/q}\leq C_qn^{1-q'/2}
\Theta_{M,j,q}^{q'}(x),
\end{align*}
where $q'=\min(2,q)$ and $C_q$ is a positive constant depending on
$q$. For any $q\geq 2$, we obtain
\begin{align*}
\E[1-\mathcal{I}_M]  \leq & \sum^{p}_{j=1}P(|X_j-X_j^{(M)}|\geq
\Delta_M) \leq \sum^{p}_{j=1}\frac{1}{\Delta_M^q}\E
|X_j-X_j^{(M)}|^q
\\ \leq & \sum^{p}_{j=1}\frac{C^{q/2}_q\Theta_{M,j,q}^q(x)}{\Delta_M^q}=\sum^{p}_{j=1}\frac{C^{q/2}_q}{\Delta_M^q}\left(\sum^{+\infty}_{l=M}\theta_{l,j,q}(x)\right)^q.
\end{align*}
Optimizing the bound with respect to $\Delta_M$ in
(\ref{eq:M-approx-1}), we deduce that
\begin{align}\label{eq:m-approx}
|\E[m(X)-m(X^{(M)})]|\lesssim &
(G_0G_1^q)^{1/(1+q)}\left(\sum^{p}_{j=1}\Theta_{M,j,q}^q(x)\right)^{1/(1+q)},
\end{align}
which along with (\ref{eq:Fbeta}) implies that
\begin{align*}
|\E[g(T_X)-g(T_{X^{(M)}})]|\lesssim &
(G_0G_1^q)^{1/(1+q)}\left(\sum^{p}_{j=1}\Theta_{M,j,q}^q(x)\right)^{1/(1+q)}+\beta^{-1}G_1\log
p,
\end{align*}
with $T_{X^{(M)}}=\max_{1\leq j\leq
p}\sum^{n}_{i=1}x_{ij}^{(M)}/\sqrt{n}$.

We give an explicit expression of the approximation error
(\ref{eq:m-approx}) in the following two examples.
\begin{example}
Consider a stationary linear process,
\begin{align*}
x_{ij}=\sum^{+\infty}_{l=0}b_{lj}\epsilon_{(i-l)j},\quad 1\leq j\leq
p,
\end{align*}
where $\sum^{+\infty}_{l=0}|b_{lj}|<\infty$ and
$\epsilon_{i}=(\epsilon_{i1},\dots,\epsilon_{ip})'$ is a sequence of
i.i.d random variables. Simple calculation yields that
$X_j-X_j^{(M)}=\frac{1}{\sqrt{n}}\sum^{n}_{i=1}\sum^{+\infty}_{l=M+1}b_{lj}\epsilon_{(i-l)j}$
and
$\theta_{l,j,q}(x)=|b_{lj}|(\E|\epsilon_{0j}-\epsilon_{0j}'|^q)^{1/q}$.
For $q\geq2$, we have
\begin{align*}
|\E[m(X)-m(X^{(M)})]| \lesssim & (G_0G_1^q)^{1/(1+q)}\max_{1\leq
j\leq
p}(\E|\epsilon_{0j}-\epsilon_{0j}'|^q)^{1/(q+1)}\left(\sum^{p}_{j=1}\left(\sum^{+\infty}_{l=M}|b_{lj}|\right)^q\right)^{1/(q+1)}.
\end{align*}
Under the assumption that $\limsup_{p}\max_{1\leq j\leq
p}(E|\epsilon_{0j}|^q)^{1/q}<\infty$ and $b_{lj}=\rho^l$ with
$\rho<1$, we get
$$|\E[m(X)-m(X^{(M)})]|\lesssim  (G_0G_1^q)^{1/(1+q)}p^{1/(1+q)}\rho^{(qM)/(1+q)}.$$
\end{example}

\begin{example}
Consider a stationary Markov chain defined by an iterated random
function
\begin{align*}
x_i=H(x_{i-1},e_i).
\end{align*}
Here $e_i$'s are i.i.d. innovations, and $H(\cdot,\cdot)$ is an
$\mathbb{R}^p$-valued and jointly measurable function, which
satisfies the following two conditions: (\rmnum{1}) there exists
some $x_0$ such that $\E|H(x_0,e_0)|^{2q}<\infty$ and (\rmnum{2})
\begin{align*}
\rho:=\sup_{x\neq
x'}\frac{(\E|H(x,e_0)-H(x',e_0)|^{2q})^{1/(2q)}}{|x-x'|}<1,
\end{align*}
where $|\cdot|$ denotes the Euclidean norm for a $p$-dimensional
vector. Then it can be shown that $\{x_i\}$ has the geometric moment
contraction (GMC) condition property \cite{wushao2004} and
$\max_{1\leq j\leq p}\Theta_{m,j,2q}(x)=O(\rho^{m})$ (see Example
2.1 in \cite{cxw2013}). Hence
\begin{align*}
|\E[m(X)-m(X^{(M)})]|\lesssim
(G_0G_1^q)^{1/(1+q)}p^{1/(1+q)}\rho^{(qM)/(1+q)}.
\end{align*}
\end{example}


We are now ready to present the main result .
Recall that $h(\cdot)$ and $l_n$ are defined in
Section~\ref{subsec:m-dep}.
\begin{theorem}\label{thm:Gaussian-dep}
Suppose $\{x_i\}$ is a stationary time series which admits the
representation (\ref{eq:causal}). Assume that $\max_{1\leq j\leq
p}\E x_{ij}^4<C_1$, and
\begin{align}
& c_1<\min_{1\leq j\leq p}\sigma_{j,j}^{(n)}\leq \max_{1\leq j\leq
p}\sigma_{j,j}^{(n)}<c_2, \quad \max_{1\leq j\leq p}\sigma^2_j<c_3, \label{eq:h3}
\\& \max_{1\leq k\leq p}j\{\theta_{j,k,3}(x)\vee \theta_{j,k,3}(y)\}\leq
\ell_j~\text{with}~\sum_{j=1}^{+\infty}\ell_j<\infty, \label{eq:h4}
\end{align}
for some constants $C_1,c_3>0$ and $0<c_1<c_2$. Suppose that there exist
$N$ and $M$ such that $N\geq M$ and Assumption \ref{assum:tail} is
fulfilled. Then for $q\geq 2$, we have
\begin{align}\label{rhon1}
\rho_n\lesssim
n^{-1/8}M^{1/2}l_n^{7/8}+\gamma+(n^{1/8}M^{-1/2}l_n^{-3/8})^{q/(1+q)}\left(\sum^{p}_{j=1}\Theta_{M,j,q}^q\right)^{1/(1+q)},
\end{align}
where $\Theta_{i,j,q}=\Theta_{i,j,q}(x)\vee \Theta_{i,j,q}(y)$.
\end{theorem}
\noindent The approximation parameter $M$ will be chosen
appropriately to optimize the bound (\ref{rhon1}). The Gaussian
sequence $\{y_i\}$ can be constructed as a causal linear process
(e.g. based on the Wold representation theorem) to capture the
second order property of $\{x_i\}$.

We note that the conditions in Theorem \ref{thm:Gaussian-dep} can be
categorized into two types: tail restriction and weak dependence
assumption. Assumption \ref{assum:tail} and the condition that
$\max_{1\leq j\leq p}\E x_{ij}^4<C_1$ impose restrictions on the
tails of $\{x_{ij}\}_{j=1}^{p}$ uniformly across $j$, while
conditions (\ref{eq:h3})-(\ref{eq:h4}) essentially require weak
dependence uniformly across all the components of $\{x_i\}$. When
$\max_{1\leq j\leq p}\Theta_{M,j,q}=O(\rho^M)$ for $\rho<1$, we have
\begin{align}
\rho_n\lesssim
n^{-1/8}M^{1/2}l_n^{7/8}+\gamma+(n^{1/8}M^{-1/2}l_n^{-3/8})^{q/(1+q)}p^{1/(1+q)}\rho^{(qM)/(1+q)}.
\end{align}
Suppose $p\lesssim \exp(n^{b})$ for some $0\leq b<1/11,$ and
$\E(\max_{1\leq j\leq p}|x_{ij}|/\mathfrak{D}_n)^4\leq  1$. Then
by choosing $M\asymp N\lesssim n^{b'}$ with $4b'+7b<1$ and $1>b'>b$,
$\gamma\asymp n^{-(1-4b'-7b)/8}=o(1)$ and assuming that
$\mathfrak{D}_n\lesssim n^{(3-12b'-13b)/32}$, Condition (\rmnum{1})
in Assumption \ref{assum:tail} holds with $h(x)=x^4$, and
\begin{align*}
\rho_n\lesssim n^{-(1-4b'-7b)/8}.
\end{align*}
The same conclusion holds under condition (\rmnum{2}) in Assumption
\ref{assum:tail} provided that $p\lesssim \exp(n^b)$, $M\asymp
N\lesssim n^{b'}$, $\gamma\asymp n^{-(1-4b'-7b)/8}=o(1)$ and
$\mathfrak{D}_n\lesssim n^{(3-4b'-13b)/8}$ with $1>b'>b$.

Below we provide some empirical evidence for two conjectures
proposed in Remark~\ref{rem:interp}, in particular the interplay
between dependence and dimensionality. To this end, we generate
$\{x_i\}$ from a multivariate ARCH model
$x_i=\Sigma^{1/2}_i\epsilon_i$, where
$\epsilon_i=(\epsilon_{i1},\dots,\epsilon_{ip})'$ with
$\sqrt{2}\epsilon_{ij}$ being a sequence of i.i.d $t(4)$ random
variables, and $\Sigma_i=(1-\beta_0)D_p+\beta_0x_{i-1}x_{i-1}'$ with
$\Sigma^{1/2}_i$ being a lower triangular matrix based on the
Cholesky decomposition of $\Sigma_i$. Here $D_p=(d_{ij})_{i,j=1}^p$
with $d_{jj}=1$ and $d_{ij}=0.5$ for $i\neq j.$ Notice that
$\{x_i\}$ are uncorrelated and $\text{cov}(x_i)=D_p$. To capture the
second order property of $\{x_i\}$, we generate independent Gaussian
vectors $\{y_i\}$ from $N(0,D_p)$. Figure \ref{fig:interplay}
illustrates the interplay between dependence and dimensionality
using the P-P plots for $n=60$, $p=100,300,500$, and
$\beta_0=0,0.2,0.5$. For moderate $p$ and $\beta_0$, the Gaussian
approximation is reasonably good, which is consistent with our
theory. Moreover, we also observe the following phenomena. On one
hand, as $p$ increases, the approximation deteriorates for the same
$\beta_0$ which controls the strength of dependence; on the other
hand, for fixed $p$, the approximation becomes worse in the right
tail which is most relevant for practical applications, as $\beta_0$
increases. Note that our theoretical results are finite sample
valid, and thus the sample size supposed not to play any role here.
Hence, we believe that the less dependent of the data vectors, the
faster diverging rate of the dimension is allowed for obtaining an
accurate Gaussian approximation.

\begin{figure}[h]
\centering
\includegraphics[height=5.2cm,width=5.2cm]{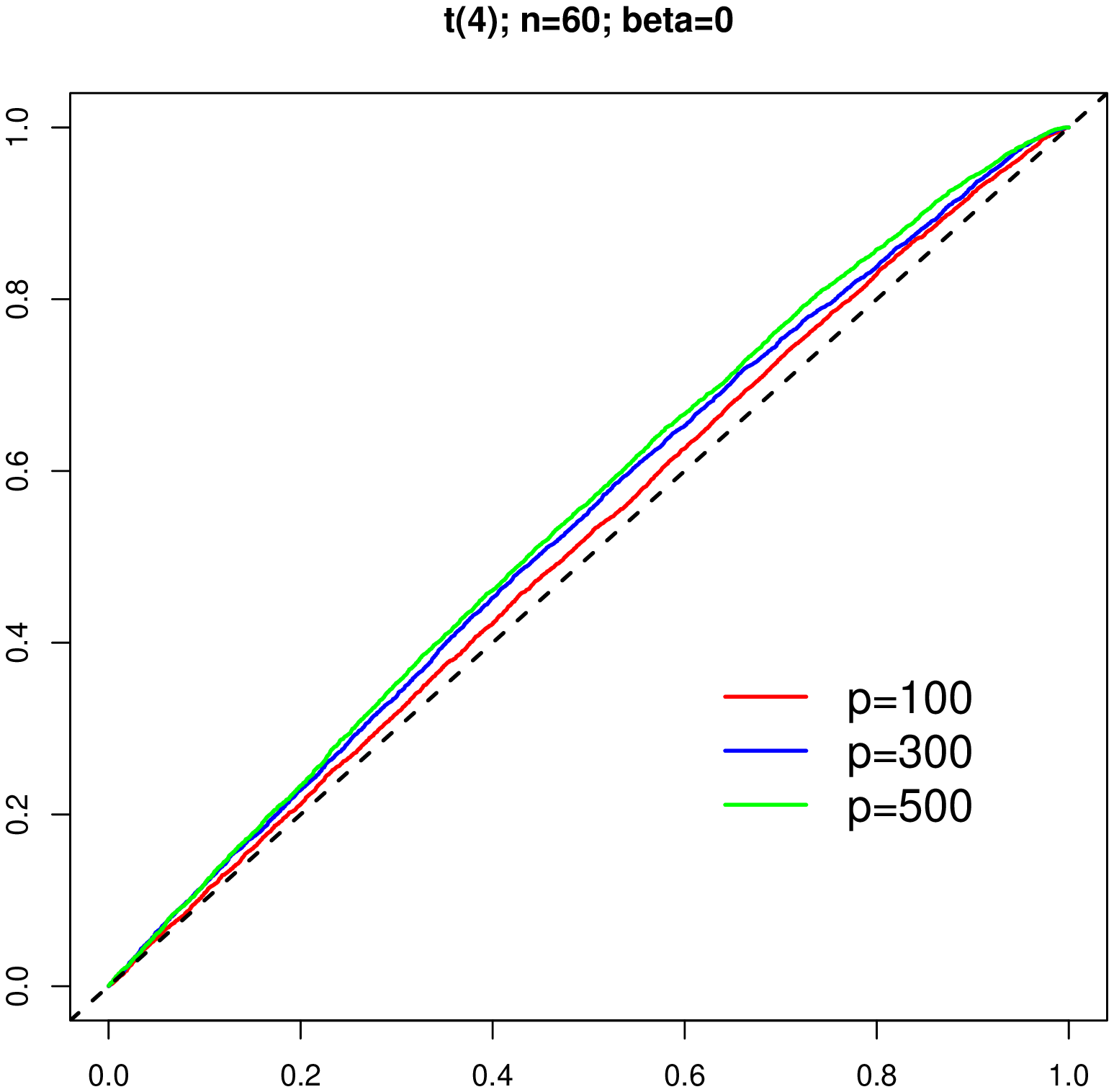}
\includegraphics[height=5.2cm,width=5.2cm]{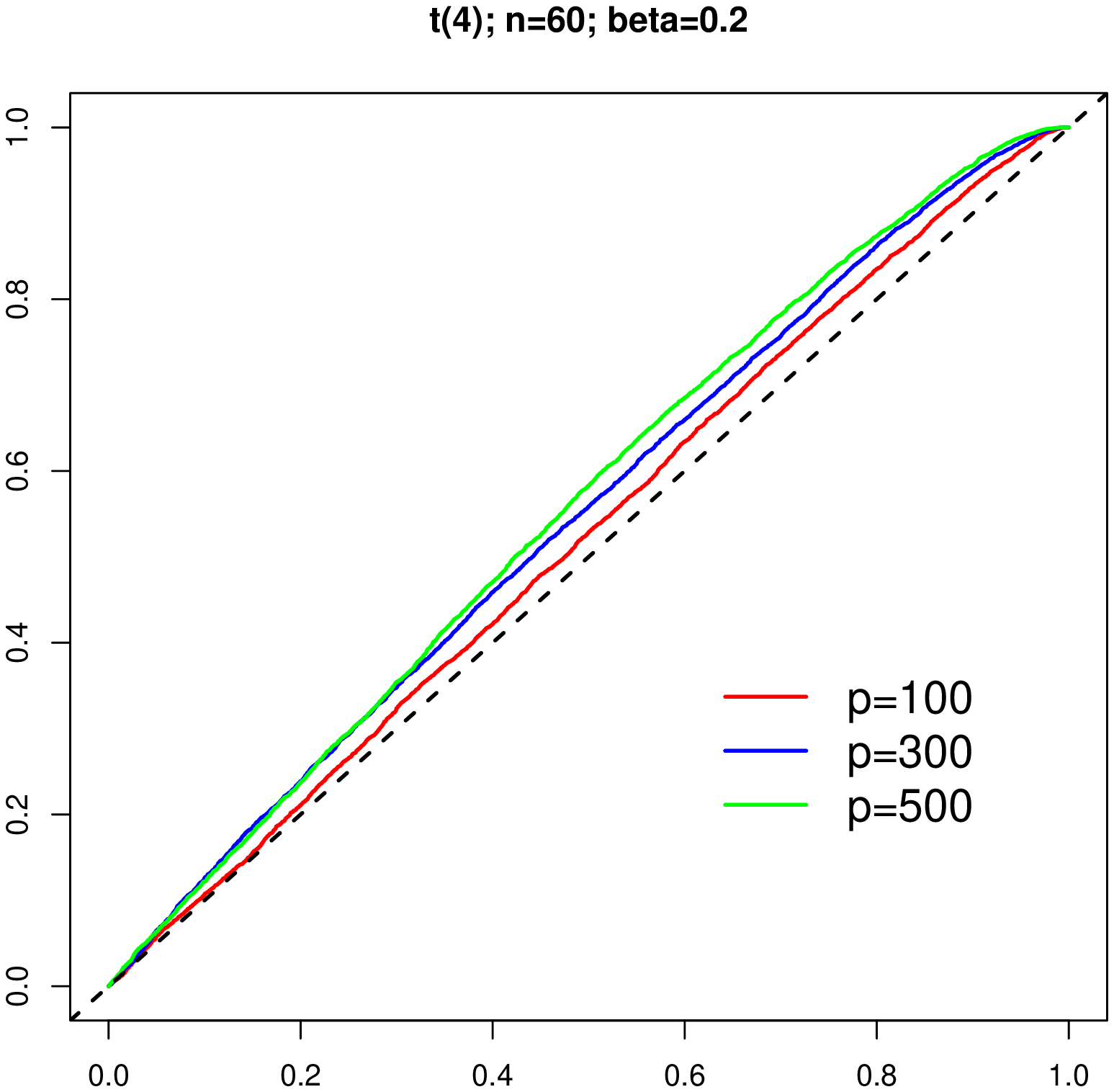}
\includegraphics[height=5.2cm,width=5.2cm]{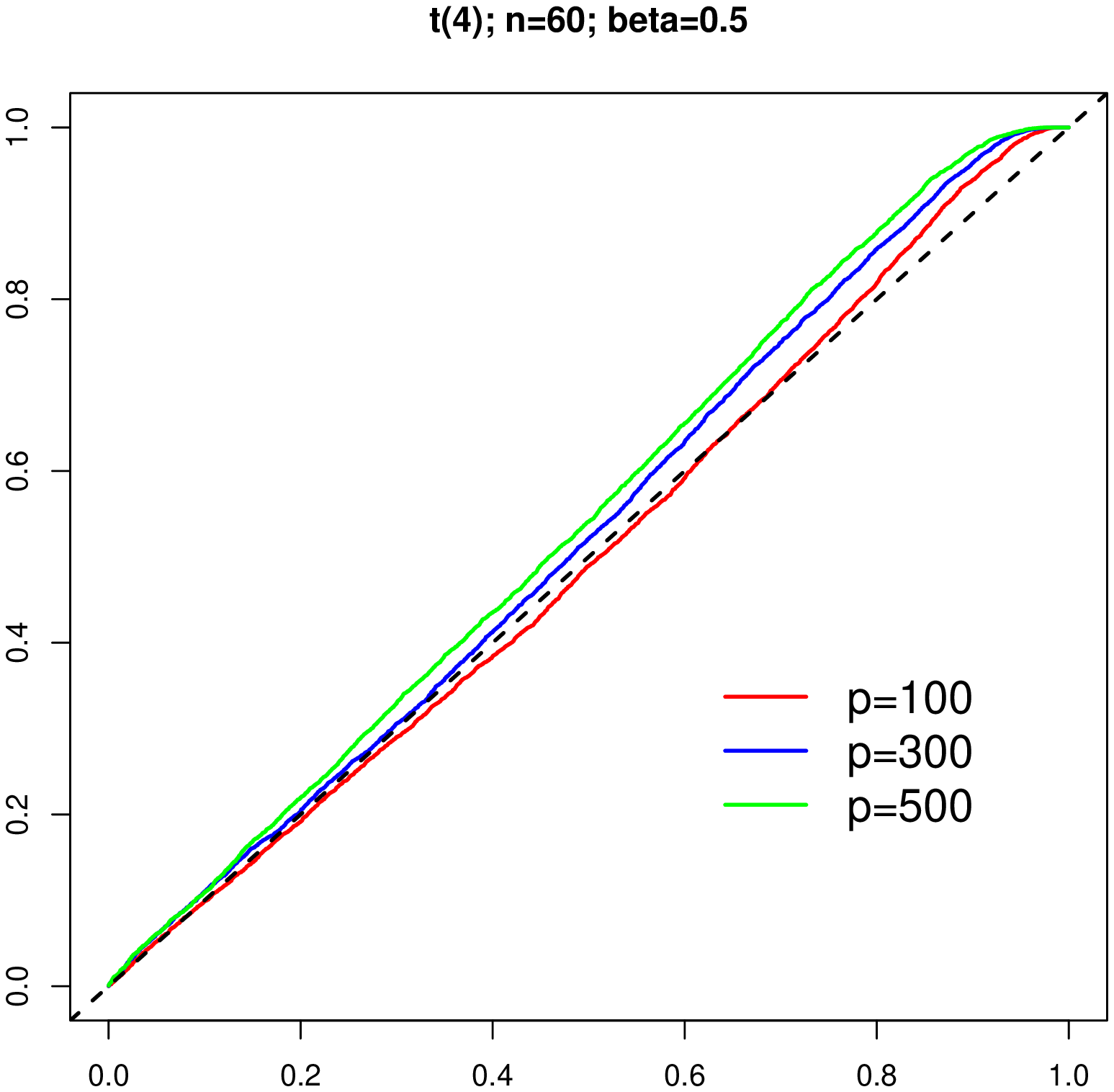}
\caption{Interplay between dependence and dimensionality: P-P plots
comparing distributions of $T_X$ and $T_Y$.}\label{fig:interplay}
\end{figure}

In the end, we discuss an intriguing question: is there any
so-called ``dimension free dependence structure''? In other words,
what kind of dependence assumption will not affect the dimension
increase rate (as compared to the independence case in
\cite{cck13})? To address this question, we consider one
possibility: the original $p$-dimensional vector can be decomposed
into two components namely one times series component and one
independence component, where the former component is asymptotically
ignorable comparing to the latter as $n$ grows. Our
contribution here is to precisely characterize such a ``dimension
free" dependence structure.

\begin{proposition}\label{prop:dim-free}
Consider a $p$-dimensional time series $\{x_i\}$. Suppose there
exists a permutation $\pi(\cdot)$ such that
$(x_{i\pi(1)},\dots,x_{i\pi(p)})=z_i'=(z_{i1}',z_{i2}')'$, where
$\{z_{i1}\}$ is a $q$-dimensional (possibly nonstationary) time
series and $\{z_{i2}\}$ is a $p-q$ dimensional sequence of
independent variables. Suppose $\{z_{i1}\}$ and $\{z_{i2}\}$ are
independent. When $\{z_{i2}\}$ satisfies the assumptions in
Corollary 2.1 of \cite{cck13}, we have
\begin{align}
\sup_{z\in\mathbb{R}}\left|P\left(\max_{q+1\leq j\leq
p}X_{\pi(j)}\leq z\right)-P\left(\max_{q+1\leq j\leq
p}Y_{\pi(j)}\leq z\right)\right|\lesssim n^{-c},\quad c>0.
\end{align}
Recall that $X_{\pi(j)}$ is
$\sum_{i=1}^{n}x_{i\pi(j)}/\sqrt{n}$ and $Y_{\pi(j)}$ is defined in
a similar manner. Then under the additional assumption that
\begin{eqnarray}\label{ass:add}
qn^{-c}+q/\E \max_{q+1\leq j\leq p}Y_j=O(n^{-c'}),\quad c'>0,
\end{eqnarray}
and $\max_{1\leq j\leq p}\E|X_j|^2<+\infty$, we have
\begin{align}\label{eq:dim-free2}
\sup_{z\in\mathbb{R}}\left|P\left(\max_{1\leq i\leq p}X_i\leq
z\right)-P\left(\max_{1\leq i\leq p}Y_i\leq z\right)\right|\lesssim
n^{-c''},\quad c''>0.
\end{align}
\end{proposition}
\noindent The additional assumption (\ref{ass:add}) implies that $q$
is of a polynomial order w.r.t. $n$ while $(p-q)$ achieves the
exponential order as specified in Corollary 2.1 of \cite{cck13}.
Therefore, the largest possible diverging rate of $p$ allowed in
Proposition~\ref{prop:dim-free} remains the same as that in the
independence case (\cite{cck13}). The independence assumption
between $\{z_{i1}\}$ and $\{z_{i2}\}$ might be relaxed. Here, we
assume it mainly for technical simplicity so that only one single
dependence assumption $\max_{1\leq j\leq q}\E|X_{\pi(j)}|^2<\infty$
needs to be imposed on $\{z_{i1}\}$.

\section{Bootstrap Inference}\label{sec:boot}
In practice, the intrinsic dependence structure of time series data
is usually unknown. Hence, the Gaussian approximation theory becomes
too restrictive to use. However, this general theory provides a
foundation in developing the bootstrap inference theory that do not
require such knowledge. In this section, we consider two types of
bootstrap procedures: (i) blockwise multiplier bootstrap; and (ii)
non-overlapping block bootstrap. The former is employed in Section~\ref{sec:statappl}, while the latter is a more flexible alternative.

\subsection{Blockwise multiplier bootstrap}\label{subsec:bmb}
To approximate the quantiles of $T_X$, we introduce a blockwise
multiplier bootstrap procedure for $M$-dependent and weakly
dependent time series considered in Sections \ref{subsec:m-dep} and
\ref{subsec:weak-dep}. Suppose $n=(N+M)r$, where $N\geq M$ and
$N,M,r\rightarrow +\infty$ as $n\rightarrow+\infty$. Let
$\{(e_{i},\widetilde{e}_i)\}$ be a sequence of i.i.d $N(0,I_2)$
variables that are independent of $\{x_i\}$. Define
\begin{equation}
T_D=\max_{1\leq j\leq p}\frac{1}{\sqrt{n}}\sum^{r}_{i=1}D_{ij},\quad
D_{ij}=A_{ij}e_i+B_{ij}\widetilde{e}_i.
\end{equation}
Recall the definitions of $A_{ij}$ and $B_{ij}$ in (\ref{eq:AB}).
Conditional on $\{x_i\}$, $D_{ij}$ are mean-zero Gaussian random
variables such that
\begin{equation}
\text{cov}(D_{ij},D_{i'k})=\delta_{ii'}(A_{ij}A_{i'k}+B_{ij}B_{i'k}),\quad
\delta_{ii'}=\mathbf{1}\{i=i'\}.
\end{equation}
Thus we have
\begin{equation}
\text{cov}\left(\sum^{r}_{i=1}D_{ij}/\sqrt{n},\sum^{r}_{i=1}D_{ik}/\sqrt{n}\right)=\frac{1}{n}\sum^{r}_{i=1}(A_{ij}A_{ik}+B_{ij}B_{ik}).
\end{equation}
Conditional on the sample $\{x_i\}^{n}_{i=1}$, define the
$\alpha$-quantile of $T_D$ as
\begin{equation}
c_{T_D}(\alpha):=\inf\{t\in\mathbb{R}:P(T_D\leq
t|\{x_i\}^{n}_{i=1})\geq \alpha\}.
\end{equation}

Our goal below is to quantify
\begin{equation}
\widetilde{\rho}_n:=\sup_{\alpha\in(0,1)}|P(T_X\leq
c_{T_D}(\alpha))-\alpha|.
\end{equation}
To this end, consider the estimation errors
\begin{equation}
\begin{split}
&E_{A}:=\max_{1\leq j,k\leq
p}\left|\frac{1}{r}\sum^{r}_{i=1}A_{ij}A_{ik}/N-\sigma_{j,k}^{(n)}\right|,\quad
\\&E_{B}:=\max_{1\leq j,k\leq
p}\left|\frac{1}{r}\sum^{r}_{i=1}B_{ij}B_{ik}/M-\sigma_{j,k}^{(n)}\right|,\\
&E_{AB}:=\max_{1\leq j,k\leq
p}\left|\frac{1}{n}\sum^{r}_{i=1}(A_{ij}A_{ik}+B_{ij}B_{ik})-\sigma_{j,k}^{(n)}\right|,
\end{split}
\end{equation}
where
$\sigma_{j,k}^{(n)}=\frac{1}{n}\sum^{n-1}_{l=1-n}(n-|l|)\gamma_{x,jk}(l)$.
Recall that $h(\cdot)$ is a nondecreasing convex function with
$h(0)=0$. Define the Orlicz norm as
\begin{align*}
||X||_{h}=\inf\left\{B>0:\E h\left(\frac{|X|}{B}\right)\leq
1\right\}.
\end{align*}

We first consider $M$-dependent stationary sequence where $M$ is
allowed to grow with the sample size $n$. Define the following
quantities which characterize the higher order properties of the
time series (e.g., $\bar{\sigma}^2_{x,N}$ and $\varsigma_{x,N}$
below characterize the fourth order property of $\{x_i\}$),
\begin{align*}
&\bar{\sigma}^2_{x,N}=\max_{1\leq j\leq p}\left\{\frac{1}{N}\sum_{i_1,i_2,i_3=-\infty}^{+\infty}|\text{cum}(x_{i_1j},x_{i_2j},x_{i_3j},x_{0j})|+\sigma_{j}^4\right\},\\
&\varsigma_{x,N}=\left(\E \max_{1\leq j\leq
p}\left|\sum^{N}_{i=1}x_{ij}/\sqrt{N}\right|^4\right)^{1/4},\\
&\zeta_{x,h,N}=\max_{1\leq j\leq
p}\left|\left|\sum^{N}_{i=1}x_{ij}/\sqrt{N}\right|\right|_{h},\quad
\varpi_{x}=\max_{1\leq j,k\leq p}\sum^{+\infty}_{l=-\infty}|l||\E
x_{i,j}x_{i+l,k}|,
\end{align*}
where $\text{\cum}$ denotes the cumulant (see e.g.
\cite{brillinger}) and
$\sigma_j^2=\sum^{+\infty}_{l=-\infty}|\gamma_{x,jj}(l)|$.

The following lemma plays an important role in the subsequent
derivations.
\begin{lemma}\label{lemma:m-dep-boot}
Suppose $\{x_i\}$ is a $M$-dependent stationary sequence. Then with
$h(x)=\exp(x)-1$,
\begin{align*}
&\E E_{A}\lesssim  \bar{\sigma}_{x,N}\sqrt{\log p/r}+\log
p\{\log(rp)\}^2\zeta_{x,h,N}^2/r+\varpi_x/N,
\\&\E E_{B}\lesssim \bar{\sigma}_{x,M}\sqrt{\log
p/r}+\log p\{\log(rp)\}^2\zeta_{x,h,M}^2/r+\varpi_x/M.
\end{align*}
Alternatively, we have
\begin{align*}
&\E E_{A}\lesssim  \bar{\sigma}_{x,N}\sqrt{\log p/r}+\log
p\varsigma_{x,N}^2/\sqrt{r}+\varpi_x/N,\\
&\E E_{B}\lesssim \bar{\sigma}_{x,M}\sqrt{\log p/r}+\log
p\varsigma_{x,M}^{2}/\sqrt{r}+\varpi_x/M.
\end{align*}
\end{lemma}

Let $c_{T_Y}(\alpha)=\inf\{t\in\mathbb{R}:P(T_Y\leq t)\geq
\alpha\}.$ In the spirit of Lemma 3.2 in \cite{cck13}, we can show
that when $c_1<\min_{1\leq j\leq p}\sigma_{j,j}^{(n)}\leq \max_{1\leq
j\leq p}\sigma_{j,j}^{(n)}<c_2$ for some $0<c_1<c_2,$
\begin{align*}
&P(c_{T_D}(\alpha)\leq c_{T_Y}(\alpha+\pi(\nu)))\geq 1-P(E_{AB}>\nu),\\
&P(c_{T_Y}(\alpha)\leq c_{T_D}(\alpha+\pi(\nu)))\geq
1-P(E_{AB}>\nu),
\end{align*}
where $\pi(\nu)=C\nu^{1/3}(1\vee \log(p/\nu))^{2/3}$ for some
constant $C>0$ depending on $c_1,c_2$. Using the arguments in
Theorem 3.1 of \cite{cck13}, it is not hard to show that
\begin{equation}\label{assum:boot-m-dep}
\sup_{\alpha\in(0,1)}\left|P(T_X\leq
c_{T_D}(\alpha))-\alpha\right|\lesssim
\rho_n+\pi(\nu)+P(E_{AB}>\nu).
\end{equation}
Because $\E E_{AB}\leq \E E_A + \E E_B$, we deduce that
\begin{align}\label{eq:boot}
\widetilde\rho_n:= \sup_{\alpha\in(0,1)}\left|P(T_X\leq
c_{T_D}(\alpha))-\alpha\right|\lesssim \rho_n+\nu^{1/3}(1\vee
\log(p/\nu))^{2/3}+\E E_{A}/\nu+\E E_{B}/\nu.
\end{align}

\begin{assumption}\label{assum:boot-m-dep}
Suppose $p\lesssim \exp(n^{b})$ with $0\leq b<1/15.$ Set $M\lesssim n^{b'}$
and $N\lesssim n^{b''}$ with $1> b''\geq b'$, $4b'+7b<1$ and $b'>2b$.
Assume that $\mathfrak{D}_n\lesssim n^{(3-12b'-13b)/32}$ under
Condition (\rmnum{1}) in Assumption \ref{assum:tail} with $h(x)=x^4$
or $\mathfrak{D}_n\lesssim n^{(3-4b'-13b)/8}$ under Condition
(\rmnum{2}) in Assumption \ref{assum:tail}. Further assume that one
of the following two conditions holds.
\\ \textbf{Condition 1:} $\bar{\sigma}_{x,M}\vee\bar{\sigma}_{x,N}\lesssim n^{s_1},
\zeta_{x,h,M}\vee \zeta_{x,h,N}\lesssim n^{s_2/2}, \varpi_x\lesssim
n^{s_3}$, where $h(x)=\exp(x)-1$ and $s_1,s_2,s_3$ satisfy that
\begin{align*}
s_b:=((1-5b-b'')/2-s_1)\wedge(1-5b-b''-s_2)\wedge(b'-2b-s_3)>0.
\end{align*}
\textbf{Condition 2:}
$\bar{\sigma}_{x,M}\vee\bar{\sigma}_{x,N}\lesssim
n^{s_1},\varsigma_{x,M}\vee \varsigma_{x,N}\lesssim n^{s_2'/2},
\varpi_x\lesssim n^{s_3}$, and $s_1,s_2,s_3$ satisfy that
\begin{align*}
s_b':=((1-5b-b'')/2-s_1)\wedge((1-6b-b'')/2-s_2')\wedge(b'-2b-s_3)>0.
\end{align*}
\end{assumption}

We are now in position to present the first main result in this
section.
\begin{theorem}\label{thm:m-dep-boot}
Consider a $M$-dependent stationary time series $\{x_i\}$. Under the
assumptions in Theorem \ref{thm2} and Assumption
\ref{assum:boot-m-dep},
\begin{equation}
\sup_{\alpha\in(0,1)}\left|P(T_X\leq
c_{T_D}(\alpha))-\alpha\right|\lesssim
\begin{cases}
n^{-c},\quad &c=\min\{s_b/4,(1-4b'-7b)/8\},
\\ &\text{under Condition 1},
\\n^{-c'},\quad & c'=\min\{s_b'/4,(1-4b'-7b)/8\},
\\ & \text{under Condition
2}.
\end{cases}
\end{equation}
\end{theorem}

Our next theorem extends the above result to weakly dependent
stationary time series.
\begin{theorem}\label{thm:dep-boot}
Consider a weakly dependent stationary time series $\{x_i\}$.
Suppose $\max_{1\leq j\leq p}\Theta_{M,j,q}=O(\rho^M)$ for $\rho<1$
and some $q\geq 4$. Then under the assumptions in Theorem
\ref{thm:Gaussian-dep} and Assumption \ref{assum:boot-m-dep},
\begin{equation}
\sup_{\alpha\in(0,1)}\left|P(T_X\leq
c_{T_D}(\alpha))-\alpha\right|\lesssim
\begin{cases}
n^{-c},\quad & c=\min\{s_b/4,(1-4b'-7b)/8\},
\\ & \text{under Condition 1},
\\ n^{-c'},\quad & c'=\min\{s_b'/4,(1-4b'-7b)/8\},
\\ & \text{under Condition
2}.
\end{cases}
\end{equation}
\end{theorem}
\noindent Remark that the results of Theorems~\ref{thm:m-dep-boot}
and ~\ref{thm:dep-boot} are still valid even when $p$ is fixed or $p$ grows slower
than the exponential rate required in Assumption~\ref{assum:boot-m-dep}.

\begin{remark}
{\rm When $\{x_i\}$ has the so-called geometric moment contraction
(GMC) property (uniformly across its components), we have
$\bar{\sigma}_{x,M}\vee\bar{\sigma}_{x,N}\lesssim 1$ (i.e., $s_1=0$)
by Proposition 2 of \cite{wushao2004} and the assumption that
$\max_j\sigma_{j}<\infty.$ }
\end{remark}

\begin{remark}
{\rm It is known that in the low dimensional setting, the tapered
block bootstrap method yields an improvement over the block
bootstrap in terms of the bias for variance estimation, and thus
provides a better MSE rate; see \cite{pp01}. Hence, we may also
want to combine the blockwise multiplier bootstrap method proposed
here with the data tapering scheme. For example, let $\mathcal{K}$:
$\mathbb{R}\rightarrow\mathbb{R}$ be a data taper with
$\mathcal{K}(x)=0$ for $x\notin [0,1)$. One can consider the
following modification,
\begin{align*}
&T_{\mathcal{K},D}=\max_{1\leq j\leq
p}\frac{1}{\sqrt{n}}\sum^{r}_{i=1}D_{\mathcal{K},ij}, \quad
D_{\mathcal{K},ij}=A_{\mathcal{K},ij}e_i+B_{\mathcal{K},ij}\widetilde{e}_i,
\\&
A_{\mathcal{K},ij}=\sum^{iN+(i-1)M}_{l=(i-1)(N+M)+1}\mathcal{K}\left(\frac{l-(i-1)(N+M)}{N}\right)x_{lj},
\\&
B_{\mathcal{K},ij}=\sum^{i(N+M)}_{l=iN+(i-1)M+1}\mathcal{K}\left(\frac{l-iN+(i-1)M}{M}\right)x_{lj}.
\end{align*}
More detailed investigation along this direction is left for future
study. }
\end{remark}

\subsection{Non-Overlapping Block bootstrap}\label{subsec:block}
In this subsection, we propose an alternative bootstrap procedure in
the high dimensional setting: non-overlapping block bootstrap
(\cite{Carl1986}). In general, this bootstrap procedure may avoid
estimating the influence function (defined in Section~\ref{sec:ts})
in contrast with blockwise multiplier bootstrap. We provide
theoretical justifications for this procedure through establishing
its equivalence with multiplier bootstrap; see
(\ref{eq:equi}).

Assume for simplicity that $n=b_nl_n$, where $b_n,l_n\in\mathbb{Z}$.
Conditional on the sample $\{x_i\}_{i=1}^{n}$, we let
$\varrho_1,\dots,\varrho_{l_n}$ be i.i.d uniform random variables on
$\{0,\dots,l_n-1\}$ and define $x_{(j-1)b_n+i}^*=x_{\varrho_jb_n+i}$
with $1\leq j\leq l_n$ and $1\leq i\leq b_n.$ In other words,
$\{x_i^*\}_{i=1}^{n}$ is a non-overlapping block bootstrap sample
with block size $b_n$. Define
\begin{equation*}
T_{X^*}=\max_{1\leq j\leq
p}\frac{1}{\sqrt{n}}\sum^{n}_{i=1}(x_{ij}^*-\bar{x}_{nj})=\max_{1\leq
j\leq
p}\frac{1}{\sqrt{n}}\sum^{l_n}_{i=1}(\mathcal{A}_{ij}^*-\bar{\mathcal{A}}_{nj}),
\end{equation*}
where $\bar{x}_{nj}=\sum^{n}_{i=1}x_{ij}/n$,
$\bar{\mathcal{A}}_{nj}=\sum^{l_n}_{i=1}\mathcal{A}_{ij}/l_n$, and
$\mathcal{A}_{ij}=\sum^{ib_n}_{l=(i-1)b_n+1}x_{lj},$ and
$\mathcal{A}_{1j}^*,\dots,\mathcal{A}_{l_nj}^*$ are i.i.d draws from
the empirical distribution of
$\mathcal{A}_{1j},\dots,\mathcal{A}_{l_nj}$. Also define
$$T_{\widetilde{X}}=\max_{1\leq j\leq
p}\frac{1}{\sqrt{n}}\sum^{l_n}_{i=1}\mathcal{A}_{ij}e_i,$$ where
$\{e_i\}^{l_n}_{i=1}$ is a sequence of i.i.d $N(0,1)$. Throughout the following discussions, we suppose that
$b_n\geq M.$ The theoretical validity of the multiplier bootstrap
based on $T_{\widetilde{X}}$ can be justified using similar
arguments in the previous subsection because the same arguments go
through when $A_{ij}$ and $B_{ij}$ are replaced by
$\mathcal{A}_{ij}$ (provided that $b_n\geq M$). By showing that with
probability $1-Cn^{-c}$,
\begin{equation}\label{eq:equi}
\sup_{t\in\mathbb{R}}|P(T_{X^*}\leq
t|\{x_i\}^{n}_{i=1})-P(T_{\widetilde{X}}\leq
t|\{x_i\}^{n}_{i=1})|\lesssim n^{-c'},\quad c'>0.
\end{equation}
we establish the validity of non-overlapping block bootstrap in
Theorem \ref{thm:block-boot}.

\begin{assumption}\label{assum:block-boot}
Assume that $\bar{\sigma}_{x,b_n}\sqrt{\log p/l_n}\lesssim n^{-c_0}$
and $\zeta_{x,h,b_n}^2\{\log(pl_n)\}^9/l_n\lesssim n^{-c_0'}$ with
$h(x)=\exp(x)-1$, where $c_0,c_0'>0.$
\end{assumption}

\begin{theorem}\label{thm:block-boot}
Suppose that $c_1<\min_{1\leq j\leq p}\sigma_{j,j}^{(b_n)}\leq
\max_{1\leq j\leq p}\sigma_{j,j}^{(b_n)}<c_2$ and $\max_{1\leq j\leq
p}\sigma_j^2<c_3$ for some constants $0<c_1<c_2<\infty$ and $c_3>0$,
where
$\sigma_{j,j}^{(b_n)}=\sum^{b_n-1}_{l=1-b_n}(b_n-|l|)\gamma_{x,jj}(l)/b_n$.
Further assume that the assumptions in
Theorem \ref{thm:m-dep-boot} or Theorem \ref{thm:dep-boot} hold with
$M=N=b_n$ and $r=n/(2b_n)$. Then (\ref{eq:equi}) holds with
probability $1-Cn^{-c}$ for some $c,C>0$. Moreover, we have
\begin{equation}
\sup_{\alpha\in(0,1)}\left|P(T_X\leq
c_{T_{X}^*}(\alpha))-\alpha\right|\lesssim n^{-c''},
\end{equation}
where $c_{T_{X^*}}(\alpha)=\inf\{t\in\mathbb{R}:P(T_{X^*}\leq
t|\{x_i\}^{n}_{i=1})\geq \alpha\}$ and $c''>0.$
\end{theorem}

\section{General Inferential Theory}\label{sec:ts}
In this section, we establish a general framework of conducting
bootstrap inference for high dimensional time series based on the
theoretical results in Section~\ref{sec:boot}. This general
framework assumes that the $q$-dimensional quantity of interest,
denoted as $\Theta_{0q}$, admits an approximately linear expansion,
and thus covers three examples considered in
Section~\ref{sec:statappl}. In particular, $\Theta_{0q}$ is
expressed as a functional of the distribution of a $p$-dimensional
\emph{weakly dependent} stationary time series
$\{u_i\}.$\footnote{Note that $p$ here is different from the
dimension of $x_i$ discussed in previous sections.}
Motivated by the testing on spectral properties, we
further extend the results in Section~\ref{subsec:als} to an
infinite dimensional parameter case in
Section~\ref{subsec:infinite}.

\subsection{Approximately linear statistics}\label{subsec:als}
In this subsection, we consider the quantities that can be expressed as
functionals of the marginal distribution of a block time series with
length $d_0$: $\{v_i\}_{i=1}^{N_0}$, where
$v_i:=(u_i,\dots,u_{i+d_0-1})'$ and $N_0=n-d_0+1$. Here, we allow
the integer $d_0$ to grow with $n$. Define
$F_{d_0}^{N_0}=\sum^{N_0}_{i=1}\delta_{v_{i}}/N_0$ as the empirical
distribution for $\{v_i\}_{i=1}^{N_0}$. The distribution function of
$v_1$ is denoted as $F_{d_0}$. We are interested in testing the
parameter
$\Theta_{q_0}=(\theta_{1},\dots,\theta_{q_0})':=\mathcal{T}(F_{d_0})$
for some functional $\mathcal{T}:=\mathcal{T}_{q_0,d_0}$. The
parameter dimension $q_0$ depends on either $p$ or $d_0$, e.g.,
$q_0=p, p^2$ or $d_0p^2$. A natural estimator for $\Theta_{q_0}$ is
then given by
$\widehat{\Theta}_{q_0}=(\widehat{\theta}_{1},\dots,\widehat{\theta}_{q_0})':=\mathcal{T}(F_{d_0}^{N_0})$.

Assume $
\widehat{\Theta}_{q_0}$ admits the following approximately linear
expansion in a neighborhood of $F_{d_0}$:
\begin{equation}\label{eq:lin}
\widehat{\Theta}_{q_0}=\Theta_{q_0}+\frac{1}{N_0}\sum^{N_0}_{i=1}IF(v_i,F_{d_0})+\mathcal{R}_{N_0},
\end{equation}
where
$IF(v_i,F_{d_0})=(IF_1(v_i,F_{d_0}),\dots,IF_{q_0}(v_i,F_{d_0}))'$
is called ``influence function" (see e.g. \cite{hrrs86}) and
$\mathcal{R}_{N_0}:=\mathcal{R}_{N_0}(v_1,\dots,v_{N_0})=(\mathcal{R}_{1N_0},\dots,\mathcal{R}_{q_0N_0})'$
is a remainder term. Examples of approximately linear statistics
include various location and scale estimators for the marginal
distribution of $\{u_i\}$, von Mises statistics and $M$-estimators
of time series models (see \cite{kun89}).

We are interested in testing the null hypothesis $H_0:
\Theta_{q_0}=\widetilde{\Theta}_{q_0}$ versus the alternative $H_a:
\Theta_{q_0}\neq \widetilde{\Theta}_{q_0}$, where
$\widetilde{\Theta}_{q_0}=(\widetilde{\theta}_{1},\dots,\widetilde{\theta}_{q_0})'$.
The test is proposed as
\begin{equation}\label{eq:test}
\phi(\widehat{\Theta}_{q_0};c(\alpha))=
\begin{cases}
1,\quad &\max_{1\leq j\leq q_0}\sqrt{N_0}|\widehat{\theta}_{j}-\widetilde{\theta}_j|\geq c(\alpha), \\
0,\quad &\text{otherwise}. \end{cases}
\end{equation}
We next apply the bootstrap theory in Section~\ref{sec:boot} to
obtain the critical value $c(\alpha)$. Specifically, we define
$x_i=(IF(v_i,F_{d_0})',-IF(v_i,F_{d_0})')'$ and
$\widehat{x}_i=(\widehat{IF}(v_i,F_{d_0}^{N_0})',-\widehat{IF}(v_i,F_{d_0}^{N_0})')'$,
where $\widehat{IF}(v_i,F_{d_0}^{N_0})$ is some estimate of
$IF(v_i,F_{d_0})$. Suppose $N_0=(N_1+M_1)r_1$, where $N_1\geq M_1$
and $N_1,M_1,r_1\rightarrow +\infty$ as $N_0\rightarrow+\infty.$
Define the estimated block sums
\begin{equation}\label{eq:app-AB}
\widehat{A}_{ij}=\sum^{iN_1+(i-1)M_1}_{l=(i-1)(N_1+M_1)+1}\widehat{x}_{lj},\quad
\widehat{B}_{ij}=\sum^{i(N_1+M_1)}_{l=iN_1+(i-1)M_1+1}\widehat{x}_{lj},
\end{equation}
where $1\leq i\leq r_1$ and $1\leq j\leq 2q_0$. Let
\begin{align*}
T_{\widehat{D}}=\max_{1\leq j\leq
2q_0}\frac{1}{\sqrt{n}}\sum^{r_1}_{i=1}\widehat{D}_{ij},
\end{align*}
where
$\widehat{D}_{ij}=\widehat{A}_{ij}e_i+\widehat{B}_{ij}\widetilde{e}_i$
with $\{(e_{i},\widetilde{e}_i)\}$ being a sequence of i.i.d
$N(0,I_2)$ independent of $\{u_i\}$. The
bootstrap critical value is given by
\begin{equation}
c_{1}(\alpha):=\inf\{t\in\mathbb{R}:P(T_{\widehat{D}}\leq
t|\{x_i\}^{n}_{i=1})\geq 1-\alpha\}.
\end{equation}

We next justify the validity of the test in (\ref{eq:test}) with
$c(\alpha)=c_1(\alpha)$ in Theorems~\ref{thm:app-1}
and~\ref{thm:bandtest}.
\begin{assumption}\label{assum:app-1}
Assume that $P(\max_{1\leq j\leq
q_0}\sqrt{N_0}|\mathcal{R}_{jN_0}|>C_1n^{-c_1}/\sqrt{\log(2q_0)})<C_1n^{-c_1}$
and $P(\mathcal{E}_{AB}\{\log(2q_0)\}^2>C_2n^{-c_2})\leq
C_2n^{-c_2}$, where $c_1,C_1,c_2,C_2>0$, and
\begin{align*}
\mathcal{E}_{AB}=\max_{1\leq j\leq
2q_0}\left|\frac{1}{n}\sum^{r}_{i=1}\{(A_{ij}-\widehat{A}_{ij})^2+(B_{ij}-\widehat{B}_{ij})^2\}\right|,
\end{align*}
with $A_{ij}=\sum^{iN_1+(i-1)M_1}_{l=(i-)(N_1+M_1)+1}x_{lj}$ and
$B_{ij}=\sum^{i(N_1+M_1)}_{l=iN_1+(i-1)M_1+1}x_{lj}$.
\end{assumption}

\begin{theorem}\label{thm:app-1}
Suppose the assumptions in Theorem \ref{thm:m-dep-boot} or Theorem
\ref{thm:dep-boot} hold for $\{x_i\}$, where $p$ is replaced by
$2q_0$. Then under Assumption \ref{assum:app-1} and $H_0$, we have
\begin{equation}
\sup_{\alpha\in (0,1)}\left|P\left(\max_{1\leq j\leq
q_0}\sqrt{N_0}|\widehat{\theta}_{j}-\widetilde{\theta}_j|\geq
c_{1}(\alpha)\right)-\alpha\right|\lesssim n^{-c},\quad c>0.
\end{equation}
\end{theorem}
Theorem~\ref{thm:app-1} applies directly to the methods described in
Sections~\ref{sec:ucb}-\ref{subsec:cov} for both M-dependent and
weakly dependent stationary time series. For example, consider the
white noise testing problem in Section \ref{subsec:cov}. Suppose $\E
u_i=0$. In this example,
$\Theta_{q_0}=(\text{vec}(\gamma_u(1))',\dots,\text{vec}(\gamma_u(L))')'$
with $\gamma_u(h)=\E u_{i}u_{i+h}'$ and $q_0=Lp^2$. Then we have
$IF(v_i,F_{d_0})=\nu_i-\Theta_{q_0}$ and
$\widehat{IF}(v_i,F_{d_0}^{N_0})=\nu_i-\sum^{N_0}_{i=1}\nu_i/n$ with
$\nu_i=(\text{vec}(u_{i}u_{i+1}')',\dots,\text{vec}(u_{i}u_{i+L}')')'$
and $N_0=n-L$. Note that the bootstrap procedures considered in
Section \ref{sec:statappl} are in fact simplified versions of the
blockwise multiplier bootstrap in Section \ref{sec:boot} with
$N=M=b_n$ and $r=l_n/2$.

Our next theorem covers the problem of testing the bandedness of
covariance matrix in Section \ref{sec:bandtest}. Recall that $$T_{band}=\max_{|j-k|\geq
\iota}\frac{1}{\sqrt{n}}\left|\sum^{n}_{i=1}(u_{ij}u_{ik})/\sqrt{\widehat{\gamma}_{u,jj}(0)\widehat{\gamma}_{u,kk}(0)}\right|,$$
where $\widehat{\gamma}_{u,jk}(0)=\sum^{n}_{i=1}u_{ij}u_{ik}/n$.
With some abuse of notation, let
$x_i=(\widetilde{u}_{i1}\widetilde{u}_{i1},\dots,\widetilde{u}_{i1}\widetilde{u}_{ip},\dots,\widetilde{u}_{ip}\widetilde{u}_{i1},\\
\dots,\widetilde{u}_{ip}\widetilde{u}_{ip})$
with $\widetilde{u}_{ij}=u_{ij}/\sqrt{\gamma_{u,jj}(0)}$.

\begin{theorem}\label{thm:bandtest}
Suppose the assumptions in Theorem \ref{thm:m-dep-boot} or Theorem
\ref{thm:dep-boot} hold for $\{x_i\}$, where $p$ is replaced by the
cardinality of the set $\{1\leq j,k\leq p: |j-k|\geq \iota\}$. Then
under Assumption \ref{assum:band} in the supplementary material and
$H_0$, we have
\begin{equation}
\sup_{\alpha\in (0,1)}\left|P\left(T_{band}\geq
c_{band}(\alpha)\right)-\alpha\right|\lesssim n^{-c},\quad c>0,
\end{equation}
where $c_{band}(\alpha)$ is given in Section~\ref{sec:bandtest}.
\end{theorem}
The proof of Theorem \ref{thm:bandtest} is similar as that of
Theorem~\ref{thm:app-1}, and thus skipped. In Section
\ref{subsec:band}, we show that Assumption \ref{assum:band} can be
verified under suitable primitive conditions.

To avoid direct estimation of the influence function,
we may alternatively apply the non-overlapping block bootstrap procedure in
Section~\ref{subsec:block}. Assume for simplicity that $N_0=b_nl_n$,
where $b_n,l_n\in\mathbb{Z}$. Let $\varrho_1,\dots,\varrho_{l_n}$ be
i.i.d uniform random variables on $\{0,\dots,l_n-1\}$ and define
$v_{(j-1)b_n+i}^*=v_{\varrho_jb_n+i}$ with $1\leq j\leq l_n$ and
$1\leq i\leq b_n.$ Compute the block bootstrap estimate
$\widehat{\Theta}_{q_0}^*$ based on the bootstrap sample
$\{v_i^*\}_{i=1}^{N_0}$. Let $c_{2}(\alpha)$ be the
$100(1-\alpha)$th quantile of the distribution of $\max_{1\leq j\leq
q_0}\sqrt{N_0}|\widehat{\theta}_{j}^*-\widehat{\theta}_j|$
conditional on the sample $\{u_i\}$. In what follows, we further
justify the validity of the non-overlapping block bootstrap in the
same framework.

\begin{assumption}\label{assum:app-2}
Assume that
$$P\left(P\left(\sqrt{N_0}\max_{1\leq j\leq q_0}|\mathcal{R}_{jN_0}^*-\mathcal{R}_{jN_0}|>C_3n^{-c_3}/\sqrt{\log(2q_0)}\bigg|\{u_i\}^{n}_{i=1}\right)>C_4n^{-c_4}\right)\leq
C_4n^{-c_4},$$ where
$\mathcal{R}_{N_0}^*=(\mathcal{R}_{1N_0}^*,\dots,\mathcal{R}_{q_0N_0}^*)=\mathcal{R}_{N_0}(v_1^*,\dots,v_{N_0}^*),$ and
$c_3,C_3,c_4,C_4>0.$
\end{assumption}

\begin{theorem}\label{thm:app-2}
Suppose the assumptions in Theorem \ref{thm:block-boot} hold for
$\{x_i\}$, where $p$ is replaced by $2q_0$. Then under Assumptions
\ref{assum:app-1}-\ref{assum:app-2}, we have
\begin{equation}\label{thm5.3}
\left|P\left(\max_{1\leq j\leq
q_0}\sqrt{N_0}|\widehat{\theta}_{j}-\widetilde{\theta}_j|\geq
c_{2}(\alpha)\right)-\alpha\right|\lesssim n^{-c},\quad c>0.
\end{equation}
\end{theorem}

\begin{remark}\label{rk:student}
An alternative way to construct the uniform confidence band or
perform hypothesis testing is based on the studentized statistic.
For example, let $\widehat{\sigma}_j^2$ be a consistent estimator of
$\lim_{n\rightarrow \infty} N_0\text{var}(\widehat{\theta}_{j})$.
Then the uniform confidence band can be constructed as
$$\left\{\Theta_{q_0}=(\theta_1,\dots,\theta_{q_0})'\in\mathbb{R}^{q_0}:\max_{1\leq
j\leq
q_0}\sqrt{N_0}\left|\widehat{\theta}_{j}-\theta_j\right|/\widehat{\sigma}_j\leq
\check{c}(\alpha)\right\}.$$ The blockwise multiplier bootstrap or
non-overlapping block bootstrap can be modified accordingly to obtain the critical
value $\check{c}(\alpha)$.
\end{remark}

\subsection{Extension to infinite dimensional parameters}\label{subsec:infinite}
To broaden the applicability of our method, we extend the above
results to cover infinite dimensional parameters that are
functionals of the joint distribution of $\{u_i\}_{i\in\mathbb{Z}}$,
denoted as $F_{\infty}$. A typical example is the spectral
quantities that depend on the distribution of the whole time series
rather than any finite dimensional distribution; see
Example~\ref{eg:spe}. Hence, the extension in this section is useful
in conducting inference for the spectrum of high dimensional time
series.

Suppose
$\Theta_{q_0}=(\theta_1,\dots,\theta_{q_0})'=\mathcal{T}_{\infty}(F_{\infty})$
and its estimator is
$\widehat{\Theta}_{q_0}:=\widehat{\Theta}_{q_0}(u_1,\dots,u_n)=
(\widehat{\theta}_1,\dots,\widehat{\theta}_{q_0})'$. Again, $q_0$ is
allowed to grow with $n$ or $p$. Assume that there exists a sequence of
approximating statistics for $\widehat{\Theta}_{q_0}$ that is a
functional of $\vartheta_n$-dimensional empirical distribution, and
a sequence of approximating (non-random) quantities
$\bar{\Theta}_{q_0}=(\bar{\theta}_1,\dots,\bar{\theta}_{q_0})'$ for
$\Theta_{q_0}$. Then our bootstrap method as proposed in
Section~\ref{subsec:als} still works provided that these two
approximation errors can be well controlled and similar regularity
conditions hold for the expansion of the approximating statistics
around $\bar{\Theta}_{q_0}$, i.e., (\ref{app:exp}). To be more precise, we impose the following assumption.
\begin{assumption}\label{assum:app-3}
For a sequence of positive integers $\vartheta_n$ that grow with
$n,$ let $v_{i,\vartheta_n}=(u_i,\dots,u_{i+\vartheta_n-1})$ with
$i=1,2,\dots,N_{0,\vartheta_n}:=n-\vartheta_n+1.$ Assume the
expansion,
\begin{equation}\label{app:exp}
\begin{split}
\mathcal{T}_{\vartheta_n}(F_{\vartheta_n}^{N_{0,\vartheta_n}})
:=&(\mathcal{T}_{1,\vartheta_n}(F_{\vartheta_n}^{N_{0,\vartheta_n}}),\dots,\mathcal{T}_{q_0,\vartheta_n}(F_{\vartheta_n}^{N_{0,\vartheta_n}}))'
\\=&\bar{\Theta}_{q_0}+\frac{1}{n}\sum^{N_{0,\vartheta_n}}_{i=1}IF(v_{i,\vartheta_n},F_{\vartheta_n})+\mathcal{R}_{N_{0,\vartheta_n}},
\end{split}
\end{equation}
where
$\mathcal{R}_{N_{0,\vartheta_n}}=(\mathcal{R}_{1,N_{0,\vartheta_n}},\dots,\mathcal{R}_{q_0,N_{0,\vartheta_n}})'$
is a remainder term. Denote
$\Upsilon_{j,\vartheta_n}=|\mathcal{R}_{j,N_{0,\vartheta_n}}|+|\widehat{\theta}_{j}-\mathcal{T}_{j,\vartheta_n}(F_{\vartheta_n}^{N_{0,\vartheta_n}})|$.
Suppose that
$$P\left(\max_{1\leq j\leq q_0}\sqrt{N_0}\Upsilon_{j,\vartheta_n}>C_1n^{-c_1}/\sqrt{\log(2q_0)}\right)<C_1n^{-c_1},$$
and $n^{c_1}\sqrt{\log(2q_0)}\max_{1\leq j\leq
q_0}\sqrt{N_0}|\bar{\theta}_{j}-\theta_j|=o(1)$ for some  $c_1,C_1>0$.
\end{assumption}

We next illustrate the validity of expansion (\ref{app:exp}) using a spectral
mean example.
\begin{example}\label{eg:spe}
Consider the spectral mean
$G(F_u,\phi)=\int^{\pi}_{-\pi}\text{tr}(\phi(\lambda)F_{u}(\lambda))d\lambda$,
where $\text{tr}$ denotes the trace of a square matrix, $F_u(\cdot)$
is the spectral density of $\{u_i\}$ and
$\phi(\cdot):[-\pi,\pi]\rightarrow \mathbb{R}^{p\times p}.$ For
simplicity, assume that $\E u_i=0.$ Suppose the quantity of interest
is $\Theta_0=(G(F_u,\phi_1),\dots,G(F_u,\phi_{q_0}))'$ with
$\phi_k(\cdot):[-\pi,\pi]\rightarrow \mathbb{R}^{p\times p}$ for
$1\leq k\leq q_0.$ Here $\Theta_0$ can be interpreted as the
projection of the spectral density matrix onto $q_0$ directions
defined by $\phi_k(\cdot)$ with $1\leq k\leq q_0.$ A sample analogue
of $F_u(\lambda)$ is the periodogram
$\mathcal{I}_{n,u}(\lambda)=(2\pi
n)^{-1}\sum^{n}_{i,j=1}u_iu_j'\exp(\I(i-j)\lambda)$ with
$\I=\sqrt{-1}$. Then a plug-in estimator for $\Theta_{q_0}$ is given
by
$\widehat{\Theta}_{q_0}=(G(\mathcal{I}_{n,u},\phi_1),\dots,G(\mathcal{I}_{n,u},\phi_{q_0}))'$.
Letting $\widehat{\Gamma}_{n,h}=\sum^{n-h}_{j=1}u_{j+h}u_{j}'/n$,
then
$G(\mathcal{I}_{n,u},\phi_{k})=\sum^{n-1}_{h=1-n}\text{tr}(\widetilde{\phi}_{hk}\widehat{\Gamma}_{n,h})$
with $\widetilde{\phi}_{hk}=\int^{\pi}_{-\pi}\phi_k(\lambda)\exp(\I
h\lambda)d\lambda/(2\pi).$ Consider the approximating quantity
$\bar{\theta}_{j}=\sum^{\vartheta_n-1}_{h=1-\vartheta_n}\text{tr}(\widetilde{\phi}_{hk}\Gamma_h)$
with $\Gamma_h=\E u_{j+h}u_j'$. It is then straightforward to see
that
\begin{equation}
\mathcal{T}_{\vartheta_n}(F_{\vartheta_n}^{N_{0,\vartheta_n}}):=\sum^{\vartheta_n-1}_{h=1-\vartheta_n}\text{tr}(\widetilde{\phi}_{hk}\widehat{\Gamma}_h)=\bar{\theta}_j
+\frac{1}{n}\sum^{N_{0,\vartheta_n}}_{i=1}IF(v_{i,\vartheta_n},F_{\vartheta_n})+\mathcal{R}_{j,N_{0,\vartheta_n}},
\end{equation}
where
$IF(v_{i,\vartheta_n},F_{\vartheta_n})=\sum^{\vartheta_n-1}_{h=1-\vartheta_n}\text{tr}\{\widetilde{\phi}_{hk}(u_{i+h}u_i'-\Gamma_h)\}$
and $\mathcal{R}_{j,N_{0,\vartheta_n}}$ is the corresponding
remainder term.
\end{example}

Recall that $\Theta_{q_0}=\mathcal{T}_{\infty}(F_{\infty})$ with
$F_{\infty}$ being the joint distribution of
$\{u_i\}_{i\in\mathbb{Z}}$. The statistic for testing the null
hypothesis $H_0: \Theta_{q_0}=\widetilde{\Theta}_{q_0}$ versus the
alternative $H_a: \Theta_{q_0}\neq \widetilde{\Theta}_{q_0}$, where
$\widetilde{\Theta}_{q_0}=(\widetilde{\theta}_{1},\dots,\widetilde{\theta}_{q_0})'$,
is given by
\begin{equation}
\max_{1\leq j\leq
q_0}\sqrt{N_0}|\widehat{\theta}_{j}-\widetilde{\theta}_j|\geq
c(\alpha).
\end{equation}
With some abuse of notation, we now define
$x_i:=x_{in}=(IF(v_{i,\vartheta_n},F_{\vartheta_n})',-IF(v_{i,\vartheta_n},F_{\vartheta_n})')'$
and
$\widehat{x}_i=(\widehat{IF}(v_{i,\vartheta_n},F_{\vartheta_n}^{N_{0,\vartheta_n}})',-\widehat{IF}(v_{i,\vartheta_n},F_{\vartheta_n}^{N_{0,\vartheta_n}})')'$
with
$\widehat{IF}(v_{i,\vartheta_n},F_{\vartheta_n}^{N_{0,\vartheta_n}})$
being some estimate of $IF(v_{i,\vartheta_n},F_{\vartheta_n})$ (note
that in this case $\{x_{in}\}_{i=1}^{N_{0,\vartheta_n}}$ is an
array). Suppose
$N_{0,\vartheta_n}=(N_{1,\vartheta_n}+M_{1,\vartheta_n})r_{1,\vartheta_n}$.
We can define $\widehat{A}_{ij}$ and $\widehat{B}_{ij}$ in a similar
way as before (see (\ref{eq:app-AB})), where $1\leq i\leq
r_{1,\vartheta_n}$ and $1\leq j\leq 2q_0$. Let
\begin{align*}
T_{\widehat{D}}=\max_{1\leq j\leq
2q_0}\frac{1}{\sqrt{n}}\sum^{r_{1,\vartheta_n}}_{i=1}\widehat{D}_{ij},
\end{align*}
where
$\widehat{D}_{ij}=\widehat{A}_{ij}e_i+\widehat{B}_{ij}\widetilde{e}_i$
with $\{(e_{i},\widetilde{e}_i)\}$ being a sequence of i.i.d
$N(0,I_2)$ independent of $\{u_i\}$. The
bootstrap critical value is then given by
\begin{equation}
c_{1}(\alpha):=\inf\{t\in\mathbb{R}:P(T_{\widehat{D}}\leq
t|\{x_i\}^{n}_{i=1})\geq 1-\alpha\}.
\end{equation}
Following the arguments in the proof of Theorem \ref{thm:app-1}, we
obtain the following result.

\begin{theorem}\label{thm:app-3}
Suppose Assumption \ref{assum:app-3} holds and the assumptions in
Theorem \ref{thm:m-dep-boot} or Theorem \ref{thm:dep-boot} are
satisfied for $\{x_i\}$, where $p$ is replaced by $2q_0$. Assume in
addition that $P(\mathcal{E}_{AB}\{\log(2q_0)\}^2>C_2n^{-c_2})\leq
C_2n^{-c_2}$, where $c_2,C_2>0$, and
\begin{align*}
\mathcal{E}_{AB}=\max_{1\leq j\leq
2q_0}\left|\frac{1}{n}\sum^{r}_{i=1}\{(A_{ij}-\widehat{A}_{ij})^2+(B_{ij}-\widehat{B}_{ij})^2\}\right|,
\end{align*}
with
$A_{ij}=\sum^{iN_{1,\vartheta_n}+(i-1)M_{1,\vartheta_n}}_{l=iN_{1,\vartheta_n}+(i-1)M_{1,\vartheta_n}-N_{1,\vartheta_n}+1}x_{lj}$
and
$B_{ij}=\sum^{i(N_{1,\vartheta_n}+M_{1,\vartheta_n})}_{l=i(N_{1,\vartheta_n}+M_{1,\vartheta_n})-M_{1,\vartheta_n}+1}x_{lj}$.
Then we have for some $c>0,$
\begin{equation}
\sup_{\alpha\in (0,1)}\left|P\left(\max_{1\leq j\leq
q_0}\sqrt{N_0}|\widehat{\theta}_{j}-\widetilde{\theta}_j|\geq
c_{1}(\alpha)\right)-\alpha\right|\lesssim n^{-c}.
\end{equation}
\end{theorem}

\newpage
\vskip 1em \centerline{\Large \bf Supplementary Material} \vskip 1em
\setcounter{subsection}{0}
\renewcommand{\thesubsection}{S.\arabic{subsection}}
\setcounter{equation}{0}
\renewcommand{\theequation}{S.\arabic{equation}}
\setcounter{lemma}{0}
\renewcommand{\thelemma}{S.\arabic{lemma}}
\setcounter{theorem}{0}
\renewcommand{\thetheorem}{S.\arabic{theorem}}
\setcounter{assumption}{0}
\renewcommand{\theassumption}{S.\arabic{assumption}}
\setcounter{proposition}{0}
\renewcommand{\theproposition}{S.\arabic{proposition}}
\setcounter{example}{0}
\renewcommand{\theexample}{S.\arabic{example}}
\setcounter{remark}{0}
\renewcommand{\theremark}{S.\arabic{remark}}

Throughout the supplementary material, define the generic constants
$C$ and $C'$ that are independent of $n$ and $p$. For a set
$\mathcal{A}$, denote by $|\mathcal{A}|$ its cardinality.

\subsection{Proofs of the main results in Section \ref{sec:maxima}}
\begin{proof}[Proof of Proposition \ref{prop1}]
Define $Z(t)=\sum^{n}_{i=1}Z_i(t)$ with the Slepian interpolation
$Z_i(t)=(\sqrt{t}\widetilde{x}_i+\sqrt{1-t}\widetilde{y}_i)/\sqrt{n}$
and $0\leq t\leq 1.$ Let $\Psi(t)=\E m(Z(t)).$ Define
$V^{(i)}(t)=\sum_{j\in \widetilde{N}_i}Z_j(t)$ and
$Z^{(i)}(t)=Z(t)-V^{(i)}(t)$. Write $\partial_{j}m(x)=\partial
m(x)/\partial x_j$, $\partial_{jk}m(x)=\partial^2 m(x)/\partial
x_j\partial x_k$ and $\partial_{jkl}m(x)=\partial^3m(x)/\partial
x_j\partial x_k\partial x_l$ for $j,k,l=1,2,\dots,p$, where
$x=(x_1,x_2,\dots,x_p)'$. Note that
\begin{equation}
\begin{split}
\E m(\widetilde{X})-\E
m(\widetilde{Y})=&\Psi(1)-\Psi(0)=\int^{1}_{0}\Psi'(t)dt
=\frac{1}{2}\sum^{n}_{i=1}\sum^{p}_{j=1}\int^{1}_{0}\E [\partial_j
m(Z(t))\dot{Z}_{ij}(t)]dt
\\=&\frac{1}{2}(I_1+I_2+I_3),
\end{split}
\end{equation}
where
$\dot{Z}_{ij}(t)=\{\widetilde{x}_{ij}/\sqrt{t}-\widetilde{y}_{ij}/\sqrt{1-t}\}/\sqrt{n},$
and
\begin{equation}
\begin{split}
&I_1=\sum^{n}_{i=1}\sum^{p}_{j=1}\int^{1}_{0}\E[\partial_j
m(Z^{(i)}(t))\dot{Z}_{ij}(t)]dt,
\\&I_2=\sum^{n}_{i=1}\sum^{p}_{k,j=1}\int^{1}_{0}\E[\partial_k\partial_j m(Z^{(i)}(t))\dot{Z}_{ij}(t)V_{k}^{(i)}(t)]dt, \\
&I_3=\sum^{n}_{i=1}\sum^{p}_{k,l,j=1}\int^{1}_{0}\int^{1}_{0}(1-\tau)\E[\partial_l\partial_k\partial_jm(Z^{(i)}(t)+\tau
V^{(i)}(t))\dot{Z}_{ij}(t)V_{k}^{(i)}(t)V_{l}^{(i)}(t)]dtd\tau.
\end{split}
\end{equation}

Using the fact that $Z^{(i)}(t)$ and $\dot{Z}_{ij}(t)$ are
independent, and $\E \dot{Z}_{ij}(t)=0$, we have $I_1=0.$ To bound
the second term, define the expanded neighborhood around $N_i$,
$$\mathcal{N}_i=\{j: \{j,k\}\in E_n\text{~for some~} k\in N_i\},$$
and
$\mathcal{Z}^{(i)}(t)=Z(t)-\sum_{l\in\mathcal{N}_i\cup \widetilde{N}_i}Z_{l}(t)=Z^{(i)}(t)-\mathcal{V}^{(i)}(t)$,
where $\mathcal{V}^{(i)}(t)=\sum_{l\in \mathcal{N}_i\setminus
\widetilde{N}_i}Z_l(t)$ with $\mathcal{N}_i\setminus
\widetilde{N}_i=\{k\in\mathcal{N}_i: k\notin \widetilde{N}_i\}$. By
Taylor expansion, we have
\begin{align*}
I_2=&\sum^{n}_{i=1}\sum^{p}_{k,j=1}\int^{1}_{0}\E[\partial_k\partial_j
m(\mathcal{Z}^{(i)}(t))\dot{Z}_{ij}(t)V_{k}^{(i)}(t)]dt
\\&+\sum^{n}_{i=1}\sum^{p}_{k,j,l=1}\int^{1}_{0}\int^{1}_{0}\E[\partial_k\partial_j\partial_l m(\mathcal{Z}^{(i)}(t)+\tau \mathcal{V}^{(i)}(t))\dot{Z}_{ij}(t)V_{k}^{(i)}(t)\mathcal{V}^{(i)}_l(t)]dtd\tau
\\=&\sum^{n}_{i=1}\sum^{p}_{k,j=1}\int^{1}_{0}\E[\partial_k\partial_j m(\mathcal{Z}^{(i)}(t))]\E[\dot{Z}_{ij}(t)V_{k}^{(i)}(t)]dt
\\&+\sum^{n}_{i=1}\sum^{p}_{k,j,l=1}\int^{1}_{0}\int^{1}_{0}\E[\partial_k\partial_j\partial_l m(\mathcal{Z}^{(i)}(t)+\tau \mathcal{V}^{(i)}(t))\dot{Z}_{ij}(t)V_{k}^{(i)}(t)\mathcal{V}^{(i)}_l(t)]dtd\tau
\\=& I_{21}+I_{22},
\end{align*}
where we have used the fact that $\dot{Z}_{ij}(t)V_{k}^{(i)}(t)$ and
$\mathcal{Z}^{(i)}(t)$ are independent. 

Let
$M_{xy}=\max\{M_x,M_y\}$. By the assumption that $2\sqrt{5}\beta
D_n^2M_{xy}/\sqrt{n}\leq 1,$
\begin{align*}
\max_{1\leq j\leq p}\left|\sum_{l\in\mathcal{N}_i\cup
\widetilde{N}_i}Z_{lj}(t)\right|\leq & \max_{1\leq j\leq
p}\sum_{l\in\mathcal{N}_i\cup \widetilde{N}_i}|Z_{lj}(t)| \leq
D_n^2\sup_{t\in [0,1]}(2\sqrt{t}+\sqrt{1-t})M_{xy}/\sqrt{n}
\\ \leq& \sqrt{5}D_n^2M_{xy}/\sqrt{n} \leq \beta^{-1}/2 \leq \beta^{-1},
\end{align*}
where the second inequality comes from the facts that
$|\widetilde{x}_{ij}|\leq 2M_{xy}$, $|\widetilde{y}_{ij}|\leq
M_{xy}$ and $|\mathcal{N}_i\cup \widetilde{N}_i|\leq D_n^2$. By
Lemma A.5 in \cite{cck13}, we have for every $1\leq j,k,l\leq p,$
\begin{align*}
|\partial_j\partial_km(z)|\leq U_{jk}(z),\quad
|\partial_j\partial_k\partial_l m(z)|\leq U_{jkl}(z),
\end{align*}
where $U_{jk}(z)$ and $U_{jkl}(z)$ satisfy that
\begin{align*}
\sum^{p}_{j,k=1}U_{jk}(z)\leq (G_2+2G_1\beta),\quad
\sum^{p}_{j,k,l=1}U_{jkl}(z)\leq (G_3+6G_2\beta+6G_1\beta^2),
\end{align*}
with $G_k=\sup_{z\in\mathbb{R}}|\partial^k g(z)/\partial z^k|$ for
$k\geq 0$. Along with Lemma A.6 in \cite{cck13}, we obtain
\begin{align*}
|I_{21}|\leq&
\sum^{n}_{i=1}\sum^{p}_{k,j=1}\int^{1}_{0}\E[U_{jk}(\mathcal{Z}^{(i)}(t))]|\E[\dot{Z}_{ij}(t)V_{k}^{(i)}(t)]|dt
\\ \lesssim & \sum^{n}_{i=1}\sum^{p}_{k,j=1}\int^{1}_{0}\E[U_{jk}(Z(t))]|\E[\dot{Z}_{ij}(t)V_{k}^{(i)}(t)]|dt
\\ \lesssim &  (G_2+G_1\beta)\int^{1}_{0}\max_{1\leq j,k\leq
p}\sum^{n}_{i=1}|\E[\dot{Z}_{ij}(t)V_{k}^{(i)}(t)]|dt.
\end{align*}
Since $2\sqrt{5}\beta D_n^2M_{xy}/\sqrt{n} \leq 1$, we have
\begin{align}\label{eq:I22}
|I_{22}|\leq&
\sum^{n}_{i=1}\sum^{p}_{k,j,l=1}\int^{1}_{0}\int^{1}_{0}\E[|\partial_k\partial_j\partial_l
m(\mathcal{Z}^{(i)}(t)+\tau
\mathcal{V}^{(i)}(t))|\cdot|\dot{Z}_{ij}(t)V_{k}^{(i)}(t)\mathcal{V}^{(i)}_l(t)|]dtd\tau
\notag
\\ \leq& \sum^{n}_{i=1}\sum^{p}_{k,j,l=1}\int^{1}_{0}\int^{1}_{0}\E[U_{kjl}(\mathcal{Z}^{(i)}(t)+\tau \mathcal{V}^{(i)}(t))|\dot{Z}_{ij}(t)V_{k}^{(i)}(t)\mathcal{V}^{(i)}_l(t)|]dtd\tau \notag
\\ \lesssim & \sum^{n}_{i=1}\sum^{p}_{k,j,l=1}\int^{1}_{0}\E[U_{kjl}(Z(t))|\dot{Z}_{ij}(t)V_{k}^{(i)}(t)\mathcal{V}^{(i)}_l(t)|]dtd\tau \notag
\\ \leq & \int^{1}_{0}\E\left[\sum^{p}_{k,j,l=1}U_{kjl}(Z(t))\max_{1\leq k,j,l\leq p}\sum^{n}_{i=1}|\dot{Z}_{ij}(t)V_{k}^{(i)}(t)\mathcal{V}^{(i)}_l(t)|\right]dtd\tau \notag
\\ \lesssim & (G_3+G_2\beta+G_1\beta^2)\int^{1}_{0}\E\max_{1\leq k,j,l\leq p}\sum^{n}_{i=1}|\dot{Z}_{ij}(t)V_{k}^{(i)}(t)\mathcal{V}^{(i)}_l(t)|dtd\tau.
\end{align}

To bound the integration on (\ref{eq:I22}), we let
$w(t)=1/(\sqrt{t}\wedge \sqrt{1-t})$ and note that
\begin{align*}
&\int^{1}_{0}\E\max_{1\leq k,j,l\leq
p}\sum^{n}_{i=1}|\dot{Z}_{ij}(t)V_{k}^{(i)}(t)\mathcal{V}^{(i)}_l(t)|dt
\\ \leq& \int^{1}_{0}\E\max_{1\leq k,j,l\leq p}\left(\sum^{n}_{i=1}|\dot{Z}_{ij}(t)|^3\right)^{1/3}\left(\sum^{n}_{i=1}|V_{k}^{(i)}(t)|^3\right)^{1/3}\left(\sum^{n}_{i=1}|\mathcal{V}^{(i)}_l(t)|^3\right)^{1/3}dt
\\ \leq& \int^{1}_{0}w(t)\left(\E\max_{1\leq j\leq p}\sum^{n}_{i=1}|\dot{Z}_{ij}(t)/w(t)|^3\E\max_{1\leq k\leq p}\sum^{n}_{i=1}|V_{k}^{(i)}(t)|^3\E\max_{1\leq l\leq p}\sum^{n}_{i=1}|\mathcal{V}^{(i)}_l(t)|^3\right)^{1/3}dt.
\end{align*}
As for $I_{21}$, by the assumption that $\E y_{ij}y_{lk}=\E
x_{ij}x_{lk}$ (in fact, we only need to require that
$\sum_{k\in\widetilde{N}_i}\E
x_{i}x_{k}'=\sum_{k\in\widetilde{N}_i}\E y_{i}y_{k}'$ for all $i$),
we have
\begin{equation}\label{eq:cov-ts}
\begin{split}
& \max_{1\leq j,k\leq
p}\sum^{n}_{i=1}|\E[\dot{Z}_{ij}(t)V_{k}^{(i)}(t)]|=\max_{1\leq
j,k\leq p}\frac{1}{n}\sum^{n}_{i=1}\left|\sum_{l\in
\widetilde{N}_i}(\E
\widetilde{x}_{ij}\widetilde{x}_{lk}-\E\widetilde{y}_{ij}\widetilde{y}_{lk})\right|
\\=&\max_{1\leq j,k\leq p}\frac{1}{n}\sum^{n}_{i=1}\left|\sum_{l\in \widetilde{N}_i}(\E \widetilde{x}_{ij}\widetilde{x}_{lk}-\E x_{ij}x_{lk})+\sum_{l\in \widetilde{N}_i}(\E y_{ij}y_{lk}-\E \widetilde{y}_{ij}\widetilde{y}_{lk})\right|
\\ \leq &\max_{1\leq j,k\leq p}\frac{1}{n}\sum^{n}_{i=1}\left|\sum_{l\in \widetilde{N}_i}\left\{\E y_{lk}(y_{ij}-\widetilde{y}_{ij})+\E \widetilde{y}_{ij}(y_{lk}-\widetilde{y}_{lk})\right\}\right|
\\&+\max_{1\leq j,k\leq p}\frac{1}{n}\sum^{n}_{i=1}\left|\sum_{l\in \widetilde{N}_i}\left\{\E x_{lk}(x_{ij}-\widetilde{x}_{ij})+\E \widetilde{x}_{ij}(x_{lk}-\widetilde{x}_{lk})\right\}\right|
\\ \leq & \phi(M_{x},M_y).
\end{split}
\end{equation}

Using similar arguments as above, we have
$|I_3| \lesssim (G_3+G_2\beta+G_1\beta^2)I_{31}$
with
$$I_{31}\leq \int^{1}_{0}w(t)\left(\E\max_{1\leq j\leq p}\sum^{n}_{i=1}|\dot{Z}_{ij}(t)/w(t)|^3\E\max_{1\leq k\leq p}\sum^{n}_{i=1}|V_{k}^{(i)}(t)|^3\E\max_{1\leq l\leq p}\sum^{n}_{i=1}|V^{(i)}_l(t)|^3\right)^{1/3}dt.$$
We first consider the term $\E\max_{1\leq j\leq
p}\sum^{n}_{i=1}|\dot{Z}_{ij}(t)/w(t)|^3$. Using the fact that
$|\dot{Z}_{ij}(t)/w(t)|\leq
(|\widetilde{x}_{ij}|+|\widetilde{y}_{ij}|)/\sqrt{n},$ we get
$$\E\max_{1\leq j\leq p}\sum^{n}_{i=1}|\dot{Z}_{ij}(t)/w(t)|^3\lesssim \frac{1}{n^{3/2}}\E \max_{1\leq j\leq p}\sum^{n}_{i=1}(|\widetilde{x}_{ij}|^3+|\widetilde{y}_{ij}|^3)
\lesssim \frac{1}{\sqrt{n}}(m_{x,3}^3+m_{y,3}^3).$$ On the other
hand, notice that
\begin{align*}
\E\max_{1\leq k\leq p}\sum^{n}_{i=1}|V_{k}^{(i)}(t)|^3  \leq &
D_n^{2}\E\max_{1\leq k\leq p}\sum^{n}_{i=1}\sum_{j\in
\widetilde{N}_i}|Z_{jk}(t)|^3 \lesssim
\frac{D_n^{2}}{n^{3/2}}\E\max_{1\leq k\leq
p}\sum^{n}_{i=1}\sum_{j\in
\widetilde{N}_i}(|\widetilde{x}_{jk}|^3+|\widetilde{y}_{jk}|^3)
\\ \lesssim & \frac{D_n^{3}}{\sqrt{n}}(m_{x,3}^3+m_{y,3}^3).
\end{align*}
Similarly, we have
\begin{align*}
\E\max_{1\leq l\leq p}\sum^{n}_{i=1}|\mathcal{V}^{(i)}_l(t)|^3\leq &
D_n^{4}\E\max_{1\leq l\leq p}\sum^{n}_{i=1}\sum_{j\in
\mathcal{N}_i}|Z_{jl}(t)|^3\leq \frac{D_n^{4}}{n^{3/2}}\E\max_{1\leq
l\leq p}\sum^{n}_{i=1}\sum_{j\in
\mathcal{N}_i}(|\widetilde{x}_{jl}|^3+|\widetilde{y}_{jl}|^3)
\\ \lesssim & \frac{D_n^{6}}{\sqrt{n}}(m_{x,3}^3+m_{y,3}^3).
\end{align*}
Note that $\int^{1}_{0}w(t)dt\lesssim 1.$ Summarizing the above results,
we have
\begin{align*}
&I_2\lesssim
(G_2+G_1\beta)\phi(M_x,M_y)+(G_3+G_2\beta+G_1\beta^2)\frac{D_n^{3}}{\sqrt{n}}(m_{x,3}^3+m_{y,3}^3),
\\
&I_3\lesssim
(G_3+G_2\beta+G_1\beta^2)\frac{D_n^{2}}{\sqrt{n}}(m_{x,3}^3+m_{y,3}^3).
\end{align*}

Alternatively, we can bound $I_3$ in the following way. By Lemmas A.5
and A.6 in \cite{cck13}, we have
\begin{align*}
|I_3|=&\sum^{n}_{i=1}\sum^{p}_{k,l,j=1}\int^{1}_{0}\int^{1}_{0}(1-\tau)\E[\partial_l\partial_k\partial_jm(Z^{(i)}(t)+\tau
V^{(i)}(t))\dot{Z}_{ij}(t)V_{k}^{(i)}(t)V_{l}^{(i)}(t)]dtd\tau
\\ \lesssim & \sum^{n}_{i=1}\sum^{p}_{k,j,l=1}\int^{1}_{0}\E [U_{kjl}(\mathcal{Z}^{(i)}(t))]\E|\dot{Z}_{ij}(t)V_{k}^{(i)}(t)V_{l}^{(i)}(t)|dt
\\ \lesssim & \sum^{n}_{i=1}\sum^{p}_{k,j,l=1}\int^{1}_{0}\E [U_{kjl}(Z(t))]\E|\dot{Z}_{ij}(t)V_{k}^{(i)}(t)V_{l}^{(i)}(t)|dt
\\ \leq & n(G_3+G_2\beta+G_1\beta^2)\int^{1}_{0}w(t)\max_{1\leq j,k,l\leq p} (\bar{\E}|\dot{Z}_{ij}(t)/w(t)|^3)^{1/3} (\bar{\E}|V_{k}^{(i)}(t)|^3)^{1/3} (\bar{\E}|V_{l}^{(i)}(t)|^3)^{1/3}dt.
\end{align*}
Notice that
\begin{align*}
\max_{1\leq j\leq p}\bar{\E}|\dot{Z}_{ij}(t)/w(t)|^3\leq &
\frac{1}{n^{3/2}}\max_{1\leq j\leq
p}\bar{\E}(|\widetilde{x}_{ij}|+|\widetilde{y}_{ij}|)^3 \lesssim
\frac{1}{n^{3/2}}(\bar{m}_{x,3}^3+\bar{m}_{y,3}^3).
\end{align*}
It is not hard to see that
\begin{align*}
\max_{1\leq k\leq p}\bar{\E}|V_{k}^{(i)}(t)|^3\leq &
D_n^{2}\max_{1\leq k\leq p}\bar{\E}\sum_{j\in
\widetilde{N}_i}|Z_{jk}(t)|^3 \lesssim
\frac{D_n^{3}}{n^{3/2}}(\bar{m}_{x,3}^3+\bar{m}_{y,3}^3).
\end{align*}
Thus we derive that
\begin{align*}
I_3\lesssim
(G_3+G_2\beta+G_1\beta^2)\frac{D_n^{2}}{\sqrt{n}}(\bar{m}_{x,3}^3+\bar{m}_{y,3}^3).
\end{align*}

Therefore, we obtain
\begin{equation}\label{eq:prop1-trun}
\begin{split}
|\E [m(\widetilde{X})-m(\widetilde{Y})]|\lesssim &
(G_2+G_1\beta)\phi(M_x,M_y)+(G_3+G_2\beta+G_1\beta^2)\frac{D_n^{3}}{\sqrt{n}}(m_{x,3}^3+m_{y,3}^3)
\\&+(G_3+G_2\beta+G_1\beta^2)\frac{D_n^{2}}{\sqrt{n}}(\bar{m}_{x,3}^3+\bar{m}_{y,3}^3).
\end{split}
\end{equation}
Using the above arguments, we can show that
\begin{align}\label{eq:prop1-I22}
I_{22}\lesssim
(G_3+G_2\beta+G_1\beta^2)\frac{D_n^{3}}{\sqrt{n}}(\bar{m}_{x,3}^3+\bar{m}_{y,3}^3),
\end{align}
provided that $2\sqrt{5}\beta D_n^3M_{xy}/\sqrt{n} \leq 1$. This proves the last statement of Proposition \ref{prop1}.

Note
that $|m(x)-m(y)|\leq 2G_0$ and $|m(x)-m(y)|\leq G_1\max_{1\leq
j\leq p}|x_j-y_j|$ with $x=(x_1,\dots,x_p)'$ and
$y=(y_1,\dots,y_p)'.$ So
\begin{equation}\label{eq:prop1-trun-error}
\begin{split}
|\E[m(X)-m(\widetilde{X})]|\leq&
|\E[(m(X)-m(\widetilde{X}))\mathcal{I}]|+|\E[(m(X)-m(\widetilde{X}))(1-\mathcal{I})]|
\\ \lesssim& G_1\Delta+G_0\E[1-\mathcal{I}],\\
|\E[m(Y)-m(\widetilde{Y})]| \lesssim&
G_1\Delta+G_0\E[1-\mathcal{I}].
\end{split}
\end{equation}
The conclusion follows by combining (\ref{eq:prop1-trun}),
(\ref{eq:prop1-I22}) and (\ref{eq:prop1-trun-error}).
\end{proof}

\begin{proof}[Proof of Corollary \ref{corollary:m-dep}]
Notice that $D_n=2M+1$, $|\widetilde{N}_i|\leq 2M+1$ and
$|\mathcal{N}_i\cup \widetilde{N}_i|\leq 4M+1$. Define the
$\mathfrak{N}_i=\{j:\{j,k\}\in E_n\text{~for some~}k\in
\mathcal{N}_i\}.$ Then $|\mathfrak{N}_i\cup \mathcal{N}_i\cup
\widetilde{N}_i|\leq 6M+1.$ Following the arguments in the proof of
Proposition \ref{prop1}, we can show that
$$\max_{1\leq l\leq p}\bar{\E}|\mathcal{V}_{l}^{(i)}(t)|^3\lesssim \frac{D_n^{3}}{n^{3/2}}(\bar{m}_{x,3}^3+\bar{m}_{y,3}^3),$$
which implies that
$$I_{22}\lesssim (G_3+G_2\beta+G_1\beta^2)\frac{D_n^{2}}{\sqrt{n}}(\bar{m}_{x,3}^3+\bar{m}_{y,3}^3).$$
The conclusion follows from the proof of Proposition \ref{prop1}.
\end{proof}

\begin{proof}[Proof of Lemma \ref{lemma:self}]
We only need to prove the result for $x>1$ as the inequality holds
trivially for $x<1$. Suppose that the distributions of $A_i$ and
$B_i$ are both symmetric, then we have
\begin{align*}
P\left(\sum^{n}_{i=1}x_{ij}>xV_{nj}\right)\leq &
P\left(\sum^{r}_{i=1}(A_{ij}+B_{ij})>xV_{nj}\right)
\\ \leq &P\left(\sum^{r}_{i=1}A_{ij}>xV_{nj}/2\right)+ P\left(\sum^{r}_{i=1}B_{ij}>xV_{nj}/2\right)
\\ \leq &P\left(\sum^{r}_{i=1}A_{ij}>xV_{1nj}/2\right)+ P\left(\sum^{r}_{i=1}B_{ij}>xV_{2nj}/2\right)
\\ \leq &2\exp(-x^2/8),
\end{align*}
where we have used Theorem 2.15 in \cite{dls2009}. 

Let
$\{\xi_{ij}\}^{n}_{i=1}$ be an independent copy of
$\{x_{ij}\}^{n}_{i=1}$ in the sense that $\{\xi_{ij}\}^{n}_{i=1}$ have the same joint
distribution as that for $\{x_{ij}\}^{n}_{i=1}$, and define
$V_{nj}'$ ($A_{ij}'$ and $B_{ij}'$) in the same way as $V_{nj}$
($A_{ij}$ and $B_{ij}$) by replacing $\{x_{ij}\}^{n}_{i=1}$ with
$\{\xi_{ij}\}^{n}_{i=1}$. Following the arguments in the proof of
Theorem 2.16 in \cite{dls2009}, we deduce that for $x>1$,
\begin{align*}
&\left\{\sum^{n}_{i=1}x_{ij}>x(a_j+b_j+V_{nj}),\sum^{n}_{i=1}\xi_{ij}\leq a_j,V_{nj}'\leq b_j\right\} \\
&\subset\left\{\sum^{n}_{i=1}(x_{ij}-\xi_{ij})\geq x(a_j+b_j+V_{nj})-a_j,V_{nj}'\leq b_j\right\} \\
&\subset\left\{\sum^{n}_{i=1}(x_{ij}-\xi_{ij})\geq x(a_j+b_j+V_{nj}^*-V_{nj}')-a_j,V_{nj}'\leq b_j\right\} \\
&\subset\left\{\sum^{n}_{i=1}(x_{ij}-\xi_{ij})\geq
xV_{nj}^*\right\},
\end{align*}
where we have used the fact that
$$V_{nj}^*\equiv\sqrt{\sum^{r}_{l=1}(A_{lj}-A_{lj}')^2+\sum^{r}_{l=1}(B_{lj}-B_{lj}')^2}\leq V_{nj}+V_{nj}'.$$

We note that $A_{lj}-A_{lj}'$ and $B_{lj}-B_{lj}'$ are symmetric,
and
$$P\left(\sum^{n}_{i=1}\xi_{ij}\leq a_j,V_{nj}'\leq b_j\right)\geq 1/2.$$
Thus we obtain
\begin{align*}
P\left(\sum^{n}_{i=1}x_{ij}\geq
x(a_j+b_j+V_{nj})\right)=&\frac{P(\sum^{n}_{i=1}x_{ij}\geq
x(a_j+b_j+V_{nj}),\sum^{n}_{i=1}\xi_{ij}\leq a_j,V_{nj}'\leq
b_j)}{P(\sum^{n}_{i=1}\xi_{ij}\leq a_j,V_{nj}'\leq b_j)}
\\ \leq &2P\left(\sum^{n}_{i=1}x_{ij}\geq x(a_j+b_j+V_{nj}),\sum^{n}_{i=1}\xi_{ij}\leq a_j,V_{nj}'\leq b_j\right)
\\ \leq &2P\left(\sum^{n}_{i=1}(x_{ij}-\xi_{ij})\geq xV_{nj}^*\right)
\\ \leq &4\exp(-x^2/8).
\end{align*}
Hence we get
$$P\left(\left|\sum^{n}_{i=1}x_{ij}\right|\geq x(a_j+b_j+V_{nj})\right)\leq 8\exp(-x^2/8).$$
In particular, we can choose $b^2_j=4\E V_{nj}^2$ and
$a^2_j=2b^2_j=8\E V_{nj}^2$ because $4\E(\sum^{n}_{i=1}x_{ij})^2\leq
8\E(\sum^{r}_{j=1}A_j)^2+8\E(\sum^{r}_{j=1}B_j)^2=8\E V_{nj}^2$.
\end{proof}

\begin{proof}[Proof of Theorem \ref{thm1}]
Note that
\begin{align*}
\E[1-\mathcal{I}]\leq &P(\max_{1\leq j\leq
p}|X_j-\widetilde{X}_j|>\Delta)+P(\max_{1\leq j\leq
p}|Y_j-\widetilde{Y}_j|>\Delta)
\\ \leq&\sum^{p}_{j=1}\left\{P(|X_j-\widetilde{X}_j|>\Delta)+P(|Y_j-\widetilde{Y}_j|>\Delta)\right\}.
\end{align*}
Let
\begin{align*}
\Lambda_j\equiv
&(2+2\sqrt{2})\sqrt{\sum^{r}_{i=1}\E(A_{ij}-\widetilde{A}_{ij})^2/n+\sum^{r}_{j=1}\E
(B_{ij}-\widetilde{B}_{ij})^2/n}
\\&+\sqrt{\sum^{r}_{i=1}(A_{ij}-\widetilde{A}_{ij})^2/n+\sum^{r}_{i=1}(B_{ij}-\widetilde{B}_{ij})^2/n}=\Lambda_{1j}+\Lambda_{2j},
\end{align*}
where
$$\widetilde{A}_{ij}=\sum^{iN+(i-1)M}_{l=(i-1)(N+M)+1}\widetilde{x}_{lj},\quad \widetilde{B}_{ij}=\sum^{i(N+M)}_{l=iN+(i-1)M+1}\widetilde{x}_{lj}.$$
Applying Lemma \ref{lemma:self} and using the union bound, we have
with probability at least $1-8\gamma$,
\begin{align*}
|X_j-\widetilde{X}_j|\leq \Lambda_j\sqrt{8\log(p/\gamma)},\quad
1\leq j\leq p.
\end{align*}

By the assumption,
$$P(\max_{1\leq i\leq }\max_{1\leq j\leq p}|x_{ij}|\leq M_x)\geq 1-\gamma,\quad P(\max_{1\leq i\leq }\max_{1\leq j\leq p}|y_{ij}|\leq M_y)\geq 1-\gamma.$$
Therefore with probability at least $1-\gamma,$
\begin{align*}
\Lambda_j\leq &(2+2\sqrt{2})\sqrt{\sum^{r}_{i=1}\E
(A_{ij}-\breve{A}_{ij})^2/n+\sum^{r}_{j=1}\E
(B_{ij}-\breve{B}_{ij})^2/n}
\\&+\sqrt{\sum^{r}_{i=1}(\E\breve{A}_{ij})^2/n+\sum^{r}_{i=1}(\E\breve{B}_{ij})^2/n},
\\ \leq &
(3+2\sqrt{2})\varphi(M_x)\sqrt{Nr\sigma_j^2/n+Mr\sigma^2_j/n}\lesssim
\varphi(M_x)\sigma_j,
\end{align*}
where we have used the fact that $\E A_{ij}=\E B_{ij}=0$ and the
Cauchy-Schwarz inequality. The same argument applies to the Gaussian
sequence $\{y_i\}$. 

Summarizing the above results and along with
(\ref{eq:m-dep1}), we deduce that
\begin{equation}
\begin{split}
|\E [m(X)-m(Y)]|\lesssim &
(G_2+G_1\beta)\phi(M_{x},M_y)+(G_3+G_2\beta+G_1\beta^2)\frac{(2M+1)^{2}}{\sqrt{n}}(\bar{m}_{x,3}^3+\bar{m}_{y,3}^3)
\\&+G_1\varphi(M_{x},M_y)\sigma_j\sqrt{8\log(p/\gamma)}+G_0\gamma,
\end{split}
\end{equation}
which also implies that
\begin{equation}
\begin{split}
|\E [g(T_X)-g(T_Y)]|\lesssim &
(G_2+G_1\beta)\phi(M_{x},M_y)+(G_3+G_2\beta+G_1\beta^2)\frac{(2M+1)^{2}}{\sqrt{n}}(\bar{m}_{x,3}^3+\bar{m}_{y,3}^3)
\\&+G_1\varphi(M_{xy})\sigma_j\sqrt{8\log(p/\gamma)}+G_0\gamma+\beta^{-1}G_1\log p,
\end{split}
\end{equation}
for $M$-dependent sequence, provided that
$2\sqrt{5}\beta(6M+1)M_{xy}/\sqrt{n}<1$. Consider a ``smooth''
indicator function $g_0\in C^3(\mathbb{R}): \mathbb{R}\rightarrow
[0,1]$ such that $g_0(s)=1$ for $s\leq 0$ and $g_0(s)=0$ for $s\geq
1.$ Fix any $t\in\mathbb{R}$ and define
$g(s)=g_0(\psi(s-t-e_{\beta}))$ with $e_{\beta}=\beta^{-1}\log p$.
The conclusion follows from the proof of Corollary F.1 in
\cite{cck13} and Lemma 2.1 in \cite{cck12} regarding the
anti-concentration property for Gaussian distribution. We omit the
details to conserve the space.
\end{proof}

\begin{proof}[Proof of Theorem \ref{thm2}]
Let $\breve{x}_{ij}=x_{ij}-\widetilde{x}_{ij}$. Define
$\chi_{(l+1)k}=(x_{(1+l)k}\wedge M_x)\vee(-M_x)$ and
$\chi_{(l+1)k}^{(l-1)}=(x_{(1+l)k}^{(l-1)}\wedge M_x)\vee(-M_x)$.
Using the fact that $x_{1j}$ and $x_{(1+l)k}^{(l-1)}$ are
independent for any $1\leq j,k\leq p$ and $\E
x_{ij}=\E\breve{x}_{ij}=0$, we obtain for $l>0$,
\begin{align*}
|\E \breve{x}_{1j}x_{(l+1)k}|=&   |\E
\breve{x}_{1j}(x_{(1+l)k}-x_{(1+l)k}^{(l-1)})|
\\ \leq & (\E\breve{x}_{1j}^2)^{1/2}(\E|(x_{(1+l)k}-x_{(1+l)k}^{(l-1)})|^2)^{1/2}
\\ \leq & (\E x_{1j}^4)^{1/2}(\E|(x_{(1+l)k}-x_{(1+l)k}^{(l-1)})|^2)^{1/2}/M_x.
\end{align*}
Using the fact that the map $x\rightarrow (x\wedge M_x)\vee (-M_x)$
is lipschitz continuous, we deduce that
\begin{align*}
|\E x_{1j}\breve{x}_{(l+1)k}|= &   |\E
x_{1j}\{\breve{x}_{(1+l)k}-\breve{x}_{(1+l)k}^{(l-1)}-\E(\chi_{(1+l)k}-\chi_{(1+l)k}^{(l-1)})\}\mathbf{I}\{|x_{(l+1)k}|>M_x\text{~or~}|x_{(l+1)k}^{(l-1)}|>M_x\}|
\\ \lesssim & (\E
|x_{1j}|^3)^{1/3}(\E|(\breve{x}_{(1+l)k}-\breve{x}_{(1+l)k}^{(l-1)})|^3+\E|\chi_{(1+l)k}-\chi_{(1+l)k}^{(l-1)}|^3)^{1/3}
\\&(P(|x_{(1+l)k}|>M_x)+P(|x_{(l+1)k}^{(l-1)}|>M_x))^{1/3}
\\ \lesssim & (\E |x_{1j}|^3)^{1/3}(\E|(x_{(1+l)k}-x_{(1+l)k}^{(l-1)})|^3)^{1/3}(\E|x_{(1+l)k}|^3+\E|x_{(1+l)k}^{(l-1)}|^3)^{1/3}/M_x.
\end{align*}
Note for $l=0,$ $|\E \breve{x}_{1j}x_{(l+1)k}|\leq (\E
x_{1j}^4)^{1/2}(\E x_{1k}^2)^{1/2}/M_x$. It is not hard to show that
the above result holds if $x_{1j}$ (or $x_{(l+1)k}$) is replaced by
its $\widetilde{x}_{1j}$ (or $\widetilde{x}_{(l+1)k}$). Therefore by
(\ref{eq:cov-ts}) and the assumptions, we have
\begin{align*}
\max_{1\leq j,k\leq
p}\sum^{n}_{i=1}|\E[\dot{Z}_{ij}(t)V_{k}^{(i)}(t)]|\lesssim
(1/M_x+1/M_y).
\end{align*}
Thus we may set $\phi(M_x,M_y)=C(1/M_x+1/M_y)$ for some constant
$C>0.$ 

Next we consider $\varphi(M_x,M_y)$. By the stationarity, we
have
\begin{equation}\label{eq:varphi}
\begin{split}
\sum^{N-1}_{l=1-N}|\E\breve{x}_{1k}\breve{x}_{(1+l)k}|=&
2\sum^{N-1}_{l=1}|\E\breve{x}_{1k}\{\breve{x}_{(1+l)k}-\breve{x}_{(1+l)k}^{(l-1)}-\E(\chi_{(1+l)k}-\chi_{(1+l)k}^{(l-1)})\}
\\&\mathbf{I}\{|x_{(1+l)k}|>M_x\text{~or~}|x_{(1+l)k}^{(l-1)}|>M_x\}|+E|\breve{x}_{1k}|^2
\\ \lesssim & 2\sum^{N-1}_{l=1}(\E|\breve{x}_{1k}|^2)^{1/2}(\E|(x_{(1+l)k}-x_{(1+l)k}^{(l-1)})|^3+\E|\chi_{(1+l)k}-\chi_{(1+l)k}^{(l-1)}|^3)^{1/3}\
\\&(P(|x_{(1+l)k}|>M_x)+P(|x_{(1+l)k}^{(l-1)}|>M_x))^{1/6}+E|\breve{x}_{1k}|^2
\\ \lesssim & 2\sum^{N-1}_{l=1}(\E|x_{1k}|^4/M_x^2)^{1/2}(\E|(x_{(1+h)k}-x_{(1+h)k}^{(h-1)})|^3)^{1/3}
\\&(\E|x_{(1+l)k}|^4/M_x^4+\E|x_{(1+l)k}^{(h-1)}|^4/M_x^4)^{1/6}+E|x_{1k}|^4/M_x^2
\\ \lesssim & 1/M_x^{5/3}.
\end{split}
\end{equation}
Also note that
$(\E\breve{A}_{ij})^2/N=N(\E\chi_{1j})^2=N\{\E(\chi_{1j}-x_{1j})\}^2\leq
N(\E x_{1j}^4/M_x^3)^2$ and $(\E\breve{B}_{ij})^2/M \leq M(\E
x_{1j}^4/M_x^3)^2$. Because
$E(A_{ij}-\widetilde{A}_{ij})^2/N\lesssim 1/M_x^{5/3}$ and
$E(B_{ij}-\widetilde{B}_{ij})^2/M\lesssim 1/M_x^{5/3}$ by
(\ref{eq:varphi}), we can choose
$\varphi(M_x)=C'(1/M_{x}^{5/6}+\sqrt{N}/M_x^3)$ for some constant
$C'>0$. By the assumption that $\max_{1\leq k\leq
p}\E|\mathcal{G}_k(\dots,\epsilon_{i-1},\epsilon_{i})|^4<\infty$ and
the fact that $\E y_{ij}^2=\E x_{ij}^2$, we have $E|x_{ij}|^3\leq
(\E|\mathcal{G}_j(\dots,\epsilon_{i-1},\epsilon_{i})|^4)^{3/4}$,
$$\E|y_{ij}|^3\leq (\E|y_{ij}|^4)^{3/4}\lesssim (\E|y_{ij}|^2)^{3/2}=(\E|x_{ij}|^2)^{3/2}\leq \E|x_{ij}|^3<\infty,$$
and
$$E|y_{ij}|^4\lesssim (E|y_{ij}|^2)^2=(E|x_{ij}|^2)^2\leq E|x_{ij}|^4<\infty.$$
Using similar arguments, we can show that
$\varphi(M_y)=C''(1/M_{y}^{5/6}+\sqrt{N}/M_y^3)$ for some constant
$C''>0$. The above argument also implies that
$\bar{m}^3_{x,3}+\bar{m}^3_{y,3}<\infty.$ Thus we ignore the
constants and set $\psi=O(n^{1/8}M^{-1/2}l_n^{-3/8})$ and $M_x=M_y=u=O(n^{3/8}M^{-1/2}l_n^{-5/8}).$

Let $2\sqrt{5}\beta(6M+1)M_{xy}/\sqrt{n}=1$, that is
$\beta=O(\sqrt{n}/(uM)).$ It is straightforward to check the
following:
\begin{align*}
&(\psi^2+\psi\beta)\phi(M_x,M_y)\lesssim \psi^2/u+\psi\sqrt{n}/(u^2M)\lesssim n^{-1/8}M^{1/2}l_n^{7/8},\\
&(\psi^3+\psi^2\beta+\psi\beta^2)\frac{(2M+1)^{2}}{\sqrt{n}} \lesssim \frac{\psi^3 M^2}{\sqrt{n}}+\frac{\psi^2M}{u}+\frac{\psi\sqrt{n}}{u^2}\lesssim n^{-1/8}M^{1/2}l_n^{7/8},\\
&\psi\varphi(M_{x},M_y)\sigma_j\sqrt{8\log(p/\gamma)}\lesssim \frac{\psi l_n^{1/2}}{u^{5/6}}+\frac{\sqrt{N}\psi l_n^{1/2}}{u^3}\lesssim n^{-1/8}M^{1/2}l_n^{7/8}, \\
&(e_{\beta}+\psi^{-1})\sqrt{1\vee\log (p\psi)}\lesssim
\frac{l_n^{3/2}Mu}{\sqrt{n}}+\psi^{-1} l_n^{1/2}\lesssim
n^{-1/8}M^{1/2}l_n^{7/8}.
\end{align*}
Therefore we get
\begin{align}
\rho_n:=&\sup_{t\in\mathbb{R}}|P(T_X\leq t)-P(T_Y\leq t)|\lesssim
n^{-1/8}M^{1/2}l_n^{7/8}+\gamma.
\end{align}

Under Condition (\rmnum{1}) in Assumption \ref{assum:tail}, $\E
h(\max_{1\leq j\leq p}|x_{ij}|/\mathfrak{D}_n)\leq 1$. By Lemma 2.2
in \cite{cck13}, we have $u_x(\gamma)\lesssim
\max\{\mathfrak{D}_nh^{-1}(n/\gamma),l_n^{1/2}\}$ and
$u_y(\gamma)\lesssim l_n^{1/2}.$ Because
$n^{3/8}M^{-1/2}l_n^{-5/8}\geq
C\max\{\mathfrak{D}_nh^{-1}(n/\gamma),l_n^{1/2}\}$, we can always
choose $u=O(n^{3/8}M^{-1/2}l_n^{-5/8})$ such that
\begin{equation}
P(\max_{1\leq i\leq n}\max_{1\leq j\leq p}|x_{ij}|\leq u)\geq
1-\gamma,\quad P(\max_{1\leq i\leq n}\max_{1\leq j\leq
p}|y_{ij}|\leq u)\geq 1-\gamma.
\end{equation}
Using similar arguments, we can prove the result under Condition
(\rmnum{2}) in Assumption \ref{assum:tail}. The proof is thus
completed.
\end{proof}

The following lemma verifies condition (\ref{eq:m-dep-couple}).
\begin{lemma}\label{lemma:m-dep}
Assume that $\max_{1\leq k\leq
p}\sum^{+\infty}_{j=1}j\theta_{j,k,3}(x)<\infty.$ Then
\begin{align*}
&\sup_{M}\max_{1\leq k\leq p}\sum^{M}_{l=1}(\E|x_{(1+l)k}^{(M)}
-x_{(1+l)k}^{(l-1)}|^3)^{1/3}\leq \max_{1\leq k\leq
p}\sum^{+\infty}_{j=1}j\theta_{j,k,3}(x)<\infty.
\end{align*}
\end{lemma}
\begin{proof}[Proof of Lemma \ref{lemma:m-dep}]
Define the projection
$\mathcal{P}_jx_{ik}=\E[x_{ik}|\mathcal{F}_{j}(i)]-\E[x_{ik}|\mathcal{F}_{j-1}(i)]$.
Then we have
\begin{align*}
x_{(1+l)k}^{(M)}
-x_{(1+l)k}^{(l-1)}=&\E[\mathcal{G}_k(\dots,\epsilon_{l},\epsilon_{l+1})|\mathcal{F}_{M}(l+1)]-\E[\mathcal{G}_k(\dots,\epsilon_{l},\epsilon_{l+1})|\mathcal{F}_{l-1}(l+1)]
=\sum^{M}_{j=l}\mathcal{P}_jx_{(l+1)k}.
\end{align*}
Note that
\begin{align*}
\mathcal{P}_jx_{ik}=&\E[x_{ik}|\mathcal{F}_{j}(i)]-\E[x_{ik}|\mathcal{F}_{j-1}(i)]
\\=&\E[\mathcal{G}_k(\dots,\epsilon_{i-1},\epsilon_i)-\mathcal{G}_k(\dots,\epsilon_{i-j}',\epsilon_{i-j+1},\dots,\epsilon_{i-1},\epsilon_i)|\mathcal{F}_j(i)]
\\=&\E[\mathcal{G}_k(\dots,\epsilon_{j-1},\epsilon_j)-\mathcal{G}_k(\dots,\epsilon_{0}',\epsilon_{1},\dots,\epsilon_{j-1},\epsilon_j)|\mathcal{F}_j(j)].
\end{align*}
Jensen's inequality yields that
$(\E|\mathcal{P}_{j}x_{ik}|^q)^{1/q}\leq \theta_{j,k,q}(x)$ which
implies that
$$(\E|x_{(1+l)k}^{(M)} -x_{(1+l)k}^{(l-1)}|^3)^{1/3}\leq \sum^{M}_{j=l}(\E|\mathcal{P}_{j}x_{(l+1)k}|^3)^{1/3}\leq \sum^{M}_{j=l}\theta_{j,k,3}(x).$$
Therefore, we obtain
\begin{align*}
\sup_{M}\max_{1\leq k\leq p}\sum^{M}_{l=1}(\E|x_{(1+l)k}^{(M)}
-x_{(1+l)k}^{(l-1)}|^3)^{1/3}\leq& \sup_{M}\max_{1\leq k\leq
p}\sum^{M}_{l=1}\sum^{M}_{j=l}\theta_{j,k,3}(x)\leq \max_{1\leq
k\leq p}\sum^{+\infty}_{j=1}j\theta_{j,k,3}(x)<\infty.
\end{align*}
\end{proof}

\begin{proof}[Proof of Theorem \ref{thm:Gaussian-dep}]
We need to verify that the $M$-dependent approximation
$\{x_i^{(M)}\}$ satisfies the assumptions in Theorem \ref{thm2}.
Using the convexity of $h(\cdot)$ and Jensen's inequality we have
$$\E h(\max_{1\leq j\leq p}|x_{ij}^{(M)}|/\mathfrak{D}_n)\leq \E h(\max_{1\leq j\leq p}|x_{ij}|/\mathfrak{D}_n)\leq 1,$$
under Condition (\rmnum{1}) in Assumption \ref{assum:tail}, and
$$\max_{1\leq j\leq p}\E \exp(|x_{ij}^{(M)}|/\mathfrak{D}_n)\leq \max_{1\leq j\leq p}\E \exp(|x_{ij}|/\mathfrak{D}_n)\leq 1,$$
under Condition (\rmnum{2}) in Assumption \ref{assum:tail}. 

 We claim
that as $M\rightarrow +\infty,$
\begin{equation}\label{eq:thm:Gaussian-dep}
\sup_p\max_{1\leq j\leq p}\sum^{+\infty}_{h=-\infty}|\E
x_{ij}^{(M)}x_{(i+h)j}^{(M)}-\E x_{ij}x_{(i+h)j}|\rightarrow 0,
\end{equation}
which implies that $\max_{1\leq j\leq
p}|\sigma_{j,j}^{(M,n)}-\sigma_{j,j}^{(n)}|\rightarrow 0$ and
$\max_{1\leq j\leq p}|(\sigma_j^{(M)})^2-\sigma_j^2|\rightarrow 0$
with $\sigma_{j,j}^{(M,n)}=\sum^{n-1}_{h=1-n}(n-|h|)\E
x_{ij}^{(M)}x_{(i+h)j}^{(M)}/n$ and
$(\sigma_j^{(M)})^2=\sum^{+\infty}_{h=-\infty}|\E
x_{ij}^{(M)}x_{(i+h)j}^{(M)}|$. Thus under the assumptions in
Theorem \ref{thm:Gaussian-dep}, we have $c_1/2<\min_{1\leq j\leq
p}\sigma_{j,j}^{(M,n)}\leq \max_{1\leq j\leq
p}\sigma_{j,j}^{(M,n)}<2c_2$ for some constants $0<c_1<c_2$
uniformly for all large enough $M$. 

To show
(\ref{eq:thm:Gaussian-dep}), we note that
\begin{align*}
&\sum^{+\infty}_{h=-\infty}|\E x_{ij}^{(M)}x_{(i+h)j}^{(M)}-\E
x_{ij}x_{(i+h)j}| \\=&\sum^{M}_{h=-M}|\E
x_{ij}^{(M)}x_{(i+h)j}^{(M)}-\E x_{ij}x_{(i+h)j}|+\sum_{|h|>M}|\E
x_{ij}x_{(i+h)j}|=I_{1j}(M)+I_{2j}(M).
\end{align*}
For the first term, we have
\begin{align*}
&I_{1j}(M)\leq \sum^{M}_{h=-M}|\E
x_{ij}^{(M)}(x_{(i+h)j}^{(M)}-x_{(i+h)j})|+\sum^{M}_{h=-M}|\E
(x_{ij}^{(M)}-\E x_{ij})x_{(i+h)j}|
\\ \leq& \sum^{M}_{h=-M}\{\E
(x_{ij}^{(M)})^2\}^{1/2}\{\E(x_{(i+h)j}^{(M)}-x_{(i+h)j})^2\}^{1/2}+\sum^{M}_{h=-M}\{\E
(x_{ij}^{(M)}-\E x_{ij})^2\}^{1/2}\{\E(x_{(i+h)j})^2\}^{1/2}
\\ \lesssim &
M\{\E(x_{1j})^2\}^{1/2}\{\E(x_{1j}^{(M)}-x_{1j})^2\}^{1/2}\leq
M\{\E(x_{1j})^2\}^{1/2}\sum^{+\infty}_{l=M+1}(\E|\mathcal{P}_lx_{1j}|^2)^{1/2}
\\ \leq & \{\E(x_{1j})^2\}^{1/2}\sum^{+\infty}_{l=M+1}l\theta_{l,j,2}(x)\leq
\{\E(x_{1j})^2\}^{1/2}\sum^{+\infty}_{l=M+1}l\theta_{l,j,3}(x),
\end{align*}
where we have used the fact that
$x_{1j}-x_{1j}^{(M)}=\sum^{+\infty}_{l=M+1}\mathcal{P}_lx_{1j}$ and
$(\E|\mathcal{P}_{l}x_{1j}|^q)^{1/q}\leq \theta_{l,j,q}(x)$. Under
the assumption that $\sum^{+\infty}_{j=1}\max_{1\leq k\leq
p}j\theta_{j,k,3}(x)\leq \sum^{+\infty}_{j=1}\ell_j<\infty$, we have
$\max_{1\leq j\leq p}I_{1j}(M)\rightarrow 0$ as
$M\rightarrow+\infty.$ On the other hand, note that for $h>M$
\begin{align*}
\E x_{ij}x_{(i+h)j}=\E x_{ij}(x_{(i+h)j}-x_{(i+h)j}^{(h-1)})\leq (\E
x_{ij}^2)^{1/2}\{\E(x_{(i+h)j}-x_{(i+h)j}^{(h-1)})^2\}^{1/2}.
\end{align*}
Thus we have
\begin{align*}
\max_{1\leq j\leq p}I_{2j}(M)\lesssim \max_{1\leq j\leq
p}\sum_{h>M}\sum_{l\geq h}\theta_{l,j,3}(x)\leq \max_{1\leq j\leq
p}\sum_{l=M+1}^{+\infty}l\theta_{l,j,3}(x)\leq
\sum^{+\infty}_{l=M+1}\ell_l, \end{align*} which implies that
$\max_{1\leq j\leq p}I_{2j}(M)\rightarrow 0$ as $M\rightarrow
+\infty.$ 

Lemma \ref{lemma:m-dep} verifies the first condition in
(\ref{eq:condition}). The same arguments apply to $\{y_i\}$. The
triangle inequality and (\ref{eq:m-approx}) imply that
\begin{align*}
|\E[m(X)-m(Y)]|\lesssim &
|\E[m(X^{(M)})-m(Y^{(M)})]|+(G_0G_1^q)^{1/(1+q)}\left(\sum^{p}_{j=1}\Theta_{M,j,q}^q\right)^{1/(1+q)},
\end{align*}
where $Y^{(M)}=\sum^{n}_{i=1}y_i^{(M)}/\sqrt{n}$ with $y_i^{(M)}$
being the $M$-dependent approximation for $\{y_i\}$. The conclusion
thus follows from Theorem \ref{thm1} and Theorem \ref{thm2}.
\end{proof}

\begin{proof}[Proof of Proposition \ref{prop:dim-free}]
Without loss of generality, we assume that $\pi(i)=i.$ Define two
events $\mathcal{D}_x=\{\max_{1\leq j\leq q}X_j>\max_{q+1\leq j\leq
p}X_j\}$ and $\mathcal{D}_y=\{\max_{1\leq j\leq q}Y_j>\max_{q+1\leq
j\leq p}Y_j\}$. Simple algebra yields that uniformly for all
$z\in\mathbb{R}$,
\begin{align*}
& \left|P\left(\max_{1\leq j\leq p}X_j\leq
z\right)-P\left(\max_{q+1\leq j\leq p}X_j\leq z\right)\right|
\\ \leq &
\left|P\left(\max_{1\leq j\leq p}X_j\leq
z,\mathcal{D}_x\right)+P\left(\max_{1\leq j\leq p}X_j\leq
z,\mathcal{D}^c_x\right)-P\left(\max_{q+1\leq j\leq p}X_j\leq
z\right)\right|
\\ \leq & \left|P\left(\max_{1\leq j\leq q}X_j\leq z,\mathcal{D}_x\right)+P\left(\max_{q+1\leq j\leq p}X_j\leq z,\mathcal{D}^c_x\right)-P\left(\max_{q+1\leq j\leq p}X_j\leq z\right)\right|
\\ \leq & \left|P\left(\max_{1\leq j\leq q}X_j\leq z,\mathcal{D}_x\right)-P\left(\max_{q+1\leq j\leq p}X_j\leq z,\mathcal{D}_x\right)\right|\leq 2P(\mathcal{D}_x).
\end{align*}

Next we analyze $P(\mathcal{D}_x)$ and $P(\mathcal{D}_y)$. Under the
assumptions in Corollary 2.1 of \cite{cck13}, we have
\begin{align}\label{eq:dim-free1}
\sup_{z\in\mathbb{R}}\left|P\left(\max_{q+1\leq i\leq p}X_i\leq
z\right)-P\left(\max_{q+1\leq i\leq p}Y_i\leq
z\right)\right|\lesssim n^{-c},\quad c>0.
\end{align}
Notice that in this case, we allow $p=O(\exp(n^b))$ with $b<1/7$
(assuming that $B_n=O(1)$ in Corollary 2.1 of \cite{cck13}). By
(\ref{eq:dim-free1}), and the independence between $\{z_{i1}\}$ and
$\{z_{i2}\}$, we obtain
\begin{align*}
P(\mathcal{D}_x)\leq& \sum^{q}_{j=1}\E\left[P\left(\max_{q+1\leq
i\leq p}X_i<X_j\bigg|X_j\right)\right] \lesssim
\sum^{q}_{j=1}\E\left[P_y\left(\max_{q+1\leq i\leq
p}Y_i<X_j\right)\right]+qn^{-c},
\end{align*}
where $P_y$ denotes the probability measure with respect to
$(Y_{q+1},\dots,Y_{p})$. 

Let $\bar{\sigma}=\max_{1\leq j\leq
p}\sigma_{j,j}$. Using the concentration inequality (see e.g. (7.3)
of \cite{l2001} and Theorem A.2.1 of \cite{vw96}),
$$P\left(\max_{q+1\leq i\leq p}Y_i\leq \E \max_{q+1\leq i\leq p}Y_i-r \right)\leq e^{-r^2/(2\bar{\sigma})},$$
for $r>0$, we have
\begin{align*}
P\left(\max_{q+1\leq i\leq p}Y_i<x\right)\leq
\exp\left(-\frac{1}{2\bar{\sigma}}\left(\E \max_{q+1\leq i\leq
p}Y_i-x\right)^2_+\right),
\end{align*}
where $x_+=x\mathbf{I}\{x\geq 0\}$. Under the assumption that $q/ \E
\max_{q+1\leq i\leq p}Y_i \rightarrow 0,$ we can choose
$\widetilde{q}\rightarrow+\infty$ such that $\widetilde{q}/\E
\max_{q+1\leq i\leq p}Y_i \rightarrow 0$ and
$q/\widetilde{q}\rightarrow 0$. Then we have
\begin{align*}
& \sum^{q}_{j=1}\E\left[P_y\left(\max_{q+1\leq i\leq
p}Y_i<X_j\right)\right]\leq
\sum^{q}_{j=1}\E\exp\left(-\frac{1}{2\bar{\sigma}}\left(\E
\max_{q+1\leq i\leq p}Y_i-X_j\right)^2_+\right)
\\ \leq & \sum^{q}_{j=1}\E\exp\left(-\frac{1}{2\bar{\sigma}}\left(\E \max_{q+1\leq i\leq p}Y_i-X_j\right)^2_+\right)\mathbf{I}\{X_j\leq \widetilde{q}\}+
\sum^{q}_{j=1}\E\mathbf{I}\{X_j>\widetilde{q}\}
\\ \leq & \exp\left(\log q-\frac{1}{2\bar{\sigma}}\left(\E \max_{q+1\leq i\leq p}Y_i-\widetilde{q}\right)^2_+\right)+q\max_{1\leq j\leq q}\E|X_j|/\widetilde{q}=o(1).
\end{align*}
Moreover, if $q/\E \max_{q+1\leq j\leq p}Y_j=O(n^{-c'})$ for $c'>0$,
we can replace $o(1)$ by $O(n^{-c''})$ for some $c''>0.$ Thus we get
\begin{align*}
&\sup_{z\in\mathbb{R}}\left|P\left(\max_{1\leq j\leq p}X_j\leq
z\right)-P\left(\max_{q+1\leq j\leq p}X_j\leq z\right)\right|\leq
2P(\mathcal{D}_x)
\\ \lesssim & \sum^{q}_{j=1}\E\left[P_y\left(\max_{q+1\leq i\leq p}Y_i<X_j\right)\right]+qn^{-c} \lesssim n^{-c''}.
\end{align*}
Similar argument applies to $\{Y_i\}$ and the conclusion follows
from (\ref{eq:dim-free1}).
\end{proof}

\subsection{Proofs of the main results in Section \ref{sec:boot}}
\begin{proof}[Proof of Lemma \ref{lemma:m-dep-boot}]
By the triangle inequality and the stationarity, we have
\begin{align*}
E_{A}\leq& \max_{1\leq j,k\leq
p}\left|\frac{1}{Nr}\sum^{r}_{i=1}(A_{ij}A_{ik}-\E
A_{ij}A_{ik})\right|+\max_{1\leq j,k\leq
p}\left|\frac{1}{Nr}\sum^{r}_{i=1}\E
A_{ij}A_{ik}-\sigma_{j,k}\right| +\max_{1\leq j,k\leq
p}|\sigma_{j,k}-\sigma_{j,k}^{(n)}|
\\ \leq & \max_{1\leq j,k\leq
p}\left|\frac{1}{Nr}\sum^{r}_{i=1}(A_{ij}A_{ik}-\E
A_{ij}A_{ik})\right|+\max_{1\leq j,k\leq p}\left|\sum_{|l|\geq N}\E
x_{i+l,j}x_{i,k}+\frac{1}{N}\sum^{N-1}_{l=1-N}|l|\E
x_{i+l,j}x_{i,k}\right|
\\ & +\max_{1\leq j,k\leq p}|\sigma_{j,k}-\sigma_{j,k}^{(n)}|
\\ \leq& \max_{1\leq j,k\leq
p}\left|\frac{1}{Nr}\sum^{r}_{i=1}(A_{ij}A_{ik}-\E
A_{ij}A_{ik})\right|+\frac{4}{N}\max_{1\leq j,k\leq
p}\sum^{+\infty}_{l=-\infty}|l||\E x_{i+l,j}x_{i,k}|.
\end{align*}

Note that for any $1\leq j,k\leq p,$ $\{A_{ij}A_{ik}\}$ is a
sequence of i.i.d random variables. Let $\sigma_{A,N}^2=\max_{1\leq
j,k\leq p}\E (A_{ij}A_{ik})^2/N^2$ and
$\mathcal{M}_{A,N}=\max_{1\leq i\leq r}\max_{1\leq j\leq
p}|A_{ij}/\sqrt{N}|^4$. Then by Lemma A.1 in \cite{cck13}, we have
\begin{align*}
\E \max_{1\leq j,k\leq
p}\left|\frac{1}{Nr}\sum^{r}_{i=1}(A_{ij}A_{ik}-\E
A_{ij}A_{ik})\right|\lesssim \sigma_{A,N}\sqrt{2\log p/r}+2\log
p\sqrt{\E \mathcal{M}_{A,N}}/r.
\end{align*}
Cauchy-Schwarz inequality yields that
\begin{align*}
\sigma^2_{A,N}\leq& \frac{1}{N^2}\max_{1\leq j\leq p}\E (A_{ij})^4
\leq \frac{1}{N^2}\max_{1\leq j\leq p}\E
\sum_{i_1,i_2,i_3,i_4=1}^Nx_{i_1j}x_{i_2j}x_{i_3j}x_{i_4j}
\\ \leq& \frac{1}{N^2}\max_{1\leq j\leq p}\sum_{i_1,i_2,i_3,i_4=1}^N\bigg\{\text{cum}(x_{i_1j},x_{i_2j},x_{i_3j},x_{i_4j})+\gamma_{x,jj}(i_1-i_3)\gamma_{x,jj}(i_2-i_4)
\\&+\gamma_{x,jj}(i_1-i_2)\gamma_{x,jj}(i_3-i_4)+\gamma_{x,jj}(i_1-i_4)\gamma_{x,jj}(i_2-i_3)\bigg\}
\\ \leq& \max_{1\leq j\leq p}\left\{\frac{1}{N}\sum_{i_1,i_2,i_3=-\infty}^{+\infty}|\text{cum}(x_{i_1j},x_{i_2j},x_{i_3j},x_{0j})|+3\left(\sum^{+\infty}_{h=-\infty}|\gamma_{x,jj}(h)|\right)^2\right\}\lesssim \bar{\sigma}^2_{x,N}.
\end{align*}

On the other hand, with $h(x)=\exp(x)-1$, we have
\begin{equation}\label{eq:block3}
\begin{split}
\left(\E \max_{1\leq i\leq r}\max_{1\leq j\leq
p}|A_{ij}/\sqrt{N}|^4\right)^{1/4}\lesssim& \left|\left|\max_{1\leq
i\leq r}\max_{1\leq j\leq p}|A_{ij}/\sqrt{N}|\right|\right|_{h}
\\ \lesssim& \log(rp)\max_{1\leq
i\leq r}\max_{1\leq j\leq p}||A_{ij}/\sqrt{N}||_{h}
\\ =& \log(rp)\max_{1\leq j\leq p}||A_{1j}/\sqrt{N}||_{h},
\end{split}
\end{equation}
where we have used Lemma 2.2.2 in \cite{vw96}. It implies that
\begin{align*}
\sqrt{\E \mathcal{M}_{A,N}}\leq \{\log(rp)\}^2\max_{1\leq j\leq
p}\left|\left|\sum^{N}_{i=1}x_{ij}/\sqrt{N}\right|\right|_{h}^2.
\end{align*}

Combining the above arguments, we deduce that
\begin{align*}
&\E E_{A}\lesssim \bar{\sigma}_{x,N}\sqrt{\log p/r}+\log
p\{\log(rp)\}^2\zeta_{x,h,N}^2/r+\varpi_x/N,\\
&\E E_{B}\lesssim \bar{\sigma}_{x,M}\sqrt{\log p/r}+\log
p\{\log(rp)\}^2\zeta_{x,h,M}^2/r+\varpi_x/M.
\end{align*}
Alternatively, note that $\left(\E \max_{1\leq i\leq r}\max_{1\leq
j\leq p}|A_{ij}/\sqrt{N}|^4\right)^{1/4}\leq
r^{1/4}\varsigma_{x,N}.$ The conclusion follows from the above
arguments.
\end{proof}

\begin{proof}[Proof of Theorem \ref{lemma:m-dep-boot}]
By Theorem \ref{thm2}, $\rho_n\lesssim
n^{-1/8}M^{1/2}l_n^{7/8}+\gamma.$ Choosing $\gamma=O(n^{-c'})$ for
some $c'>(1-4b'-7b)/8,$ we have $\rho_n=O(n^{-(1-4b'-7b)/8})$. Pick
$\nu=O(n^{-v})$ with
$$v=3((1-5b-b'')/2-s_1)\wedge(1-5b-b''-s_2)\wedge(b'-2b-s_3)/4+2b.$$
Then it is easy to verify that the terms $\nu^{1/3}(1\vee
\log(p/\nu))^{2/3}$ and $\E E_{A}/\nu+\E E_{B}/\nu$ are both of
order $O(n^{-c''})$ with $c''=s_b/4$. Finally by (\ref{eq:boot}), we
have
$$\sup_{\alpha\in(0,1)}\left|P(T_X\leq
c_{T_D}(\alpha))-\alpha\right|\lesssim n^{-c},\quad
c=\min\{s_b/4,(1-4b'-7b)/8\}.$$ The result under Condition 2 can be
proved in a similar manner.
\end{proof}

\begin{proof}[Proof of Theorem \ref{thm:dep-boot}]
Let $x_i^{(M)}$ be the $M$-dependent approximation sequence for
$x_i$. Define $A_{ij}^{(M)}$, $B_{ij}^{(M)}$, $E_A^{(M)}$ and
$E_B^{(M)}$ in a similar way as $A_{ij}$, $B_{ij}$, $E_A$ and $E_B$
by replacing $x_i$ with $x_i^{(M)}$. Notice that
\begin{align*}
&\E\max_{1\leq j,k\leq
p}\left|\frac{1}{r}\sum^{r}_{i=1}(A_{ij}^{(M)}A_{ik}^{(M)}-A_{ij}A_{ik})/N\right|
\\ \leq& \frac{1}{rN}\sum_{1\leq j,k\leq p}\E\left|\sum^{r}_{i=1}(A_{ij}^{(M)}A_{ik}^{(M)}-A_{ij}A_{ik}^{(M)}+A_{ij}A_{ik}^{(M)}-A_{ij}A_{ik})\right|
\\ \leq& \frac{1}{N}\sum_{1\leq j,k\leq p}\left(\E\left|A_{1j}^{(M)}A_{1k}^{(M)}-A_{1j}A_{1k}^{(M)}\right|+\E\left|A_{1j}A_{1k}^{(M)}-A_{1j}A_{1k}\right|\right)
\\ \leq& \frac{1}{N}\sum_{1\leq j,k\leq
p}\left\{\left(\E\left|A_{1j}^{(M)}-A_{1j}\right|^2\right)^{1/2}(\E|A_{1k}^{(M)}|^2)^{1/2}+\left(\E\left|A_{1k}^{(M)}-A_{1k}\right|^2\right)^{1/2}(\E|A_{1j}|^2)^{1/2}\right\}.
\end{align*}

By Lemma A.1 of \cite{ll2009}, we have
$(\E|A_{1j}^{(M)}-A_{1j}|^2)^{1/2}/\sqrt{N}\leq
C_q\Theta_{M,j,q}(x)$ for some $q\geq 2$. It follows that
\begin{align*}
&\E\max_{1\leq j,k\leq
p}\left|\frac{1}{r}\sum^{r}_{i=1}(A_{ij}^{(M)}A_{ik}^{(M)}-A_{ij}A_{ik})/N\right|\lesssim
p^2\rho^M.
\end{align*}
Similarly we have
$$\E\max_{1\leq j,k\leq p}\left|\frac{1}{r}\sum^{r}_{i=1}(B_{ij}^{(M)}B_{ik}^{(M)}-B_{ij}B_{ik})/M\right|\lesssim p^2\rho^M.$$
Using similar arguments in the proof of Theorem
\ref{thm:Gaussian-dep}, we have
\begin{align*}
\max_{1\leq j,k\leq p}\sum^{n-1}_{h=1-n}\left|\E
x_{ij}^{(M)}x_{(i+h)k}^{(M)}-\E x_{ij}x_{(i+h)k}\right|\lesssim
M\rho^M.
\end{align*}
Thus by (\ref{eq:boot}), we have
\begin{align*}
\sup_{\alpha\in(0,1)}\left|P(T_X\leq
c_{T_D}(\alpha))-\alpha\right|\lesssim &\rho_n+\nu^{1/3}(1\vee
\log(p/\nu))^{2/3}+\E E_{A}^{(M)}/\nu+\E
E_{B}^{(M)}/\nu+(p^2+M)\rho^M/\nu.
\end{align*}

Then by Lemma A.1 in \cite{cck13}, we have
\begin{align*}
&\E \max_{1\leq j,k\leq
p}\left|\frac{1}{Nr}\sum^{r}_{i=1}(A_{ij}^{(M)}A_{ik}^{(M)}-\E
A_{ij}^{(M)}A_{ik}^{(M)})\right|
\\ \lesssim &
\frac{1}{N}\max_{1\leq j\leq p}\{\E (A_{ij}^{(M)})^4\}^{1/2}
\sqrt{2\log p/r}+2\log p\sqrt{\E \max_{1\leq i\leq r}\max_{1\leq
j\leq p}|A_{ij}^{(M)}/\sqrt{N}|^4}/r
\\ \lesssim & \frac{1}{N}\max_{1\leq
j\leq p}\{\E (A_{ij})^4\}^{1/2} \sqrt{2\log p/r}+2\log p\sqrt{\E
\max_{1\leq i\leq r}\max_{1\leq j\leq p}|A_{ij}/\sqrt{N}|^4}/r
\\ &+\sqrt{2\log p/r}\max_{1\leq j\leq p}\Theta_{M,j,4}^2(x)
+2\log p\sqrt{rp\max_{1\leq j\leq p}\Theta_{M,j,4}^4(x)}/r,
\end{align*}
where the first two terms can be bounded using similar arguments in
the proof of Lemma \ref{lemma:m-dep-boot}, and the last two terms
decay exponentially. The same arguments apply to the terms
associated with $B_{ij}$. 

By Theorem \ref{thm:Gaussian-dep}, we have
$$\rho_n\lesssim
n^{-1/8}M^{1/2}l_n^{7/8}+\gamma+(n^{1/8}M^{-1/2}l_n^{-3/8})^{q/(1+q)}\left(\sum^{p}_{j=1}\Theta_{M,j,q}^q\right)^{1/(1+q)}.$$
The assumption that $\max_{1\leq j\leq p}\Theta_{M,j,q}=O(\rho^M)$
for $\rho<1$, and $M=O(n^{b'})$ with $b'>2b$ implies that
$\left(\sum^{p}_{j=1}\Theta_{M,j,q}^q\right)^{1/(1+q)}$ decays
exponentially. The rest of the proof is similar to those in the
proof of Theorem \ref{lemma:m-dep-boot}. \end{proof}

\begin{proof}[Proof of Theorem \ref{thm:block-boot}]
Our arguments below apply to $M$-dependent time series, and can be
easily extended to weakly dependent time series by employing the
$M$-approximation techniques (that incurs only an asymptotically
ignorable error). 

Let $c$, $c^*$, and $C$ be some generic constants
which can be different from line to line. Define
$$\breve{T}_{\widetilde{X}}=\max_{1\leq j\leq p}\frac{1}{\sqrt{n}}\sum^{l_n}_{i=1}(\mathcal{A}_{ij}-\bar{\mathcal{A}}_j)e_i.$$
Following the arguments in the proof of Lemma
\ref{lemma:m-dep-boot}, we have
\begin{align*}
\E\max_{1\leq j\leq
p}\left|\frac{1}{l_n}\sum^{l_n}_{i=1}(\mathcal{A}_{ij}/\sqrt{b_n})^2-\sigma_{j,j}^{(b_n)}\right|\lesssim&
\bar{\sigma}_{x,b_n}\sqrt{\log p/l_n}+\log
p\{\log(l_np)\}^2\zeta_{x,h,b_n}^2/l_n\leq Cn^{-c}.
\end{align*}
Similarly we can show that
\begin{align*}
\E\max_{1\leq j\leq
p}\left|\frac{1}{l_n}\sum^{l_n}_{i=1}\mathcal{A}_{ij}/\sqrt{b_n}\right|\lesssim
\max_{1\leq j\leq p}\sigma_{j}\sqrt{\log p/l_n}+\log
p\{\log(l_np)\}\zeta_{x,h,b_n}/l_n\leq Cn^{-c}
\end{align*}
where we have used the fact that
\begin{equation}\label{eq:block-boot}
\E\max_{1\leq i\leq l_n}\max_{1\leq j\leq
p}|\mathcal{A}_{ij}/\sqrt{b_n}|\lesssim \log(l_np)\zeta_{x,h,b_n}.
\end{equation}

By Markov's inequality, we have with probability $1-Cn^{-c}$,
$$\left|\frac{1}{n}\sum^{l_n}_{i=1}(\mathcal{A}_{ij}-\bar{\mathcal{A}}_j)^2-\sigma_{j,j}^{(b_n)}\right|\leq (c_1/2)\wedge c_2,$$
uniformly for $1\leq j\leq p.$ It implies that with probability
$1-Cn^{-c}$, $c_1/2\leq
\frac{1}{n}\sum^{l_n}_{i=1}(\mathcal{A}_{ij}-\bar{\mathcal{A}}_j)^2\leq
2c_2$. By (\ref{eq:block-boot}), we have with probability with
$1-Cn^{-c}$, $\max_{1\leq i\leq l_n}\max_{1\leq j\leq
p}|\mathcal{A}_{ij}/\sqrt{b_n}|\leq
n^{c^*}\log(l_np)\zeta_{x,h,b_n}$ for some small $c^*>0.$ Because
$\zeta_{x,h,b_n}^2\{\log(pl_n)\}^9/l_n\lesssim n^{-c_0'}$, we can
apply Corollary 2.1 in \cite{cck13} to conclude that with
probability $1-Cn^{-c}$,
\begin{align}\label{eq:block1}
\sup_{t\in\mathbb{R}}|P(T_{X^*}\leq
t|\{x_i\}^{n}_{i=1})-P(\breve{T}_{\widetilde{X}}\leq
t|\{x_i\}^{n}_{i=1})|\lesssim n^{-c'},\quad c'>0.
\end{align}

Next, notice that
\begin{align*}
|\breve{T}_{\widetilde{X}}-T_{\widetilde{X}}|\leq \max_{1\leq j\leq
p}|\bar{\mathcal{A}}_j/\sqrt{b_n}|\left|\frac{1}{\sqrt{l_n}}\sum^{l_n}_{i=1}e_i\right|.
\end{align*}
With probability $1-Cn^{-c}$, we have
$|\breve{T}_{\widetilde{X}}-T_{\widetilde{X}}|\leq n^{c^*}\sqrt{\log
p/l_n}\left|\frac{1}{\sqrt{l_n}}\sum^{l_n}_{i=1}e_i\right|$. Using
the tail property of standard normal distribution, we can choose
$\zeta=n^{2c^*}\sqrt{\log p/l_n}$ such that with probability
$1-o(1)$,
\begin{align*}
P(|\breve{T}_{\widetilde{X}}-T_{\widetilde{X}}|>\zeta|\{x_i\}^{n}_{i=1})\lesssim
n^{-c'},
\end{align*}
and $\sqrt{\log p}\zeta\lesssim n^{-c'}$ for some properly chosen
$c^*$ and $c'$. Therefore by Lemma 2.1 in \cite{cck13}, we obtain
that with probability $1-Cn^{-c}$,
\begin{align}\label{eq:block2}
\sup_{t\in\mathbb{R}}|P(T_{\widetilde{X}}\leq
t|\{x_i\}^{n}_{i=1})-P(\breve{T}_{\widetilde{X}}\leq
t|\{x_i\}^{n}_{i=1})|\lesssim n^{-c'}.
\end{align}
By (\ref{eq:block1}) and (\ref{eq:block2}), (\ref{eq:equi}) holds
with probability $1-Cn^{-c}$. The second part of the theorem follows
from Theorem \ref{thm:m-dep-boot} and Theorem \ref{thm:dep-boot}.
\end{proof}

\subsection{Proofs of the main results in Section \ref{sec:ts}}
\begin{proof}[Proof of Theorem \ref{thm:app-1}]
Define $T_{D}=\max_{1\leq j\leq
2q_0}\frac{1}{\sqrt{n}}\sum^{r_1}_{i=1}D_{ij},$ where
$D_{ij}=A_{ij}e_i+B_{ij}\widetilde{e}_i$ and
\begin{equation*}
A_{ij}=\sum^{iN_1+(i-1)M_1}_{l=(i-1)(N_1+M_1)+1}x_{lj},\quad
B_{ij}=\sum^{i(N_1+M_1)}_{l=iN_1+(i-1)M_1+1}x_{lj},\quad 1\leq i\leq
r_1,1\leq j\leq 2q_0.
\end{equation*}
Since $\max_{1\leq j\leq
q_0}\left|\sum^{N_0}_{i=1}IF(v_i,F_{d_0})\right|/\sqrt{N_0}=\max_{1\leq
j\leq 2q_0}\sum^{N_0}_{i=1}x_i/\sqrt{N_0},$ we have
$$\left|\max_{1\leq j\leq
q_0}\sqrt{N_0}|\widehat{\theta}_{j}-\widetilde{\theta}_j|-\max_{1\leq
j\leq 2q_0}\sum^{N_0}_{i=1}x_i/\sqrt{N_0}\right|\leq \max_{1\leq
j\leq q_0}\sqrt{N_0}|\mathcal{R}_{jN_0}|.$$ Let
$\zeta_1=Cn^{-c}/\sqrt{\log(2q_0)}$ and $\zeta_2=Cn^{-c}$ for some
large enough $C$ and small enough $c$ (e.g. $c<c_1$) such that
$$P(\max_{1\leq j\leq q_0}\sqrt{N_0}|\mathcal{R}_{jN_0}|>\zeta_1)<\zeta_2.$$
We show that
$P(P(|T_D-T_{\widehat{D}}|>\zeta_1|\{x_i\}^{n}_{i=1})>\zeta_2)\leq
\zeta_2.$ Because $|T_D-T_{\widehat{D}}|\leq \max_{1\leq j\leq
2q_0}\left|\frac{1}{\sqrt{n}}\sum^{r}_{i=1}(D_{ij}-\widehat{D}_{ij})\right|$
and
$$\frac{1}{\sqrt{n}}\sum^{r}_{i=1}(D_{ij}-\widehat{D}_{ij})\sim
N\left(0,\frac{1}{n}\sum^{r}_{i=1}\{(A_{ij}-\widehat{A}_{ij})^2+(B_{ij}-\widehat{B}_{ij})^2\}\right)$$
conditional on $\{x_i\}^{2q_0}_{i=1}$, we have
$\E[|T_D-T_{\widehat{D}}||\{x_i\}^{n}_{i=1}]\leq
C'\sqrt{\mathcal{E}_{AB}\log(2q_0)}$ for some large enough constant
$C'.$ It thus implies that
\begin{align*}
P(P(|T_D-T_{\widehat{D}}|>\zeta_1|\{x_i\}^{n}_{i=1})>\zeta_2)\leq&
P(\E[|T_D-T_{\widehat{D}}||\{x_i\}^{n}_{i=1}]>\zeta_1\zeta_2)
\\ \leq & P(\mathcal{E}_{AB}\{C'\log(2q_0)\}^2>C^4n^{-4c})\leq
Cn^{-c},
\end{align*}
for large enough $C$ and sufficiently small $c$ (e.g. $c<c_2/4$). By
Theorem \ref{thm:dep-boot}, and Lemma 3.3 and the arguments in the
proof of Theorem 3.2 in \cite{cck13}, we derive that under $H_0,$
\begin{align*}
\sup_{\alpha\in (0,1)}|P(\max_{1\leq j\leq
q_0}\sqrt{N_0}|\widehat{\theta}_{j}-\widetilde{\theta}_j|>c_{1}(\alpha))-\alpha|\lesssim&n^{-\widetilde{c}}+\zeta_1\sqrt{1\vee
\log(2q_0/\zeta_1)}+\zeta_2\lesssim n^{-c''},
\end{align*}
for $c''>0$, where $\widetilde{c}=c$ or $c'$, which are defined in
Theorem \ref{thm:dep-boot}.
\end{proof}

\begin{proof}[Proof of Theorem \ref{thm:app-2}]
Note that
$$\sqrt{N_0}(\widehat{\Theta}^*-\widehat{\Theta})=\frac{1}{\sqrt{N_0}}\sum^{N_0}_{i=1}\left\{IF(v_i^*,F_{d_0})-\frac{1}{N_0}\sum^{N_0}_{i=1}IF(v_i,F_{d_0})\right\}+\sqrt{N_0}(\mathcal{R}_{N_0}^*-\mathcal{R}_{N_0}),$$
which implies that
$$J\equiv\left|\max_{1\leq j\leq q_0}\sqrt{N_0}|\widehat{\theta}^*_j-\widehat{\theta}_j|-\max_{1\leq j\leq 2q_0}\frac{1}{\sqrt{N_0}}\sum^{N_0}_{i=1}(x_{ij}^*-\bar{x}_j)\right|
\leq \max_{1\leq j\leq
q_0}\sqrt{N_0}|\mathcal{R}_{jN_0}^*-\mathcal{R}_{jN_0}|,$$ where
$\bar{x}_j=\sum^{N_0}_{i=1}x_{ij}/N_0.$

Denote by
$\widetilde{c}_{2}(\alpha)$ the $(1-\alpha)$ quantile of the
distribution of $\max_{1\leq j\leq
2q_0}\{\sum^{N_0}_{i=1}(x_{ij}^*-\bar{x}_j)/\sqrt{N_0}\}$
conditional on the sample $\{u_{i}\}$. Let
$\zeta_1=Cn^{-c}/\sqrt{\log(2q_0)}$ and $\zeta_2=Cn^{-c}$ for
$C>C_i$ and $c<c_i$ with $i=1,2,3,4$. Assumption \ref{assum:app-2}
and Lemma 3.3 of \cite{cck13} imply that
$$P\left(P\left(J>\zeta_1|\{u_i\}^{n}_{i=1}\right)>\zeta_2\right)\leq \zeta_2,$$
and thus
\begin{align*}
&P(\widetilde{c}_{2}(\alpha)\leq c_{2}(\alpha+\zeta_2)+\zeta_1)\geq 1-\zeta_2, \\
&P(c_{2}(\alpha)\leq \widetilde{c}_{2}(\alpha+\zeta_2)+\zeta_1)\geq
1-\zeta_2.
\end{align*}
Then on the event $\{c_{2}(\alpha)\leq \widetilde{c}_{2}(\alpha+\zeta_2)+\zeta_1\}
\cup\{\widetilde{c}_{2}(\alpha-\zeta_2)\leq c_{2}(\alpha)+\zeta_1\}\cup\{\max_{1\leq j\leq q_0}\sqrt{N_0}|\mathcal{R}_{jN_0}|\leq \zeta_1\}$,
we have
\begin{align*}
&\left|P\left(\max_{1\leq j\leq
2q_0}\sum^{N_0}_{i=1}x_{ij}/\sqrt{N_0}\leq
\widetilde{c}_{2}(\alpha)\right)-P\left(\max_{1\leq j\leq
q_0}\sqrt{N_0}|\widehat{\theta}_j-\widetilde{\theta}_j|\leq
c_2(\alpha )\right)\right|
\\ \leq & P\left(\widetilde{c}_2(\alpha-\zeta_2)-2\zeta_1\leq \max_{1\leq j\leq
2q_0}\sum^{N_0}_{i=1}x_{ij}/\sqrt{N_0}\leq
\widetilde{c}_2(\alpha)\right)
\\ &+ P\left(\widetilde{c}_2(\alpha)\leq \max_{1\leq j\leq
2q_0}\sum^{N_0}_{i=1}x_{ij}/\sqrt{N_0}\leq
\widetilde{c}_{2}(\alpha+\zeta_2)+2\zeta_1\right).
\end{align*}
The conclusion follows from similar arguments in the proof of
Theorem \ref{thm:block-boot}.
\end{proof}

\subsection{Technical details for Section~\ref{sec:bandtest}}\label{subsec:band}
To justify the validity of the procedure in
Section~\ref{sec:bandtest}, we impose the following assumptions
which are parallel to those in Assumption \ref{assum:app-1}.
\begin{assumption}\label{assum:band}
Assume that under $H_0$,
$$P\left(\max_{|j-k|\geq \iota}\left|\frac{1}{\sqrt{n}}\sum^{n}_{i=1}\frac{\sqrt{\gamma_{u,jj}(0)\gamma_{u,kk}(0)}-\sqrt{\widehat{\gamma}_{u,jj}(0)\widehat{\gamma}_{u,kk}(0)}}
{\sqrt{\gamma_{u,jj}(0)\gamma_{u,kk}(0)\widehat{\gamma}_{u,jj}(0)\widehat{\gamma}_{u,kk}(0)}}
u_{ij}u_{ik}\right|>C_1n^{-c_1}/\sqrt{\log(p)}\right)<C_1n^{-c_1}$$
and $P(\mathcal{E}_{AB}\{\log(p)\}^2>C_2n^{-c_2})\leq C_2n^{-c_2}$,
where $c_1,C_1,c_2,C_2>0$, and
\begin{align*}
\mathcal{E}_{AB}=\max_{|j-k|\geq
\iota}\left|\frac{1}{n}\sum^{r}_{i=1}\{(\widetilde{A}_{i,jk}-\widehat{A}_{i,jk})^2+(\widetilde{B}_{i,jk}-\widehat{B}_{i,jk})^2\}\right|,
\end{align*}
with
$\widetilde{A}_{i,jk}=\sum^{iN+(i-1)M}_{l=iN+(i-1)M-N+1}\widetilde{u}_{lj}\widetilde{u}_{lk}$
and
$\widetilde{B}_{i,jk}=\sum^{i(N+M)}_{l=i(N+M)-M+1}\widetilde{u}_{lj}\widetilde{u}_{lk}$.
\end{assumption}

Below we provide some primitive conditions under which Assumption
\ref{assum:band} holds. To this end, we consider a $M$-dependent
stationary sequence $\{x_i\}$, where $M$ is allowed to grow with the
sample size.
\begin{lemma}\label{lemma:band}
Assumption \ref{assum:band} holds under the following conditions,
\begin{align*}
&c_0<\min_j\gamma_{u,jj}(0)\leq \max_j\gamma_{u,jj}(0)<C_0,\quad c_0,C_0>0,\\
&\max_{1\leq j,k\leq p}\{\E (A_{1,jk}/\sqrt{N})^2\}^{1/2}\vee \{\E
(B_{1,jk}/\sqrt{M})^2\}^{1/2}\lesssim n^{s_1},\\
&\max_{1\leq j,k\leq p}||A_{1,jk}/\sqrt{N}||_h\vee
||B_{1,jk}/\sqrt{M}||_h\lesssim n^{s_2}, \\
&n^{s_1}\sqrt{\log(p)/(rM)}+n^{s_2}\log(rp)\log(p)/(r\sqrt{M})\lesssim n^{-c},\\
&Nn^{-c}(\log p)^2\lesssim n^{-c'},\quad \frac{\sqrt{n\log
p}}{n^{3c/2}}\lesssim n^{-2c''},
\\&n^{-c}\sqrt{N}(\log p)^2\log(rp)n^{s_1}\lesssim n^{-c'''}.
\end{align*}
\end{lemma}

\begin{proof}[Proof of Lemma \ref{lemma:band}]
Define the block sums
$A_{i,jk}=\sum^{iN+(i-1)M}_{l=iN+(i-1)M-N+1}u_{lj}u_{lk}$ and
$B_{i,jk}=\sum^{i(N+M)}_{l=i(N+M)-M+1}u_{lj}u_{lk}$. Note that
\begin{align*}
& P(\max_{1\leq j,k\leq
p}|\gamma_{u,jk}(0)-\widehat{\gamma}_{u,jk}(0)|>\jmath)\leq
\E\max_{1\leq j,k\leq
p}|\gamma_{u,jk}(0)-\widehat{\gamma}_{u,jk}(0)|/\jmath
\\ \leq& \frac{1}{\jmath}\E\max_{1\leq j,k\leq p}\left|\sum^{r}_{i=1}(A_{i,jk}-\E A_{i,jk})/(Nr)\right|+\frac{1}{\jmath}\E\max_{1\leq j,k\leq p}\left|\sum^{r}_{i=1}(B_{i,jk}-\E
B_{i,jk})/(Mr)\right|.
\end{align*}
By Lemma A.1 in \cite{cck13} and the assumptions,
\begin{align*}
&\E\max_{1\leq j,k\leq p}\left|\sum^{r}_{i=1}(A_{i,jk}-\E
A_{i,jk})/(Nr)\right|
\\ \lesssim& \max_{1\leq j,k\leq p}\{\E (A_{1,jk}/N)^2\}^{1/2}\sqrt{\log(p)/r}+
\sqrt{\E\max_{1\leq i\leq r}\max_{1\leq j,k\leq
p}|A_{i,jk}/N|^2}\log(p)/r
\\ \lesssim& \max_{1\leq j,k\leq p}\{\E (A_{1,jk}/\sqrt{N})^2\}^{1/2}\sqrt{\log(p)/(rN)}+
\max_{1\leq j,k\leq
p}||A_{1,jk}/\sqrt{N}||_h\log(rp)\log(p)/(r\sqrt{N})
\\ \lesssim&
n^{s_1}\sqrt{\log(p)/(rN)}+n^{s_2}\log(rp)\log(p)/(r\sqrt{N})\lesssim n^{-c}.
\end{align*}
With $\jmath=n^{-c/2}$, we have
$$P\left(\max_{1\leq j,k\leq p}|\gamma_{u,jk}(0)-\widehat{\gamma}_{u,jk}(0)|>n^{-c/2}\right)\lesssim n^{-c/2}.$$

On the event $\max_{1\leq j,k\leq
p}|\gamma_{u,jk}(0)-\widehat{\gamma}_{u,jk}(0)|\leq n^{-c/2}$, we
have $c_0/2\leq \widehat{\gamma}_{u,jj}(0)\leq 2C_0$ uniformly for
$1\leq j\leq p$ and some $c_0,C_0>0.$ Hence we get
\begin{align*}
&\frac{\sqrt{\gamma_{u,jj}(0)\gamma_{u,kk}(0)}-\sqrt{\widehat{\gamma}_{u,jj}(0)\widehat{\gamma}_{u,kk}(0)}}
{\sqrt{\gamma_{u,jj}(0)\gamma_{u,kk}(0)\widehat{\gamma}_{u,jj}(0)\widehat{\gamma}_{u,kk}(0)}}\lesssim
\sqrt{\gamma_{u,jj}(0)\gamma_{u,kk}(0)}-\sqrt{\widehat{\gamma}_{u,jj}(0)\widehat{\gamma}_{u,kk}(0)}
\\ \lesssim& \max_{1\leq j\leq
p}\frac{|\gamma_{u,jj}(0)-\widehat{\gamma}_{u,jj}(0)|}{|\sqrt{\gamma_{u,jj}(0)}+\sqrt{\widehat{\gamma}_{u,jj}(0)}|}
\lesssim \max_{1\leq j\leq
p}|\gamma_{u,jj}(0)-\widehat{\gamma}_{u,jj}(0)|.
\end{align*}
On the other hand, using similar arguments above, we have
\begin{align*}
&I=P\left(\max_{|j-k|\geq
\iota}\left|\frac{1}{\sqrt{n}}\sum^{n}_{i=1}\frac{\sqrt{\gamma_{u,jj}(0)\gamma_{u,kk}(0)}-\sqrt{\widehat{\gamma}_{u,jj}(0)\widehat{\gamma}_{u,kk}(0)}}
{\sqrt{\gamma_{u,jj}(0)\gamma_{u,kk}(0)\widehat{\gamma}_{u,jj}(0)\widehat{\gamma}_{u,kk}(0)}}
u_{ij}u_{ik}\right|>C_1n^{-c_1}/\sqrt{\log(p)}\right)
\\ \lesssim &P\left(\max_{|j-k|\geq \iota}\left|\frac{1}{\sqrt{n}}\sum^{n}_{i=1}
u_{ij}u_{ik}\right|>\jmath'\right)+n^{-c/2}\leq
\frac{1}{\jmath'}\E\max_{|j-k|\geq
\iota}\left|\frac{1}{\sqrt{n}}\sum^{n}_{i=1}
u_{ij}u_{ik}\right|+n^{-c/2}
\\ \leq& \frac{\sqrt{n}}{\jmath'}\E\max_{|j-k|\geq
\iota}\left|\sum^{r}_{i=1}A_{i,jk}/(Nr)\right|+\frac{\sqrt{n}}{\jmath'}\E\max_{|j-k|\geq
\iota}\left|\sum^{r}_{i=1}B_{i,jk}/(Mr)\right|+n^{-c/2},
\end{align*}
where $\jmath'=Cn^{c/2-c_1}/\sqrt{\log p}.$ 

Again by Lemma A.1 in
\cite{cck13},
\begin{align*}
\E\max_{|j-k|>\iota}\left|\sum^{r}_{i=1}A_{i,jk}/(Nr)\right|
\lesssim& \max_{|j-k|\geq \iota}\{\E
(A_{1,jk}/N)^2\}^{1/2}\sqrt{\log(p)/r}
\\&+\sqrt{\E\max_{1\leq i\leq
r}\max_{|j-k|\geq \iota}|A_{i,jk}/N|^2}\log(p)/r
\\ \lesssim&
n^{s_1}\sqrt{\log(p)/(rN)}+n^{s_2}\log(rp)\log(p)/(r\sqrt{N})\lesssim n^{-c},
\end{align*}
which implies that
\begin{align*}
I\lesssim P\left(\max_{|j-k|\geq
\iota}\left|\frac{1}{\sqrt{n}}\sum^{n}_{i=1}
u_{ij}u_{ik}\right|>\jmath'\right)+n^{-c/2}\lesssim
\frac{\sqrt{n\log p}}{n^{3c/2-c_1}}+n^{-c/2}\lesssim n^{-c_1},
\end{align*}
for properly chosen $c_1.$ Next we show that
$P(\mathcal{E}_{AB}\{\log(p)\}^2>C_2n^{-c_2})\leq C_2n^{-c_2}$. 

Let
$$\mathcal{E}_A=\max_{|j-k|\geq
\iota}\left|\frac{1}{n}\sum^{r}_{i=1}(\widetilde{A}_{i,jk}-\widehat{A}_{i,jk})^2\right|,
\quad \mathcal{E}_B=\max_{|j-k|\geq
\iota}\left|\frac{1}{n}\sum^{r}_{i=1}(\widetilde{B}_{i,jk}-\widehat{B}_{i,jk})^2\right|.$$
We shall show that
$$P(\mathcal{E}_{A}\{\log(p)\}^2>C_2n^{-c_2})\leq
C_2n^{-c_2}.$$ Similar arguments apply to $\mathcal{E}_B.$ Note that
\begin{align*}
P\left(\max_{1\leq i\leq
r}\max_{|j-k|>\iota}|A_{i,jk}/\sqrt{N}|>a\right)\leq & \frac{1}{a}\E
\max_{1\leq i\leq r}\max_{|j-k|>\iota}|A_{i,jk}/\sqrt{N}|\leq
\frac{n^{s_1}\log(rp)}{a},
\end{align*}
where the value of $a$ will be determined later. On the events
$\max_{1\leq i\leq r}\max_{|j-k|>\iota}|A_{i,jk}/\sqrt{N}|<a$ and
$\max_{1\leq j,k\leq
p}|\gamma_{u,jk}(0)-\widehat{\gamma}_{u,jk}(0)|\leq n^{-c/2}$, we
have
\begin{align*}
\frac{1}{n}\sum^{r}_{i=1}\frac{A_{i,jk}^2}{\widehat{\gamma}_{u,jj}(0)\widehat{\gamma}_{u,kk}(0)}
 \leq & \frac{1}{n}\sum^{r}_{i=1}\frac{\sqrt{N}a|A_{i,jk}|}{\widehat{\gamma}_{u,jj}(0)\widehat{\gamma}_{u,kk}(0)}
\\ \leq&  \frac{1}{n}\sum^{n}_{i=1}\frac{\sqrt{N}a|u_{i,j}u_{i,k}|}{\widehat{\gamma}_{u,jj}(0)\widehat{\gamma}_{u,kk}(0)}\leq
 \frac{\sqrt{N}a}{\sqrt{\widehat{\gamma}_{u,jj}(0)\widehat{\gamma}_{u,kk}(0)}}\lesssim \sqrt{N}a.
\end{align*}
We also note that
\begin{align*}
\mathcal{E}_{A}\lesssim &
\max_{|j-k|>\iota}\frac{1}{n}\sum^{r}_{i=1}\left\{\frac{A_{i,jk}(\sqrt{\gamma_{u,jj}(0)\gamma_{u,kk}(0)}-\sqrt{\widehat{\gamma}_{u,jj}(0)\widehat{\gamma}_{u,kk}(0)}
)}{\sqrt{\gamma_{u,jj}(0)\gamma_{u,kk}(0)\widehat{\gamma}_{u,jj}(0)\widehat{\gamma}_{u,kk}(0)}}\right\}^2+N^2r\max_{|j-k|>\iota}|\widehat{\gamma}_{jk}|^2/n
\\ \lesssim & n^{-c} \max_{|j-k|>\iota}\frac{1}{n}\sum^{r}_{i=1}\frac{A_{i,jk}^2}{\widehat{\gamma}_{u,jj}(0)\widehat{\gamma}_{u,kk}(0)}+Nn^{-c} \lesssim n^{-c}\sqrt{N}a+Nn^{-c}.
\end{align*}
The conclusion therefore follows provided that $a=n^{s_1}\log(rp)n^{c'''/2},$
$Nn^{-c}(\log p)^2\lesssim n^{-c'}$ and $n^{-c}\sqrt{N}(\log
p)^2\log(rp)n^{s_1}\lesssim n^{-c'''}$ for some $c',c'''>0.$
\end{proof}

\subsection{General functions on vector sum}\label{subsec:extension}
In this section, we extend the results in Section
\ref{subsec:dep-graph} to general smooth functions
$\mathcal{L}:\mathbb{R}^p\rightarrow \mathbb{R}$ on the
high-dimensional vector sum. We impose the following
assumption~\ref{assum:general-g} regarding the smoothness of
$\mathcal{L}$. Write $\partial_{j} \mathcal{L}(x)=\partial
\mathcal{L}(x)/\partial x_j$, $\partial_{jk}
\mathcal{L}(x)=\partial^2 \mathcal{L}(x)/\partial x_j\partial x_k$
and $\partial_{jkl} \mathcal{L}(x)=\partial^3
\mathcal{L}(x)/\partial x_j\partial x_k\partial x_l$ for
$j,k,l=1,2,\dots,p$, where $x=(x_1,x_2,\dots,x_p)'$.
\begin{assumption}\label{assum:general-g}
Suppose that
\begin{equation}
\sum^{p}_{j=1}|\partial_{j} \mathcal{L}(x)|\lesssim L_1(p),\quad
\sum^{p}_{j,k=1}|\partial_{jk} \mathcal{L}(x)|\lesssim L_2(p),\quad
\sum^{p}_{j,k,l=1}|\partial_{jkl} \mathcal{L}(x)|\lesssim L_3(p),
\end{equation}
where the constants $L_1(p)$, $L_2(p)$ and $L_3(p)$ do not depend on
$x$. Further assume that for any
$\omega=(\omega_1,\dots,\omega_p)'\in\mathbb{R}^p$ with $\max_{1\leq
j\leq p}|\omega_j|\in \mathcal{B}_p$ for some set
$\mathcal{B}_p\subset \mathbb{R}$,
\begin{align*}
&\partial_{j}\mathcal{L}(x)\lesssim
\partial_{j}\mathcal{L}(x+\omega) \lesssim
\partial_{j}\mathcal{L}(x),
\\&\partial_{jk}\mathcal{L}(x)\lesssim \partial_{jk}\mathcal{L}(x+\omega) \lesssim \partial_{jk}\mathcal{L}(x),
\\&\partial_{jkl}\mathcal{L}(x)\lesssim \partial_{jkl}\mathcal{L}(x+\omega) \lesssim \partial_{jkl}\mathcal{L}(x),
\end{align*}
where $1\leq j,k,l\leq p.$ Here, ``$\lesssim$" means $\leq$ up to a
universal constant.
\end{assumption}

\begin{example}\label{example:general-g}
Consider
$\mathcal{L}_{\lambda}(x)=\sum^{p}_{j=1}g_{j,\lambda}(x_j)/p$, where
$x=(x_1,\dots,x_p)'$ and $\lambda$ is a thresholding parameter. Here
we assume that $g_{j,\lambda}(x)=0$ for $|x|<\lambda$ and
$g_{j,\lambda}$ satisfies that $\sum^{p}_{j=1}|\partial
g_{j,\lambda}(x)/\partial x|/p\leq C$ for some constant $C>0.$ It is
straightforward to verify that
$\sum^{p}_{j=1}|\partial_j\mathcal{L}_{\lambda}(x)|/p\leq C$,
$\partial_{j}\partial_k\mathcal{L}_{\lambda}(x)=0$ and
$\partial_{j}\partial_k\partial_l \mathcal{L}_{\lambda}(x)=0$ for
$1\leq j,k,l\leq p.$ Note that with proper choice of
$g_{j,\lambda}$, $\mathcal{L}_\lambda(x)$ provides a smooth
approximation to the function
$\sum^{p}_{j=1}|x_j|\mathbf{1}\{|x_j|>\lambda\}$ which serves as a
building block for the higher criticism test in \cite{zcx2013}.
\end{example}

Assumption \ref{assum:general-g} generalizes the results in Lemmas
A.5 and A.6 of \cite{cck13}. Consider the dependency graph in
Section \ref{subsec:dep-graph}. Parallel to Proposition \ref{prop1},
we have the following result. With slightly abuse of notation, set
$m=g\circ \mathcal{L}$ with $g\in C_b^{3}(\mathbb{R})$.
\begin{proposition}\label{prop3}
Assume that $2\sqrt{5} D_n^2M_{xy}/\sqrt{n} \in \mathcal{B}_p$ with
$M_{xy}=\max\{M_x,M_y\}$. Then under Assumption
\ref{assum:general-g}, we have for any $\Delta>0,$
\begin{equation}\label{eq:prop3}
\begin{split}
&|\E [m(X)-m(Y)]|\lesssim
\{G_2L^2_1(p)+G_1L_2(p)\}\phi(M_x,M_y)\\&+\{G_3L_1^3(p)+3G_2L_1(p)L_2(p)+G_1L_3(p)\}\frac{D_n^{2}}{\sqrt{n}}(\bar{m}_{x,3}^3+\bar{m}_{y,3}^3)
\\&+\{G_3L_1^3(p)+3G_2L_1(p)L_2(p)+G_1L_3(p)\}\frac{D_n^{3}}{\sqrt{n}}(m_{x,3}^3+m_{y,3}^3)+G_1\Delta+G_0\E[1-\mathcal{I}],
\end{split}
\end{equation}
where $G_k=\sup_{z\in\mathbb{R}}|\partial^k g(z)/\partial z^k|$ for
$k\geq 0$. In addition, if $2\sqrt{5}D_n^3M_{xy}/\sqrt{n} \in
\mathcal{B}_p$, we can replace $m_{x,3}^3+m_{y,3}^3$ by
$\bar{m}_{x,3}^3+\bar{m}_{y,3}^3$ in the above upper bound.
\end{proposition}
With the aid of Assumption \ref{assum:general-g}, Proposition
\ref{prop3} follows from similar arguments in the proof of
Proposition \ref{prop1} (the technical details are omitted to
conserve space). When specialized to stationary $M$-dependent time
series, we have the following result.
\begin{theorem}\label{thm-general-g}
Suppose $2\sqrt{5}(6M+1)M_{xy}/\sqrt{n}\in\mathcal{B}_p$ with
$M_{xy}=\max\{M_x,M_y\}$, and $M_x>u_x(\gamma)$ and
$M_y>u_y(\gamma)$ for some $\gamma\in (0,1)$. Then
\begin{equation}\label{eq:general-g-m-dep}
\begin{split}
|\E [m(X)-m(Y)]|\lesssim & \{G_2L^2_1(p)+G_1L_2(p)\}\phi(M_{x},M_y)
\\&+\{G_3L_1^3(p)+3G_2L_1(p)L_2(p)+G_1L_3(p)\}\frac{(2M+1)^{2}}{\sqrt{n}}(\bar{m}_{x,3}^3+\bar{m}_{y,3}^3)
\\&+G_1\varphi(M_{x},M_y)\sigma_j\sqrt{8\log(p/\gamma)}+G_0\gamma.
\end{split}
\end{equation}
Under Condition (\ref{eq:m-dep-couple}), we may set
$\phi(M_{x},M_y)=C(1/M_x+1/M_y)$ and
$\varphi(M_{x},M_y)=C'(1/M_x^{5/6}+1/M_y^{5/6})$ for some constants
$C,C'>0$ in (\ref{eq:general-g-m-dep}).
\end{theorem}

\begin{remark}\label{os:rem}
Consider
$\mathcal{L}_{\lambda}(x)=\sum^{p}_{j=1}g_{j,\lambda}(x_j)/p$ in
Example \ref{example:general-g}. When $M_x>u_x(\gamma)$ and
$M_y>u_y(\gamma)$ for some $\gamma\in (0,1),$ we have
\begin{equation}
\begin{split}
|\E [m_{\lambda}(X)-m_{\lambda}(Y)]|\lesssim &
G_2\phi(M_{x},M_y)+G_3\frac{(2M+1)^{2}}{\sqrt{n}}(\bar{m}_{x,3}^3+\bar{m}_{y,3}^3)
\\&+G_1\varphi(M_{x},M_y)\sigma_j\sqrt{8\log(p/\gamma)}+G_0\gamma,
\end{split}
\end{equation}
where $m_{\lambda}=g\circ \mathcal{L}_{\lambda}$. Under Condition
(\ref{eq:m-dep-couple}),
\begin{equation}
\begin{split}
|\E [m_{\lambda}(X)-m_{\lambda}(Y)]|\lesssim &
G_2(1/M_x+1/M_y)+G_3\frac{(2M+1)^{2}}{\sqrt{n}}(\bar{m}_{x,3}^3+\bar{m}_{y,3}^3)
\\&+G_1(1/M_x^{5/6}+1/M_y^{5/6})\sigma_j\sqrt{8\log(p/\gamma)}+G_0\gamma.
\end{split}
\end{equation}
By letting $M_x\rightarrow +\infty$, $M_y\rightarrow +\infty$, and
$\gamma=M^2/\sqrt{n}$, we deduce that $|\E
[m_{\lambda}(X)-m_{\lambda}(Y)]|\lesssim
M^2(\bar{m}_{x,3}^3+\bar{m}_{y,3}^3)/\sqrt{n}$. Note that in this
case, $p$ is allowed to grow arbitrarily.
\end{remark}

\end{document}